\documentclass[11pt,reqno]{amsart}
\usepackage{lmodern} 
\usepackage{amsmath,amssymb,amsfonts,amsthm}
\usepackage{graphicx,multirow}
\usepackage{tikz}
\usetikzlibrary{calc, positioning, arrows.meta, chains}
\usepackage{booktabs} 
\usepackage{float}
\usepackage [numbers,sort&compress]{natbib}
\usepackage{caption}
\usepackage{array}      
\usepackage{cellspace}  
\usepackage{mathrsfs}   
\usepackage[all]{xy}    
\usepackage{xcolor} 
\usepackage{cancel} 
\usepackage{url}        
\usepackage{longtable}
\usepackage{arydshln} 
\usepackage{enumerate}
\usepackage{hyperref}

\def\Der{{\rm \textup{Der}}}
\def\Inn{{\rm \textup{Inn}}}

\def\ad{{\rm \textup{ad}}}

\def\Aut{{\rm \textup{Aut}}}

\theoremstyle{plain}
\newtheorem{theorem}{Theorem}[section]
\newtheorem{lemma}[theorem]{Lemma}
\newtheorem{proposition}[theorem]{Proposition}

\theoremstyle{definition}

\newtheorem{example}[theorem]{Example}
\newtheorem{remark}[theorem]{Remark}
\newtheorem{notation}[theorem]{Notation}
\newtheorem{construction}[theorem]{Construction}
\newtheorem{question}[theorem]{Question}

\begin{document}

\title[]{Lie algebras with only inner derivations\\ and
applications to finite $p$-groups}

\author[Yunsong Gan]{Yunsong Gan}
\address{Yunsong Gan \newline \indent Department of Mathematics, Southern University of Science and Technology,
Shenzhen, Guangdong, 518055, P.R. China}
\email{{\tt songsirr@126.com}}

\author[Bakhrom Omirov]{Bakhrom Omirov}
\address{Bakhrom A. Omirov \newline \indent
Institute for Advanced Study in Mathematics,
Harbin Institute of Technology, Harbin 150001 \newline \indent
Suzhou Research Institute, Harbin Institute of Technology, Harbin  215104, Suzhou, China}
\email{{\tt omirovb@mail.ru}}

\author[Jie Ruan]{Jie Ruan}
\address{Jie Ruan \newline \indent School of Mathematics and Statistics,
Northeast Normal University, Changchun,  130024, China.}
\email{{\tt ruanriss@126.com}}


\begin{abstract}
This paper begins by studying Lie algebras of the form $\mathcal L=\mathcal S\ltimes\mathcal N$ with only inner derivations. We extend Sato's construction to infinite families of such Lie algebras and answer Carles's question by exhibiting a $19$-dimensional complete non-perfect Lie
algebra with nilpotent radical. 
Applying our construction, we derive a family of finite $p$-groups $P$ with $|\Aut(G)|=|P|$ for any $p\geq7$. Moreover, we show that for arbitrarily small $\alpha>0$ and every prime $p\geq7$, there exists a finite $p$-group $P$ such that $|\Aut(P)|<\varphi(|P|)^\alpha$. This is a preliminary version of the manuscript, and some sections are still being revised.
\end{abstract}
\thanks{\textit{MSC2020:}Primary 17B40; Secondary 17B30, 17B56, 20D15, 20D45.}
\thanks{\textit{Key words:} Lie algebra; nilpotent Lie algebra; derivation; complete Lie algebra; Hochschild–Serre factorization; finite \(p\)-groups; automorphism groups.}

\maketitle


\renewcommand{\arraystretch}{1.2} 
\section{Introduction}

The systematic study of derivations in abstract algebra with investigation of algebraic operators capturing the formal properties of differentiation was initiated in \cite{Jac-37}. It is well known that the exceptional simple Lie algebras $F_4$ and $G_2$ arise as the derivation algebras of the Cayley algebra and the exceptional Jordan algebra, respectively \cite{Che-50,Jac-39}. More generally, the structure of all exceptional Lie algebras is closely related to derivations of suitable non-associative algebras. Moreover, derivations admit a natural interpretation as infinitesimal automorphisms \cite{Hoc-42}. More precisely, a derivation $d$ can be viewed as the tangent vector at the identity of a one-parameter group of automorphisms $(h_t)$, related by the exponential map $h_t=\exp(td)$. The condition that $h_t$ preserves the algebraic product, $h_t(x\cdot y)=h_t(x)\cdot h_t(y)$, is equivalent, at first order in $t$, to the Leibniz rule
\[
d(x\cdot y)=d(x)\cdot y + x\cdot d(y).
\]
Inner derivations then form precisely the Lie algebra corresponding to the normal subgroup of inner automorphisms. These connections have significantly stimulated the study of derivations in Lie theory.

The structure of the derivation algebra $\Der(\mathcal{L})$ reflects deep properties of the Lie algebra $\mathcal{L}$. It is a classical result that a finite-dimensional Lie algebra over a field of characteristic zero is simple if and only if its derivation algebra is simple \cite{jaco62}. This result has recently been extended to arbitrary dimension over general fields \cite{FeSi-26}. Furthermore, it is proved that the existence of a non-singular derivation implies nilpotency \cite{Jac-55}, while the converse fails \cite{Dix-57}. This observation led to the introduction of characteristically nilpotent Lie algebras, that is, nilpotent Lie algebras whose derivations are all nilpotent.

Derivations of the nilradical play a fundamental role in the structure theory of solvable Lie algebras. In particular, any solvable (non-nilpotent) Lie algebra can be reconstructed from its nilradical (which are never characteristically nilpotent) together with special types of derivations \cite{Mub-63}.

A classical problem in Lie theory is the characterization of Lie algebras all of whose derivations are inner. While every semisimple Lie algebra has this property and every nilpotent Lie algebra admits a non-trivial outer derivation \cite{Jac-55}, the solvable case is considerably more delicate. Recent progress in this direction has been made in \cite{Omi-26}. By Levi's decomposition, this problem naturally reduces to studying extensions of solvable radicals by semisimple Lie algebras, where representation-theoretic methods play a crucial role. In particular, since the representation theory of $\mathfrak{sl}_2$ is completely understood, it is natural to focus on the case where the Levi factor is $\mathfrak{sl}_2$. Nevertheless, determining whether $\Der(\mathcal{L})=\Inn(\mathcal{L})$ for a general Lie algebra $\mathcal{L}$ remains a difficult problem.

In this context, G. Leger has posed the question of whether there exists a Lie algebra with only inner derivations and non-trivial center \cite{Leg-63}. This question was answered affirmatively by E. Luks, who constructed a $58$-dimensional example \cite{Luk-70}. Subsequently, T. Sato exhibited a Lie algebra of the form $\mathcal{L}=\mathfrak{sl}_2 \ltimes \mathcal{N}$ satisfying $\Der(\mathcal{L})=\Inn(\mathcal{L})$ \cite{Sato-71}. Such examples are of particular interest, as they provide non-semisimple Lie algebras with vanishing first cohomology. 

Motivated by properties of semisimple Lie algebras, several constructions have been proposed. In particular, Angelopoulos  introduced the notion of \emph{sympathetic} Lie algebras, characterized by properties analogous to semisimplicity, such as trivial center, completeness, and vanishing first cohomology, and provided explicit constructions \cite{An-1988}. Later,  Benayadi refined this approach using  isomorphic copies of irreducible $\mathfrak{sl}_2$-modules, obtaining a $25$-dimensional example \cite{bena96}. Their constructions relies on identifying $\mathfrak{sl}_2$ and the nilradical with polynomial algebras of the form $\mathbb{C}[u,w]$, allowing explicit computation of multiplication tables \cite{ABBBP-1992}.

Lie algebras of the form $\mathcal{S}\ltimes \mathcal{N}$ with only inner derivations are relatively rare, and their classification remains largely open. In particular, a question raised by R. Carles (see \cite{Car-84}) concerning the existence of complete non-perfect Lie algebras highlights that the structure theory of such algebras is still far from being fully understood. 

Recall, the second adjoint cohomology group governs infinitesimal deformations of an algebra and provides a criterion for its rigidity \cite{Campoamor-Stursberg, Gru-88, GO-93}. In particular, it parametrizes infinitesimal deformations, and its vanishing implies that all formal deformations are trivial up to equivalence. From a geometric viewpoint, $n$-dimensional Lie algebra structures form an affine algebraic variety $\mathcal{L}_n$ endowed with a natural $\mathrm{GL}_n$-action, where orbits correspond to isomorphism classes and degenerations correspond to Zariski closures of orbits. The relationship between deformations and degenerations has been clarified in \cite{Lauret, Weimar}.

Thus, alongside the study of the first cohomology group, it is natural to investigate the second cohomology group. Regarding this, the adjoint second cohomology of Benayadi's example was computed in \cite{Burde24}, where it was shown to be one-dimensional.

The main objective of this paper is to  construct new families of Lie algebras of the form $\mathfrak{sl}_2 \ltimes \mathcal{N}$ that admit only inner derivations. Generalizing Sato’s construction, we introduce a class of generalized \emph{quasi-cyclic} nilpotent Lie algebras, that is, Lie algebras admitting a decomposition $\mathcal{N}=\bigoplus \mathcal{U}^i$ with $\mathcal{U}^i=[\mathcal{U}^i,\mathcal{U}^1]$ and $\mathcal{U}^1$ generating $\mathcal{N}$ (see \cite{Leg-63}). We combine this class with $\mathfrak{sl}_m$-actions for arbitrary $m\geq 2$. Furthermore, we apply the Hochschild--Serre factorization theorem to describe their second adjoint cohomology. We also adapt constructions of sympathetic Lie algebras to obtain new families with only inner derivations, thereby providing a positive answer to Carles's question on the existence of complete non-perfect Lie algebras of the form $\mathcal{S}\ltimes \mathcal{N}$. Finally, we present an additional construction, distinct from that of Sato, yielding perfect Lie algebras with non-trivial center and only inner derivations.

Lie algebras that perfect with non-trivial center and only inner derivations also arise in the study of
automorphism groups of finite $p$-groups.
Gonz\'alez-S\'anchez and Jaikin-Zapirain \cite{GJ2015}
used Sato's $41$-dimensional algebra to construct, for every prime $p$, finite
$p$-groups $G_i$ satisfying
$|\operatorname{Aut}(G_i)|/|G_i|\to0$.
Their result disproved the divisibility conjecture
recorded as Problem~12.77 in the Kourovka Notebook.

A related question posed by MacHale asks whether a
nontrivial finite $p$-group of odd order can satisfy
$|\operatorname{Aut}(G)|=|G|$;
see \cite[Problem~16.63]{Kourovka}.
A candidate example obtained through computer calculations is given in \cite{Kriebel26}.
Using the $20$-dimensional Lie algebra constructed in
Theorem~\ref{thm:classB-example}, we give a theoretical
existence proof for every prime $p\geq7$.

Thanks to Professor Evgeny Khukhro for drawing our attention to manuscript~\cite{Kriebel26} on Problem~16.63 of the
\emph{Kourovka Notebook}~\cite{Kourovka}. \cite{Kriebel26} starts from the
$31$-dimensional Lie algebra in \cite[Theorem~4.8]{OR2026}. This led us to seek a smaller
Lie algebra with the same properties. Using the
$20$-dimensional algebra $\mathcal L_{20}$, we give a theoretical
proof that applies to every prime $p\geq7$.

\begin{theorem}
\label{thm:Kourovka-application}
For every prime $p\geq7$, there exists a nontrivial finite
$p$-group $G_p$ such that
\[
|\operatorname{Aut}(G_p)|=|G_p|.
\]
\end{theorem}

Starting from $\mathcal L_{20}$, we also construct Lie algebras
with larger centers and use a similar method to address a
question about lower bounds for the orders of automorphism
groups. For a finite $p$-group $P$ of order $p^n$ with
$n\geq252$, Burmester and Exarchakos~\cite{BE1988} proved that
\[
|\operatorname{Aut}(P)|
\geq |\operatorname{Aut}(P)|_p
\geq p^{\lfloor\sqrt{7n}\rfloor},
\]
where $|\operatorname{Aut}(P)|_p$ denotes the $p$-part of
$|\operatorname{Aut}(P)|$. Can the lower bound for
$|\operatorname{Aut}(P)|$ be strengthened to $p^{\alpha n}$
for some absolute constant $\alpha>0$?
Gonz\'alez-S\'anchez and Jaikin-Zapirain asked a more general
question, recorded as Problem~20.26(a) of the
\emph{Kourovka Notebook}~\cite{Kourovka}.

\begin{question}[{\cite[Question~4.1]{GJ2015}}]
\label{que1}
Let $\varphi$ denote Euler's totient function. Does there exist
a constant $\alpha>0$ such that
$|\operatorname{Aut}(G)|\geq\varphi(|G|)^\alpha$
for every finite group $G$?
\end{question}

Part~(b) asks the same question for finite nilpotent groups.
The following theorem gives a negative answer to both parts.

\begin{theorem}
\label{thm1}
For every $\alpha>0$ and prime $p\geq7$, there exist a finite
$p$-group $P$ such that
\[
|\Aut(P)|<\varphi(|P|)^\alpha.
\]
\end{theorem}

The paper is organized as follows. In Section~\ref{sec2},
we generalize Sato's construction by first constructing the
quasi-cyclic nilpotent Lie algebras $\mathfrak N(n)$ and then
introducing the basic algebras $\mathcal L_{n,t}$ and the family
$\mathfrak{sl}_2\mathcal G(t_i)$. We show that every algebra in
this family admits only inner derivations
(Theorem~\ref{thm2.9}). We also construct quasi-cyclic
nilradicals with an $\mathfrak{sl}_m$-action for $m\geq3$.

Section~\ref{sec3} is devoted to the second adjoint cohomology
of the generalized Sato-type Lie algebras. We show that
$\dim H^2(\mathcal L_{5,1},\mathcal L_{5,1})=2$
(Proposition~\ref{prop3.3}) and that
$\dim H^2(\mathcal L,\mathcal L)\geq k$ for members of
$\mathfrak{sl}_2\mathcal G(t_i)$ constructed from $k$ radicals
(Theorem~\ref{thm3.4}).

In Section~\ref{sec4}, we generalize Angelopoulos's construction
of sympathetic Lie algebras and prove that the second adjoint
cohomology of the algebras in Example~\ref{exam4.2} is nonzero
when $n\geq5$ (Theorem~\ref{thm4.3}). We then construct two
families whose smallest members are the $19$-dimensional
complete non-perfect Lie algebra $\mathcal L_{19}$ and the
$20$-dimensional perfect Lie algebra $\mathcal L_{20}$ with
nonzero center and only inner derivations
(Theorems~\ref{thm:classA-example} and~\ref{thm:classB-example}).

In Section~\ref{sec5}, we use $\mathcal L_{20}$ to construct,
for every prime $p\geq7$, a nontrivial finite $p$-group
$\mathsf G_p$ satisfying
$|\operatorname{Aut}(\mathsf G_p)|=|\mathsf G_p|$.
This answers Kourovka Notebook Problem~16.63 for these primes
(Theorem~\ref{thm:Kourovka-application}).

In Section~\ref{sec6}, we extend the construction of
$\mathcal L_{20}$ to obtain Lie algebras with larger centers.
Using these algebras, we construct, for every $\alpha>0$,
a finite $p$-group $G$ of odd order with
$|\operatorname{Aut}(G)|<\varphi(|G|)^\alpha$.
This gives a negative answer to Kourovka
Notebook Problem~20.26 (Theorem~\ref{thm1}).

\section{Construction of generalized Sato-type Lie algebras}\label{sec2}

In this section, motivated by Sato’s example, we construct a family of quasi-cyclic nilpotent Lie algebras by introducing inductively on its  natural gradation the Lie brackets. This leads to the algebra $\mathfrak{N}(n)$ for odd $n$, which is then extended via an $\mathfrak{sl}_2$-action to obtain generalized quasi-cyclic Lie algebras $\mathcal{N}_{n,t}$. We further investigate their structure and derivations, showing in particular that the corresponding semidirect products admit only inner derivations. For background on Lie algebras results, as well as on complete Lie algebras, we refer the reader to \cite{Hum-78, jaco62, Meng-94, MengWang-95, MengZhu-96} and the references therein.

\subsection{Sato-type Lie algebras with Levi factor \(\mathfrak{sl}_2\)}
Recall that every nilpotent Lie algebra $\mathcal N$ of nilindex $n$ admits a decomposition adapted to its lower
central series:
$$\mathcal N=\bigoplus_{k=1}^{n}\mathcal U^k,
\qquad \mathcal N^k=\bigoplus_{j=k}^{n}\mathcal U^j, \quad 1\leq k\leq n.$$
Such a decomposition is obtained by choosing a complement $\mathcal U^k$ to $\mathcal N^{k+1}$ in $\mathcal N^k$ for each $k$ and any basis of $\mathcal U^1=\mathcal U$ is a minimal generating set of $\mathcal N$. (see \cite{Santharoubane1982})

Using the decomposition we construct a quasi-cyclic nilpotent Lie algebras. Fix an integer $n\geq 5$ and let $\mathcal{U}^k$ are defined as follows:
\[
\begin{array}{llll}
\mathcal{U} = \mathcal{V}_1 \cup \mathcal{W}_1, 
& \mathcal{U}^2 = \mathcal{V}_2, 
& \mathcal{U}^k = \mathcal{V}_k \cup \mathcal{W}_k, \quad 3 \leq k \leq n,\\[2mm]
\mathcal{V}_1 = \{x_1, x_2\}, 
& \mathcal{V}_2 = \{c, z_1, z_2, z_3\}, 
& \mathcal{V}_k = \{x_1^k, x_2^k, \dots, x_{k+1}^k\}, \\[2mm]
\mathcal{W}_1 = \{y_1, y_2\}, 
& \mathcal{W}_k = \{y_1^k, y_2^k, \dots, y_{k+1}^k\}.
\end{array}
\]

We prescribe the following nonzero products involving $\mathcal U$:
\begin{equation}\label{eq1}
\begin{array}{llllllllll}
\mathcal{U}^2 = [\mathcal{U},\mathcal{U}] &:&  [x_1,x_2] = c, & [x_1,y_1] = z_1, & [x_2,y_1] = z_2, \\[3mm]
& & [y_1,y_2] = c, & [x_1,y_2] = z_2, & [x_2,y_2] = z_3. \\[3mm]
\mathcal{U}^3 = [\mathcal{U},\mathcal{U}^2] &:& [x_1,z_i] = x_i^3,  & [y_1,z_i] = y_{i}^3,\\[3mm]
& &[x_2,z_i] = x_{i+1}^3, & [y_2,z_i] = y_{i+1}^3, & 1\leq i\leq 3, \\[3mm]
\mathcal{U}^k = [\mathcal{U},\mathcal{U}^{k-1}]&:& [x_1,x_i^{k-1}] = x_i^k, & [y_1,y_i^{k-1}] = y_i^k, &1\leq i\leq k,\\[3mm]
& & [x_2,x_i^{k-1}] = x_{i+1}^k, & [y_2,y_i^{k-1}] = y_{i+1}^k, & 4\leq k\leq n-1, \\[3mm]
\mathcal{U}^n = [\mathcal{U},\mathcal{U}^{n-1}] &:& [x_1,x_i^{n-1}] = x_i^n, & [y_1,x_i^{n-1}] = y_i^{n}, \\[3mm]
& & [x_2,x_i^{n-1}] = x_{i+1}^n, & [y_2,x_i^{n-1}] = y_{i+1}^{n}, & 1\leq i\leq n.\\[3mm] 
\end{array}
\end{equation}

For convenience, we introduce the following notation 
$$\mathcal J(x,y,z):=[x,[y,z]]+[y,[z,x]]+[z,[x,y]].$$

\begin{proposition}\label{prop2.1}
Suppose that $n\geq 5$ is odd. Equip $\mathcal N$ with the multiplication determined by
\eqref{eq1} and the additional products
\begin{equation}\label{eq2}
\begin{split}
[z_j, x_i^{n-2}] &= -y_{i+(j-1)}^{n}, \quad 1\leq i\leq n-1, \quad 1\leq j\leq 3,\\
[x_j^p, x_i^{n-p}] &= (-1)^{p-1} y_{i+(j-1)}^{n}, \quad 3\leq p\leq \frac{n-1}{2}, \quad  1\leq i\leq n-p+1,\quad 1\leq j\leq p+1,
\end{split}
\end{equation}
Then $\mathcal N$ is a quasi-cyclic Lie algebra of nilindex $n$.
\end{proposition}
\begin{proof}
We verify the Jacobi identity on triples of basis elements. If all three elements belong to $\bigoplus\limits_{r=2}^n\mathcal U^r$, then each term is zero. For triples containing at least two elements of $\mathcal U$, the terms cancel directly by \eqref{eq1}, except possibly for $\mathcal J(x_s,y_t,x_i^{n-2})$. For $s,t\in\{1,2\}$ and $1\leq i\leq n-1$, \eqref{eq1} and \eqref{eq2} give $\mathcal J(x_s,y_t,x_i^{n-2})=0.$

It remains to consider triples with exactly one element of $\mathcal U$. For this calculation, we set $x_i^2=z_i$ for $1\leq i\leq3$. Since $n$ is odd, \eqref{eq2} and skew-symmetry deduce
\begin{equation}\label{eq3}
[x_i^p,x_j^{n-p}]=(-1)^{p-1}y_{i+j-1}^n,
\end{equation}
for $2\leq p\leq n-2$, $1\leq i\leq p+1$ and $1\leq j\leq n-p+1$.
The only triples that can have nonzero terms are $\{x_s,x_i^p,x_j^q\}$ with $s\in\{1,2\}$, $p,q\geq2$ and $p+q=n-1$. For these triples, \eqref{eq1} and \eqref{eq3} imply 
$\mathcal J(x_s,x_i^p,x_j^q)
=0,$ since $p+q$ is even. Thus $\mathcal N$ is a Lie algebra.

The products \eqref{eq1} together with the grading derive 
$\mathcal N^k=\bigoplus\limits_{j=k}^n\mathcal U^j$ for $1\leq k\leq n$ and $\mathcal N^{n+1}=0$, so $\mathcal N$ is quasi-cyclic of nilindex $n$.
\end{proof}

We shall denote the algebra obtained in Proposition \ref{prop2.1} by $\mathfrak{N}(n)$ and refer it as {\it model quasi-cyclic Lie algebra}.

\begin{remark} Note that when $n$ is even, the products in~\eqref{eq1} cannot be extended to a quasi-cyclic Lie algebra structure on $\mathcal N$. Indeed, writing $n=2k$, the recursive relations force
$$[x_1^{k-1},x_1^{k+1}]=(-1)^{k-2}y_1^{2k},$$
which contradicts the Jacobi identity for the triple $\{x_1,x_1^{k-1},x_1^k\}$.
\end{remark}

We next specify additional brackets on
$$\mathcal N=\mathcal U\oplus\mathcal U^2\oplus\cdots\oplus\mathcal U^n,$$
extending those given by~\eqref{eq1}--\eqref{eq2}.
To form the semidirect product $\mathfrak{sl}_2\ltimes\mathcal N$, for $1\leq i \leq k+1$ and $1\leq j \leq 3$, 
we prescribe the action of $\mathfrak{sl}_2$ on $\mathcal N$ by
\begin{equation}\label{eq4}
\left\{\begin{array}{llllll}
[h, v_i^{k}]=(k+2-2i)v_i^{k}, &[e,v_i^{k}]  = (i-1)v_{i-1}^{k}, & [f ,v_i^{k}] = (k+1-i)v_{i+1}^{k},\\[3mm]
[h, z_j] = (4-2j)z_j, & [e , z_j] = (j-1)z_{j-1}, & [f, z_j] = (3-j)z_{j+1},\\
\end{array}\right.
\end{equation}
where $v_i^k$ denotes either $x_i^k$ or $y_i^k$, with the convention
$x_j^2=z_j$, and $[\mathfrak{sl}_2,c]=0$.

For the subsequent formulas, introduce notation
$$\begin{array}{llllllll}
u(i, j) &= &(i - 1 - (j - 1)n) y_{i + j - 2}^{n-1}, & 1 \leq i \leq n + 1, &  1\leq j \leq 2\\[3mm]
v(p, i, j) &=& (-1)^p \left( (i - 1)p - (j - 1)(n - p + 1) \right) y_{i + j - 2}^{n-1}, & 1 \leq i \leq n - p +2, & 1 \leq j \leq p + 1,
\end{array}
$$
where $2\leq p\leq (n+1)/2$. Throughout, vectors with indices outside their prescribed ranges
are understood to be zero.

\begin{proposition}\label{prop2.3}
Let $n\geq5$ be odd, and let $\mathcal N=\mathcal U\oplus\mathcal U^2\oplus\cdots\oplus\mathcal U^n$
be equipped with the products~\eqref{eq1}--\eqref{eq2} and the
$\mathfrak{sl}_2$-action~\eqref{eq4}. Prescribe the additional brackets
\begin{equation}\label{eq5}
[x_1,y_2^n]=a\,y_1^{n-1},
\qquad
[y_1,x_2^n]=-b\,y_1^{n-1}.
\end{equation}
These products extend to a Lie algebra structure on
$\mathfrak{sl}_2\ltimes\mathcal N$ if and only if $a=b$.
In this case, setting $t=a=b$, the additional brackets are given by
\begin{equation}\label{eq6}
[x_j,y_i^n]=t\,u(i,j),
\qquad
[y_j,x_i^n]=-t\,u(i,j),
\qquad
[x_j^p,x_i^{n-p+1}]=t\,v(p,i,j).
\end{equation}
\end{proposition}
\begin{proof} The Jacobi identity forces $a=b$. Indeed, for $1\leq i\leq n$, one obtains
$$\mathcal{J}(x_1,y_2,x_i^{n-1}) = \frac{n+1}{2}(a-b)y_i^{n-1}.$$
Now set $a=b=t$ and consider the equality $\mathcal J(x_1, y_1, x_2^{n-1})=0$. A direct computation gives
\begin{equation}\label{eq7}
    [z_1,x_2^{n-1}] = 2t\,y_1^{n-1}.
\end{equation}
It follows by induction that
\begin{equation}\label{eq8}
 [x_1^p,x_2^{n-p+1}]=(-1)^p pt\,y_1^{n-1}
 \quad \text{for } 3\leq p\leq \frac{n-1}{2}+1.
\end{equation}
Indeed, from $\mathcal J(x_1,z_1,x_2^{n-2})=0$ and product \eqref{eq7},
we derive
$[x_1^3,x_2^{n-2}]=-3t\,y_1^{n-1}.$
This verifies the base of the induction. Then applying the induction
assumption in the Jacobi identity
$\mathcal J(x_1,x_1^{p-1},x_2^{n-p+1})=0$, we obtain \eqref{eq8}.

Below, we are going to produce the remaining non-zero products by the
action of $\mathfrak{sl}_2$ given in \eqref{eq4} to products
\eqref{eq5}, \eqref{eq7}, and \eqref{eq8}.

The actions of elements $e$ and $h$ to the product $[x_1,y_3^n]$
provide
$[x_1,y_3^n] = 2t\,y_2^{n-1},$
which can be used as a base of induction to obtain
$[x_1,y_i^n]=t(i-1)y_{i-1}^{n-1}$
for $1\leq i\leq n+1$.
These products are derived by once again considering the actions of the
elements $e$ and $h$ on the product $[x_1,y_{i+1}^n]$, together with
the induction hypothesis.

Now, since $\mathcal J(f,x_1,y_i^n)=0$, we obtain
$[x_2,y_i^n]=t(i-1-n)y_i^{n-1}$
for $1\leq i\leq n+1$,
which completes the proof of the first product of \eqref{eq6}.
By a similar approach, starting from the action of $\mathfrak{sl}_2$
on $[y_1,x_2^n]$, one can establish the second equality in
\eqref{eq6}.

Consider the case $p=2$ in \eqref{eq6}. The $\mathfrak{sl}_2$-action
applied to the product $[z_1,x_2^{n-1}]$ yields
\[
[z_1,x_i^{n-1}]
=
2t(i-1)y_{i-1}^{n-1},
\qquad
1\leq i\leq n.
\]
By successive applications of the raising and lowering operators $e$
and $f$, respectively, and using induction, one obtains
\[
[z_j,x_i^{n-1}]
=
t\bigl(2(i-1)-(j-1)(n-1)\bigr)y_{i+j-2}^{n-1},
\quad
1\leq j\leq3,\quad 1\leq i\leq n.
\]

The case
$3\leq p\leq \frac{n-1}{2}+1$
is established through a double induction argument.
Beginning with $j=1$, an induction on $i$ yields the base relations
\[
[x_1^p,x_i^{n-p+1}]
=
(-1)^p pt(i-1)y_{i-1}^{n-1},
\qquad
1\leq i\leq n-p+2.
\]
Subsequent application of the lowering operator $f$ to these products,
combined with the Leibniz identity, produces the inductive step for
$j=2$
\[
[x_2^p,x_i^{n-p+1}]
=
(-1)^p t
\bigl((i-1)p-(n-p+1)\bigr)y_i^{n-1}.
\]

The general case follows by induction on $j$ through repeated application
of the $\mathfrak{sl}_2$-action on
$[x_2^p,x_i^{n-p+1}]$, completing the proof of \eqref{eq6}.

To complete the proof of sufficiency, it remains to verify that the resulting multiplication, defined by~\eqref{eq1}--\eqref{eq2}, \eqref{eq4}, and~\eqref{eq6},
satisfies all remaining Jacobi identities.
\end{proof}

We call the nilpotent Lie algebra $\mathcal N$ constructed in Proposition~\ref{prop2.3} a \emph{generalized quasi-cyclic Lie algebra}
and denote it by $\mathcal N_{n,t}$. The specialization $\mathcal N_{n,0}=\mathfrak N(n)$ is the naturally graded. We also set $\mathcal L_{n,t}:= \mathfrak{sl}_2\ltimes\mathcal N_{n,t}$
and refer to $\mathcal L_{n,t}$ as a \emph{Sato-type Lie algebra}.

To study the innerness of derivations of $\mathcal L_{n,t}$, we first recall the following result from~\cite{Leg-53}.

\begin{theorem}{\rm (\cite{Leg-53})}\label{thm2.4}
The Lie algebra $\mathcal S \ltimes \mathcal R$ has no outer derivation if and only if 
any derivation of $\mathcal R$ which commutable with all the elements of $\operatorname{ad}({\mathcal S})_{| {\mathcal R}}$ is an inner derivation.  
\end{theorem}

\begin{proposition}\label{prop2.5} Any algebra of the family $\mathcal L_{n,t}$ with $t\in\mathbb{C}^{\times}$, admits only inner derivations. 
\end{proposition}
\begin{proof} Let $D$ be a derivation of $\operatorname{Der}(\mathcal{N}_{n,t})$ that commutes with the $\mathfrak{sl}_2$-action. 
By Schur’s lemma, $D$ has the following form:
$$D(x_i) = \alpha x_i + \gamma y_i, \qquad   
D(y_i) = \beta y_i + \delta x_i, \qquad  
1 \leq i \leq 2, \quad \alpha, \beta, \gamma, \delta \in \mathbb{C}.$$

Applying the Leibniz rule for $D$ to the products: 
$$[x_1, x_2], \quad [y_1, y_2], \quad [x_1, y_1], \quad [x_1, z_1], \quad [y_1, x_1^3],$$
we obtain that $\alpha = \beta$ and $\gamma = \delta = 0$. Similarly, from the products $[x_1, y_1], [y_1, z_1], [y_1, y_1^{i}]$, by recursion we deduce $D(y_1^{n-1}) = (n-1)\alpha\, y_1^{\,n-1}.$

On the other hand, from $D([x_1, y_2]) = 2\alpha z_2$ we obtain $D([x_1, z_2]) = 3\alpha x_2^3$.  
Using this as the initial step of a recursive argument for $D([x_1, x_2^{i}])$, we find $D(x_2^{n-1}) = (n-1)\alpha\, x_2^{\,n-1}.$ Since $y_2^n = [y_1, x_2^{n-1}]$, it follows that $D(y_2^n) = n\alpha\, y_2^n.$ Furthermore, noting that $t y_1^{n-1} = [x_1, y_2^n]$, we deduce
$D(y_1^{n-1}) = (n+1)\alpha\, y_1^{\,n-1}.$
Comparing the two expressions for $D(y_1^{n-1})$, we conclude that $\alpha = 0$.  
Hence $D = 0$, and by Theorem~\ref{thm2.4} the proof of the theorem is complete. \end{proof}

We remark that the Lie algebra constructed by Sato (see \cite{Sato-71})
is isomorphic to $\mathcal{L}_{5,\frac{1}{60}}$. The corresponding isomorphism is given on basis elements by
$$\left\{
\begin{array}{llllll}
\varphi(x_7) = x_7 + x_5, &
\varphi(x_{34}) = \tfrac{1}{60}x_{34}, &
\varphi(x_{35}) = \tfrac{1}{30}x_{35}, &
\varphi(x_{36}) = \tfrac{1}{20}x_{36}, \\[3mm]
\varphi(x_{37}) = \tfrac{1}{15}x_{37}, &
\varphi(x_{38}) = \tfrac{1}{12}x_{38}, &
\varphi(x_i) = x_i, & \text{otherwise.}
\end{array}
\right.$$

The next result extends Proposition~\ref{prop2.5} to direct sums of
generalized quasi-cyclic Lie algebras.

\begin{proposition}\label{prop2.7}
Let $n\geq5$ be odd and let $t_1,\ldots,t_s\in\mathbb{C}^{\times}$.
Then the Lie algebra
\[
\mathfrak{sl}_2\ltimes
\bigl(
\mathcal{N}_{n,t_1}\oplus\cdots\oplus\mathcal{N}_{n,t_s}
\bigr)
\]
admits only inner derivations.
\end{proposition}
\begin{proof}
For the sake of convenience, we restrict our discussion to the case
$s=2$, as the general case can be treated by analogous reasoning. Thus, we consider
$\mathfrak{sl}_2\ltimes
\bigl(\mathcal{N}_{n,t_1}\oplus\mathcal{N}_{n,t_2}\bigr)$
with $t_1,t_2\in\mathbb{C}^{\times}$.

Let $D$ be a derivation of
$\mathcal{N}_{n,t_1}\oplus\mathcal{N}_{n,t_2}$ that commutes with the
$\mathfrak{sl}_2$-action. Denote by
$\{x_1,x_2,y_1,y_2\}$ and $\{u_1,u_2,v_1,v_2\}$ the generators of
$\mathcal{N}_{n,t_1}$ and $\mathcal{N}_{n,t_2}$, respectively.
Applying Schur's lemma and the same argument as in the proof of
Proposition~\ref{prop2.5} to each direct summand, we conclude that
\begin{align*}
D(x_i) &= \alpha_1x_i+\alpha_2u_i+\alpha_3v_i,
&
D(u_i) &= \gamma_1x_i+\gamma_2y_i+\gamma_3u_i,\\[3mm]
D(y_i) &= \alpha_1y_i+\beta_1u_i+\beta_2v_i,
&
D(v_i) &= \delta_1x_i+\delta_2y_i+\gamma_3v_i,
\end{align*}
where $1\leq i\leq2$.

From $D([h_i,k_j])=0$ for
$h_i\in\{x_1,y_1\}$ and $k_j\in\{u_1,v_1\}$, we obtain
$D(h)=\lambda h,\;
D(k)=\mu k$
for
$h\in\{x_1,x_2,y_1,y_2\},\;
k\in\{u_1,u_2,v_1,v_2\}.$
The remainder of the proof is analogous to the latter part of
Proposition~\ref{prop2.5}. In the first summand, we obtain
$$D(y_1^{n-1}) =(n-1)\lambda y_1^{n-1} =(n+1)\lambda y_1^{n-1},$$
and hence $\lambda=0$, since $t_1\neq0$. Similarly, for the second
summand,
$$D(v_1^{n-1}) =(n-1)\mu v_1^{n-1} =(n+1)\mu v_1^{n-1}, $$
which gives $\mu=0$, since $t_2\neq0$. Therefore $D=0$, and the result
follows from Theorem~\ref{thm2.4}.
\end{proof}

Let $t_1,t_2\in\mathbb C^\times$, and consider a Lie algebra homomorphism $\phi\colon\mathcal N_{n,t_2}
\rightarrow\operatorname{Der}(\mathcal N_{n,t_1}).$ 
The associated semidirect product
$\mathcal N_{n,t_1}\rtimes_\phi\mathcal N_{n,t_2}$
is the Lie algebra with underlying vector space
$\mathcal N_{n,t_1}\oplus\mathcal N_{n,t_2}$ and bracket
$$[(x_1,y_1),(x_2,y_2)] = \bigl(
[x_1,x_2]+\phi(y_1)(x_2)-\phi(y_2)(x_1),
[y_1,y_2]\bigr),$$
where $x_1,x_2\in\mathcal N_{n,t_1}$ and
$y_1,y_2\in\mathcal N_{n,t_2}$.
For background on semidirect products of Lie algebras,
see~\cite{Bou-75}.

The assignments
\begin{equation}\label{eq9}
\phi(u_i)=\operatorname{ad}(x_i),
\qquad
\phi(v_i)=\operatorname{ad}(y_i),
\qquad 1\leq i\leq2,
\end{equation}
extend to a Lie algebra homomorphism $\phi\colon\mathcal N_{n,t_2}
\rightarrow\operatorname{Der}(\mathcal N_{n,t_1}),$ 
thereby defining the semidirect product
$\mathcal N_{n,t_1}\rtimes_\phi\mathcal N_{n,t_2}$.

\begin{proposition}\label{prop2.8} For arbitrary $t_1,t_2\in\mathbb{C}^{\times}$, the algebra $\mathfrak{sl}_2\ltimes \Big(\mathcal{N}_{n,t_1} * \mathcal{N}_{n,t_2}\Big)$ with $*\in \{\rtimes, \ltimes\}$ admits only inner derivations.
\end{proposition}
\begin{proof} By symmetry, it suffices to prove the assertion for $*=\rtimes$;
the case $*=\ltimes$ follows by interchanging the two factors.

Let $D$ be a derivation of
$\mathcal N_{n,t_1}\rtimes\mathcal N_{n,t_2}$
that commutes with the $\mathfrak{sl}_2$-action.
As in the proof of Proposition~\ref{prop2.7}, equivariance implies
$$D(\boldsymbol{\xi}_i)=
\begin{pmatrix}
\alpha_1 & \alpha_2 & \alpha_3 & \alpha_4\\
\beta_1  & \beta_2  & \beta_3  & \beta_4\\
\gamma_1 & \gamma_2 & \gamma_3 & \gamma_4\\
\delta_1 & \delta_2 & \delta_3 & \delta_4
\end{pmatrix}
\boldsymbol{\xi}_i,
\qquad \boldsymbol{\xi}_i=(x_i,y_i,u_i,v_i)^T,
\qquad i=1,2.$$

From the chain of equalities
$$D(c_1) = D([x_1, x_2]) = D([y_1, y_2])= D([x_1, u_2]) = D([y_1, v_2]),$$
we derive $D(c_1) = 2\alpha_1 c_1$. Similarly, the relations
$$D(z_1) = D([x_1, y_1]), \quad D(x_1^3) = D([x_1, z_1]), \quad D([y_1, x_1^3]) = 0,$$ imply $D_{|\{x_1, x_2, y_1, y_2\}} = \alpha_1 \mathrm{id}.$ Next, checking the Leibniz rule for the following brackets,
$$D(c_2) = D([u_1, u_2]), \quad D(c_2) = D([v_1, v_2]), \quad D(d_1) = D([u_1, v_1]), \quad D([v_1, u_1^3]) = 0,$$
yields $\delta_1 = \delta_3 = \gamma_2 = \gamma_4 = 0.$

Applying arguments analogous to those used in the proof of 
Proposition~\ref{prop2.5}, we obtain
$$\begin{array}{lclll}
D(y_1^{n-1})=&(n-1)\alpha_1 y_1^{n-1}&=(n+1)\alpha_1 y_1^{n-1}, \\[3mm]
D(v_1^{n-1})=&(n-1)\gamma_1y_1^{n-1}+(n-1)\gamma_3v_1^{n-1}&=(n+1)\gamma_1y_1^{n-1}+(n+1)\gamma_3v_1^{n-1}.\\
\end{array}$$
Hence, $\alpha_1 = \gamma_1 = \gamma_3 = 0$. 
Therefore, we conclude that $D = 0$, which completes the proof.
\end{proof}

Let $\{x_1^{(i)}, x_2^{(i)}, y_1^{(i)}, y_2^{(i)}\}$ be the generating set for the algebra $\mathcal{N}_{n,t_i}$. For later use, we introduce the following notation:
$$
\phi_{s\sim k,\ell}
=\displaystyle\sum_{m=s}^{k}\phi_{m,\ell},
\qquad
\phi_{\ell,s\sim k}
=\displaystyle\sum_{m=s}^{k}\phi_{\ell,m},
$$
where each map
$\phi_{m,\ell}\colon \mathcal{N}_{n,t_m}\rightarrow
\operatorname{Der}\bigl(\mathcal{N}_{n,t_\ell}\bigr)$
is defined by \eqref{eq9}.

Then the structure of the algebra
$$ \mathfrak{sl}_2\mathcal{G}(t_i) := \mathfrak{sl}_2\ltimes \Big(
\mathcal{N}_{n,t_1}* \big(
\mathcal{N}_{n,t_2}* \cdots * \big(
\mathcal{N}_{n,t_{k-1}}* \mathcal{N}_{n,t_k}
\big) \big)
\Big) \quad \mbox{with} \quad t_1,\ldots,t_k\in\mathbb C^\times
$$
where $*\in\{\ltimes,\rtimes\}$, is completely determined by the set of actions
$$\Phi=\{\phi_{\ell,s\sim k},\phi_{s\sim k,\ell}\mid 2\leq s\leq k,\ 1\leq \ell<s\}.$$

\begin{theorem}\label{thm2.9}
The algebra $\mathfrak{sl}_2\mathcal{G}(t_i)$ 
admits only inner derivations.
\end{theorem}
\begin{proof}  We argue by induction on the number of factors, proving that every
derivation of the iterated product commuting with the $\mathfrak{sl}_2$-action is zero. The cases of one and two factors are established in the proofs of  Propositions~\ref{prop2.5}
and~\ref{prop2.8}, respectively.

Write $\mathcal N_i=\mathcal N_{n,t_i}$ and set $\widetilde{\mathcal N} = \mathcal N_2*
\bigl(\mathcal N_3*\cdots* (\mathcal N_{k-1}*\mathcal N_k)\bigr).$

Let $D\in\Der(\mathcal N_1*\widetilde{\mathcal N})$ that commutes with
the $\mathfrak{sl}_2$--action. We are going to prove that restriction of
$D$ on the set of generators
$$
\mathcal U=
\{x_i^{(1)},y_i^{(1)},\cdots,x_i^{(k)},y_i^{(k)}\}
$$
is trivial, which implies that $D=0$. Due to Schur's lemma, one can
assume that the matrix form of
$D_{|\mathcal U}=(a_{i,j})$ has the following form
$$
\begin{pmatrix}
D_{11} & D_{12}\\
D_{21} & D_{22}
\end{pmatrix}.
$$
Then by induction we have $D_{11}=D_{22}=0$. Note that due to the given
set of actions $\Phi$ there exists a component $\mathcal N_s$ in
$\widetilde{\mathcal N}$ such that for all $2\leq \ell\leq k$ the
embedding $[\mathcal N_\ell,\mathcal N_s]\subseteq\mathcal N_\ell$
holds. Then regarding chosen such $s$ we distinguish the possible cases.

{\bf Case~1}: Let
$\mathfrak{sl}_2\ltimes(\mathcal N_1\rtimes\widetilde{\mathcal N})$.
Then applying derivation property in the following chain of equalities
$$
D(c^{(1)})
=
D([x_1^{(1)},x_2^{(1)}])
=
D([x_1^{(1)},x_2^{(s)}])
=
D([y_1^{(1)},y_2^{(1)}])
=
D([y_1^{(1)},y_2^{(s)}])
$$
and compare coefficients at appropriate basis elements, we deduce
$D_{12}=0$.

Similarly, from equalities
$$
D(c^{(i)})
=
D([x_1^{(i)},x_2^{(i)}])
=
D([y_1^{(i)},y_2^{(i)}])
$$
it follows that
$a_{2i-1,1}=a_{2i,2}$, $2\leq i\leq k$.

Computing the values of
$$
D(z_1^{(i)}),\qquad
D\big((x_1^{(i)})^3\big),\qquad
D\big([y_1^{(i)},(x_1^{(i)})^3]\big)
$$
and applying the same arguments as in the proof of
Proposition~\ref{prop2.5}, we obtain
$a_{2i-1,2}=a_{2i,1}=0$ for $2\leq i\leq k$.
Moreover, by recursion we derive
$$
D\big((y_1^{(i)})^{n-1}\big)
=
(n-1)a_{2i-1,1}(y_1^{(i)})^{n-1}
=
(n+1)a_{2i-1,1}(y_1^{(i)})^{n-1}.
$$
It implies $a_{2i-1,1}=0$. Therefore, $D_{21}=0$ and hence, $D=0$.

{\bf Case~2}: Let
$\mathfrak{sl}_2\ltimes(\mathcal N_1\ltimes\widetilde{\mathcal N})$.
Then from equalities
$$
D(c^{(s)})
=
D([x_1^{(s)},x_2^{(s)}])
=
D([x_1^{(1)},x_2^{(s)}])
=
D([y_1^{(1)},y_2^{(1)}])
=
D([y_1^{(1)},y_2^{(s)}]).
$$
we ensure that
$$
a_{2s-1,1}=a_{2s-1,2}=a_{2s,1}=a_{2s,2}=0,
\qquad
a_{1,i}=a_{2,i}=0,
\qquad
3\leq i\leq2k.
$$
Consequently, we obtain $D_{12}=0$.

Now, for any $i\notin\{1,s\}$ computing
$
D(c^{(i)})
=
D([x_1^{(i)},x_2^{(s)}]),
$
yields $a_{2i-1,1}=a_{2i-1,2}=0$, which implies $D_{21}=0$.
This completes the proof.
\end{proof}

We refer to the Lie algebras in the family
$\mathfrak{sl}_2\mathcal{G}(t_i)$ as \emph{generalized Sato-type Lie algebras};
the algebras $\mathcal L_{n,t}$ form the basic members of this class.

Below we show that the construction in Proposition~\ref{prop2.1} cannot
be extended by replacing the two generating copies of $V_1$ with arbitrary
irreducible $\mathfrak{sl}_2$-modules. More precisely, suppose that
$\mathcal{V}_1\cong V_n$ and $
\mathcal{W}_1\cong V_m,$
where $V_i$ denotes the irreducible $\mathfrak{sl}_2$-module of highest
weight $i$. We show that the construction cannot produce a quasi-cyclic
nilpotent Lie algebra with non-trivial center unless $(n,m)=(1,1)$.

Assume first that $[\mathcal{V}_1,\mathcal{V}_1]\neq0$. According to the
construction, we would then have
$[x_i,x_j]=z_k$
for some $x_i,x_j\in\mathcal{V}_1$ and $z_k\in\mathcal{V}_2$.
Considering the identity
\[
[x_i,[x_j,y_\ell]]
=
[[x_i,x_j],y_\ell]+[x_j,[x_i,y_\ell]],
\]
we observe that the left-hand side lies in $\mathcal{V}_3$, whereas the
right-hand side has a non-zero component in $\mathcal{W}_3$. This
contradiction forces
$[\mathcal{V}_1,\mathcal{V}_1]=0.$

Next, suppose that
$[x_1,x_{n+1}]
=
c\in Z(\mathfrak{sl}_2\ltimes\mathcal N).$
Examining the Jacobi identity for the triple
$\{e,x_2,x_{n+1}\}$ again leads to a contradiction. Consequently, the
vector space
$\mathfrak{sl}_2\oplus\mathcal V_1\oplus\mathcal W_1$
cannot be endowed with a Lie algebra structure via the construction
described in Proposition~\ref{prop2.1}.

We now consider the case in which $\mathcal{N}$ has three pairs of generators. 
To the modules $\mathcal{V}_1$ and $\mathcal{W}_1$, we adjoin a third pair 
$\mathcal{M}_1 = \{m_1, m_2\}$, extending the relations \eqref{eq1}-\eqref{eq2} between 
$\mathcal{V}_1$ and $\mathcal{W}_1$ by imposing analogous relations between 
$\mathcal{W}_1$ and $\mathcal{M}_1$, and denoting by $d_1, d_2, d_3$ the elements 
corresponding to $z_1, z_2, z_3$ in this new setting. Then we obtain 
$$\mathcal{N}=\mathcal{U}^1\oplus \mathcal{U}^2\oplus\cdots \oplus \mathcal{U}^n,$$
where $\mathcal{U}^k=\operatorname{Span}\{\mathcal{V}_k\cup \mathcal{W}_k\cup \mathcal{M}_k\},$ for $k\in \{1, 3, \dots, n\}$ and $\mathcal{U}^2=\operatorname{Span}\{c,z_1,z_2,z_3,d_1,d_2,d_3\}.$

Now we state the result similar to Proposition \ref{prop2.1}. 
\begin{proposition}\label{prop2.10}
Let $n\geq5$ be odd, and let 
$\mathcal N=\mathcal U^1\oplus\mathcal U^2\oplus\cdots\oplus\mathcal U^n,$
where the subspaces $\mathcal U^i$ are defined as above.
Equip $\mathcal N$ with the previously specified products and the
additional brackets
\begin{equation}\label{eq10}
\begin{aligned}
[z_j,y_i^{n-2}]&=-m_{i+j-1}^{n},
&&1\leq i\leq n-1,\quad 1\leq j\leq3,\\
[x_1,y_i^{n-1}]&=-m_i^{n},
&&1\leq i\leq n,\\
[x_2,y_i^{n-1}]&=-m_{i+1}^{n},
&&1\leq i\leq n.
\end{aligned}
\end{equation}
Then $\mathcal N$ is a quasi-cyclic Lie algebra.
\end{proposition}
\begin{proof} 
From the identity $\mathcal{J}(y_1, z_j, y_i^{\,n-3}) = 0$ we obtain the corresponding products, which can then be substituted into the relations $\mathcal{J}(x_j, y_1, y_i^{\,n-2}) = 0$ for $j=1,2$. This procedure yields the remaining products. 
\end{proof}

By the same argument as in the remark following Proposition~\ref{prop2.1}, one can verify that the statement of Proposition~\ref{prop2.10} does not hold for even values of $n$.

Despite the construction of $\mathfrak{N}(n)$ for two copies of $V_1$, Sato's extended construction does not work for three copies of $V_1$. Specifically, the method of Proposition~\ref{prop2.3} cannot be applied to the quasi-cyclic Lie algebra obtained from Proposition~\ref{prop2.10} to yield a non-quasi-cyclic structure. Taking $t=1$ in the application to the $\mathcal W$ and $\mathcal M$ parts, we obtain $[m_1, y_2^n] = -m_1^{n-1}$ and $[m_2, y_1^n] = n m_1^{n-1}$. Then, applying the Jacobi identity to the triples $\{y_1, m_2, x_1^{n-1}\}$ and $\{y_2, m_1, x_1^{n-1}\}$ leads to a contradiction.

\begin{proposition}\label{prop2.11}
Let $\mathcal L=\mathfrak{sl}_2 \ltimes \mathcal{N}\,$ be a Lie algebra whose nilradical $\mathcal{N}$ is as in Proposition~\ref{prop2.1} or Proposition~\ref{prop2.10}, and suppose that the action of $\mathfrak{sl}_2$ on $\mathcal{N}$ is given by~\eqref{eq4}. Then \[
 \dim\bigl(
 \operatorname{Der}(\mathcal L)/
 \operatorname{ad}(\mathcal L)
 \bigr)=1.
\]
\end{proposition}
\begin{proof}
Let $D\in\operatorname{Der}(\mathcal L)$.
The restriction $D_{| \mathfrak{sl}_2}$ is a
$1$-cocycle of $\mathfrak{sl}_2$ with values in $\mathcal L$. By Whitehead's first lemma, there exists $x\in\mathcal L$ such that
$d_2=D-\operatorname{ad}(x)$ vanishes on
$\mathfrak{sl}_2$.
Since $\mathcal N$ is characteristic in $\mathcal L$, $d_2$ restricts to a derivation of $\mathcal N$. Moreover, $d_2([s,u])=[s,d_2(u)]$ for $s\in\mathfrak{sl}_2$ and $u\in\mathcal N$, so this
restriction is an $\mathfrak{sl}_2$-module homomorphism.

The same argument as in the proof of
Proposition~\ref{prop2.5} shows that
${d_2}_{| \mathcal U^1}=\lambda\operatorname{id}_{\mathcal U^1}$
for some $\lambda\in\mathbb C$.
Since $\mathcal U^k=[\mathcal U^1,\mathcal U^{k-1}]$
for $2\leq k\leq n$, induction gives
$${d_2}_{| \mathcal U^k} =k\lambda\operatorname{id}_{\mathcal U^k},
\qquad 1\leq k\leq n.$$
Thus, $d_2=\lambda E$, and hence
$D=\operatorname{ad}(x)+\lambda E$.

To show that $E$ is outer derivation, suppose that $E=\operatorname{ad}(y)$ for some $y\in\mathcal L$.
Since $E$ vanishes on $\mathfrak{sl}_2$, we have $[y,\mathfrak{sl}_2]=0$.
The decomposition
$\mathcal L=\mathfrak{sl}_2\ltimes\mathcal N$
and the equality $Z(\mathfrak{sl}_2)=0$ imply that $y\in\mathcal N$.
Therefore, $y$ belongs to the trivial
$\mathfrak{sl}_2$-submodule of $\mathcal N$,
which is central in $\mathcal L$.
This gives $E=\operatorname{ad}(y)=0$,
contrary to $E_{| \mathcal U^1}
=\operatorname{id}_{\mathcal U^1}$.
\end{proof}

\subsection{Quasi-cyclic nilradicals with an
\texorpdfstring{$\mathfrak{sl}_m$}{sl(m)}-action}

For $\mathcal L_{n,t}$, the one-dimensional trivial
$\mathfrak{sl}_2$-module $\langle c\rangle$ coincides with
$Z(\mathcal L_{n,t})$. The following result
shows that an analogous central trivial module cannot be obtained from
the natural $\mathfrak{sl}_3$-module by the same approach.

According to our convention, the action of
$\mathfrak{sl}_3=\operatorname{Span}\{h_1,h_2,e_{12},e_{23},e_{13},f_{12},f_{23},f_{13}\}$
on its natural module
$\mathcal V=\operatorname{Span}\{x_1,x_2,x_3\}$ is given by
$$
\begin{array}{lllll}
[h_1,x_1]=x_1, & [h_1,x_2]=-x_2, &
[h_2,x_2]=x_2, & [h_2,x_3]=-x_3, &
[e_{12},x_2]=x_1, \\[3mm]
[e_{23},x_3]=x_2, & [e_{13},x_3]=x_1, &
[f_{12},x_1]=x_2, & [f_{23},x_2]=x_3, &
[f_{13},x_1]=x_3.
\end{array}
$$

\begin{proposition}\label{cla2.12}
There is no semidirect product Lie algebra
$\mathcal L=\mathfrak{sl}_3\ltimes\mathcal N$
in which $\mathcal V\subseteq\mathcal N$ carries the natural
$\mathfrak{sl}_3$-action and
$$[\mathcal V,[\mathcal V,\mathcal V]] = \langle c\rangle \subseteq Z(\mathcal L),$$
where $c\neq0$ spans a trivial $\mathfrak{sl}_3$-module.
\end{proposition}

\begin{proof}
Assume the contrary. Let
$\mathcal V$ has corresponding standard weights
$\lambda_1,\lambda_2,\lambda_3$. We set
$$
Z(\mathcal L)=\langle c\rangle,
\qquad \mbox{and} \qquad
[x_1,[x_2,x_3]]=\lambda c,
\quad \lambda\neq0.
$$
Then $c$ has weight $(0,0)$. Putting
$\wedge^2\mathcal V=\operatorname{Span}\{w_1,w_2,w_3\}$, where
$\operatorname{wt}(w_i)=-\lambda_i$, we may write
$$
[x_2,x_3]=a_1w_1,\qquad
[x_3,x_1]=a_2w_2,\qquad
[x_1,x_2]=a_3w_3,
\qquad a_i\in\mathbb C.
$$
Note that $a_1\neq0$, since $[x_2,x_3]\neq0$.

Applying $e_{12}$ to $[x_2,x_3]=a_1w_1$ and $e_{23}$ to
$[x_3,x_1]=a_2w_2$, respectively, and using the
$\mathfrak{sl}_3$-equivariance of the bracket, we obtain
$a_1=a_2=a_3$. By rescaling the basis elements $w_i$, if necessary,
we may assume that
$a_1=a_2=a_3=1.$
Hence
$$
[x_2,x_3]=w_1,\qquad
[x_3,x_1]=w_2,\qquad
[x_1,x_2]=w_3,
$$
and therefore
$[x_1,w_1]=\lambda c.$

Since $[\mathcal V,[\mathcal V,\mathcal V]]\subseteq\langle c\rangle$
and $\langle c\rangle$ is a trivial $\mathfrak{sl}_3$-module, it follows
from
$f_{12}\cdot[x_1,w_1]=0$
that
$[x_2,w_1]=0.$
Similarly, after obtaining $[x_2,w_2]\in\langle c\rangle$, the action of
$f_{23}$ gives
$[x_3,w_2]=0.$

Applying the Jacobi identity to the triples
$\{e_{12},x_2,w_1\}$ and $\{e_{23},x_3,w_2\}$, and using
$[e_{12},w_1]=-w_2$ and $[e_{23},w_2]=-w_3$, respectively, we obtain
$$
[x_1,w_1]=[x_2,w_2]=[x_3,w_3]=\lambda c.
$$
Finally, the Jacobi identity for the triple
$\{x_1,x_2,x_3\}$ gives
$$
0
=
[x_1,[x_2,x_3]]
+[x_2,[x_3,x_1]]
+[x_3,[x_1,x_2]]
=
3\lambda c,
$$
which contradicts $\lambda\neq0$. Therefore such a Lie algebra
structure does not exist.
\end{proof}

Therefore, Sato’s construction of $\mathfrak{sl}_2$ with the trivial module (corresponding to the one-dimensional center of $\mathcal{L}$) cannot be extended to the case of $\mathfrak{sl}_3$. However, by excluding the trivial module case, it is still possible to construct a Lie algebra of the form $\mathfrak{sl}_m \ltimes \mathcal{N}$ ($m \ge 3$) whose nilradical $\mathcal{N}$ is quasi-cyclic.

Let $\mathcal V$ be the natural $\mathfrak{sl}_m$-module with the standard basis $\{e_1, e_2, \dots, e_m\}$ and let $S(\mathcal V) = \bigoplus\limits_{k \geq 0} \operatorname{Sym}^k(\mathcal V)$ be the symmetric algebra (isomorphic to the polynomial ring \(\mathbb{C}[e_1, \dots, e_m]\)). We denote the commutative multiplication in \(S(\mathcal V)\) by ``\(\cdot\)''.

The space \(\operatorname{Sym}^k(\mathcal V)\) has a standard monomial basis:
$$\{ e_1^{a_1} \cdot e_2^{a_2} \cdot \ldots \cdot e_m^{a_m} \mid a_1 + a_2 + \dots + a_m = k, \ a_i \ge 0 \}.$$ On the set of standard monomial basis we consider lexicographical order.

We denote by $\mathcal{B}_k = \{b_1^k, b_2^k, \dots, b_{d_k}^k\}$, where 
$d_k = \binom{m+k-1}{k}$, the basis of $\operatorname{Sym}^k(\mathcal V)$ 
ordered in descending lexicographical order. We set $\mathcal{N} = \bigoplus\limits_{k=1}^{n} \mathcal{U}^k,$
identifying its components with copies of subspaces of $S(\mathcal V)$ as follows:
$$\mathcal{U}^2 \cong \operatorname{Sym}^2(\mathcal V), \quad
\mathcal{U}^k = \mathcal{V}_k \oplus \mathcal{W}_k,$$
where $\mathcal{V}_k \cong \mathcal{W}_k \cong \operatorname{Sym}^k(\mathcal V)$ for $k \in \{1,3,\dots,n\}.$

We define linear isomorphisms to identify the Lie algebra components with the symmetric powers
$$\rho_{\mathcal{U}^2} : \ \operatorname{Sym}^2(\mathcal V) \to \mathcal{U}^2, \quad \rho_{\mathcal V,k} : \ \operatorname{Sym}^k(\mathcal V) \to \mathcal{V}_k, \quad 
\rho_{\mathcal W,k} : \ \operatorname{Sym}^k(\mathcal V) \to \mathcal{W}_k,$$
by setting 
$$z_i = \rho_{\mathcal U^2}(b_i^2), \quad x_i^k = \rho_{\mathcal V,k}(b_i^k), \quad y_i^k = \rho_{\mathcal W,k}(b_i^k).$$ 

The Lie brackets on $\mathcal{N}$ are defined using the isomorphisms 
$\rho_{\mathcal{U}^2}, \rho_{\mathcal V,k}, \rho_{\mathcal W,k}$ and the multiplication in $S(\mathcal V)$. The non-zero brackets on $u, v \in \operatorname{Sym}^1(\mathcal V)$ and any monomial $P$ in $\operatorname{Sym}^k(\mathcal V)$ are given by
\begin{equation}\label{eq11}
\begin{array}{lllllll}
[\rho_{\mathcal V,1}(u),\, \rho_{\mathcal W,1}(v)] = \rho_{\mathcal{U}^2}(u \cdot v),& [\rho_{\mathcal W,1}(u),\, \rho_{\mathcal V,n-1}(P)] = \rho_{\mathcal W,n}(u \cdot P) \\[3mm]
    [\rho_{\mathcal V,1}(u),\, \rho_{\mathcal{U}^2}(P)] = \rho_{\mathcal V,3}(u \cdot P),& 
    [\rho_{\mathcal W,1}(u),\, \rho_{\mathcal{U}^2}(P)] = \rho_{\mathcal W,3}(u \cdot P), \\[3mm]
    [\rho_{\mathcal V,1}(u),\, \rho_{\mathcal V,k}(P)] = \rho_{\mathcal V,k+1}(u \cdot P),&
    3 \le k \le n-1, \\[3mm] 
    [\rho_{\mathcal W,1}(u),\, \rho_{\mathcal W,k}(P)] = \rho_{\mathcal W,k+1}(u \cdot P),&
     3 \le k \le n-2. \\[3mm]
\end{array}
\end{equation}
It should be noted that the basis elements of $\mathcal{N}$ are equipped with the weights induced by the isomorphisms $\rho$.

In order to provide a clearer explanation of the brackets described above, we present a diagram that illustrates the relationship between the homogeneous components $\mathcal{U}^k$:
\begin{center} 
\begin{tikzpicture}[
scale=0.6, transform shape, 
node distance=1.2cm and 1.5cm,
every node/.style={align=center}, 
mycircle/.style={font=\large}, 
arrow/.style={->, >=Stealth, thick},
dashed_line/.style={dashed, thick}
]
    \node[mycircle] (n1) at (-3, 0) {$\{x_1^1,\ldots,x_{d_1}^1\}$};
    \node[mycircle] (n2) at (3, 0)  {$\{y_1^1,\ldots,y_{d_1}^1\}$};
    \node[mycircle] (n3) at (0, -2) {$\{z_1,\ldots, z_{d_2}\}$};
    \node[mycircle] (n4) at (-3, -4) {$\{x_1^3,\ldots, x_{d_3}^3\}$};
    \node[mycircle] (n8) at (3, -4)  {$\{y_1^3,\ldots, y_{d_3}^3\}$};
    \node[mycircle] (n5) at (-3, -6) {$\{x_1^4,\ldots, x_{d_4}^4\}$};
    \node[mycircle] (n9) at (3, -6)  {$\{y_1^4,\ldots, y_{d_4}^4\}$};
    \node (dotsL) at (-3, -7.2) {$\vdots$};
    \node (dotsR) at (3, -7.2)  {$\vdots$};   
    \node[mycircle] (nDotsL) at (-3, -8.5) {$\{\dots\}$};
    \node[mycircle] (nDotsR) at (3, -8.5)  {$\{\dots\}$};
    \node[mycircle] (n6) at (-3, -10.5) {$\{x_1^{n-1},\ldots, x_{d_{n-1}}^{n-1}\}$};
    \node[mycircle] (n10) at (3, -10.5) {$\{y_1^{n-1},\ldots, y_{d_{n-1}}^{n-1}\}$};
    \node[mycircle] (n7) at (-3, -12.5) {$\{x_1^{n},\ldots, x_{d_n}^{n}\}$};
    \node[mycircle] (n11) at (3, -12.5) {$\{y_1^{n},\ldots, y_{d_n}^{n}\}$};
    \draw[arrow] (n1) -- (n3);
    \draw[arrow] (n2) -- (n3);
    \draw[arrow] (n3) -- (n4);
    \draw[arrow] (n4) -- (n5);
    \draw[dashed_line] (n5) -- (nDotsL);
    \draw[dashed_line] (nDotsL) -- (n6);
    \draw[arrow] (n6) -- (n7);
    \draw[arrow] (n3) -- (n8);
    \draw[arrow] (n8) -- (n9);
    \draw[dashed_line] (n9) -- (nDotsR);
    \draw[dashed_line] (nDotsR) -- (n10);
    \draw[arrow] (n1) to[bend right=45] (n4);
    \draw[arrow] (n1) to[bend right=50] (n5);
    \draw[arrow] (n1) to[bend right=60] (nDotsL);
    \draw[arrow] (n1) to[bend right=65] (n6);
    \draw[arrow] (n1) to[bend right=70] (n7);
    \draw[arrow] (n2) to[bend left=45] (n8);
    \draw[arrow] (n2) to[bend left=50] (n9);
    \draw[arrow] (n2) to[bend left=60] (nDotsR);
    \draw[arrow] (n2) to[bend left=65] (n10); 
    \draw[arrow] (n2) to[bend left=70] (n11); 
        \draw[arrow] (n6) -- (n11);
\end{tikzpicture}
\end{center}
where the arrows represent the brackets specified in~\eqref{eq11}.


For convenience, we set
$\rho_{\mathcal V,2}:=\rho_{\mathcal U^2}$.

\begin{theorem}\label{thm2.13}
Let $n\geq5$ be odd, and let $\mathcal N$ be the vector space
defined above, equipped with the products~\eqref{eq11}.
Then $\mathcal N$ is a Lie algebra if and only if its multiplication
also satisfies
\begin{equation}\label{eq12}
[\rho_{\mathcal V,p}(A),\rho_{\mathcal V,n-p}(B)]
=
(-1)^{p-1}\rho_{\mathcal W,n}(A\cdot B),
\qquad
2\leq p\leq\frac{n-1}{2},
\end{equation}
for all $A\in\operatorname{Sym}^{p}(\mathcal V)$ and
$B\in\operatorname{Sym}^{n-p}(\mathcal V)$.
\end{theorem}
\begin{proof} 
A straightforward verification of the Jacobi identity shows that any $\mathcal{N}$ satisfying \eqref{eq11}--\eqref{eq12} is indeed a Lie algebra.

We now show, by induction on $p$, that the relations \eqref{eq12} are consequences of the Jacobi identity and \eqref{eq11}.
For elements $u, v \in \operatorname{Sym}^{1}(\mathcal{V})$ and 
$P \in \operatorname{Sym}^{n-2}(\mathcal{V})$, the equality
$\mathcal{J}\big(\rho_{\mathcal V,1}(u), \rho_{\mathcal W,1}(v), \rho_{\mathcal V,n-2}(P)\big)=0$
implies
$$[\rho_{\mathcal V,2}(u \cdot v), \rho_{\mathcal V,n-2}(P)]
= - \rho_{\mathcal W,n}(u \cdot v \cdot P),$$
which establishes the base step of the induction. 

Let \(A = u \cdot A' \in \operatorname{Sym}^{p+1}(\mathcal{V})\) with  
\(u \in \operatorname{Sym}^{1}(\mathcal{V})\), \(A' \in \operatorname{Sym}^{p}(\mathcal{V})\), 
and let \(B \in \operatorname{Sym}^{n-p-1}(\mathcal{V})\).
Using the induction hypothesis, we obtain
\[
[\rho_{\mathcal V,p}(A'),\,
 [\rho_{\mathcal V,n-p-1}(B), \rho_{\mathcal V,1}(u)]]
= (-1)^{p} \rho_{\mathcal W,n}(A \cdot B).
\]

Finally, applying the Jacobi identity to the triple  
\(\{\rho_{\mathcal V,1}(u),\, \rho_{\mathcal V,n-p-1}(B),\, \rho_{\mathcal V,p}(A')\}\),
we obtain
\[
[\rho_{\mathcal V,p+1}(A),\rho_{\mathcal V,n-p-1}(B)]
= (-1)^p \rho_{\mathcal W,n}(A\cdot B),
\]
which completes the induction step and hence the proof of \eqref{eq12}.  
The remaining products can be checked directly to vanish.
\end{proof}

Note that \(\mathcal{N}\) in Theorem \ref{thm2.13} is a quasi-cyclic Lie algebra.

\begin{proposition}\label{prop2.14}
Let $m\geq3$ and let $\mathcal{L} = \mathfrak{sl}_m \ltimes \mathcal{N}$ be the Lie algebra described in Theorem~\ref{thm2.13}, where $\mathcal{V}_1$ and $\mathcal{W}_1$ are copies of the natural $\mathfrak{sl}_m$–module. Then \[
 \dim\bigl(
 \operatorname{Der}(\mathcal L)/
 \operatorname{ad}(\mathcal L)
 \bigr)=2.
\]
\end{proposition}
\begin{proof} Arguing as in the proof of Proposition~\ref{prop2.11}, any derivation 
$D \in \operatorname{Der}(\mathcal{L})$ admits a decomposition $D = D_0 + \operatorname{ad}(x),$ where \(D_0\) vanishes on \(\mathfrak{sl}_m\) and commutes with the $\mathfrak{sl}_m$–action. In particular,
${D_0}_{| \mathcal{U}^2} = \eta\,\mathrm{id}$, and on each pair 
\((\mathcal{V}_k, \mathcal{W}_k)\), $k\in\{1,3,\dots,n\}$, we may write
\[
D_0(v) = a_k v + c_k w, 
\qquad
D_0(w) = b_k v + d_k w,
\]
where \(v \in \mathcal{V}_k\) and  
\(w \in \mathcal{W}_k\).

The relations
\begin{equation}\label{eq13}
[\mathcal{W}_1,\mathcal{V}_{n-1}] = \mathcal{W}_n,\qquad
[\mathcal{W}_1,\mathcal{U}^2]=\mathcal{W}_3,\qquad
[\mathcal{W}_1,\mathcal{W}_k]=\mathcal{W}_{k+1},\quad 3 \le k \le n-2,
\end{equation}
imply that \(b_k = 0\) for all \(k\).  
Replacing \(\mathcal{W}_i\) with \(\mathcal{V}_i\) in \eqref{eq13}, and using additionally the relation 
\([\mathcal{V}_1,\mathcal{W}_{n-1}] = 0\), an identical argument shows that \(c_k = 0\) for all \(k\).

Next, the brackets
\[
[\mathcal{V}_1,\mathcal{W}_1]=\mathcal{U}^2,\quad
[\mathcal{V}_1,\mathcal{U}^2]=\mathcal{V}_3,\quad
[\mathcal{V}_1,\mathcal{V}_k]=\mathcal{V}_{k+1},\quad
[\mathcal{W}_1,\mathcal{W}_k]=\mathcal{W}_{k+1},\quad 3\le k\le n-2,
\]
yield the recurrence relations
\[
a_{k+1} = a_1 + a_k, \qquad
d_{k+1} = d_1 + d_k.
\]
Solving these recurrences gives
\[
a_k = (k-1)a_1 + d_1,\qquad 3\le k\le n, \quad 
d_k = a_1 + (k-1)d_1,\qquad 3\le k\le n-1.
\]

Finally, the bracket 
\([\mathcal{W}_1,\mathcal{V}_{\,n-1}] = \mathcal{W}_n\) forces
$d_n = (n-2)a_1 + 2d_1.$ Thus the entire family of coefficients \(\{a_k,d_k\}\) is determined linearly by the two free parameters \(a_1\) and \(d_1\). The proof is complete. \end{proof}

\section{Second Cohomology of Generalized Sato-Type Lie Algebras}\label{sec3}

In this section, we combine the Hochschild--Serre factorization theorem
with the decomposition of $\mathfrak{sl}_2$-modules into irreducible
summands to study the second adjoint cohomology of generalized
Sato-type Lie algebras.

We begin by recalling the necessary definitions from Lie algebra
cohomology; see~\cite{CE-1978,jaco62} for further details. Let $\mathcal{L}$ be a Lie algebra and let $\mathcal{M}$ be an $\mathcal{L}$-module. For each $n \geq 0$, the space of $n$-cochains is defined as $C^n(\mathcal{L}, \mathcal{M}) = \operatorname{Hom}_{\mathbb{F}}(\wedge^n \mathcal{L}, \mathcal{M})$. For $n=0$, we set $C^0(\mathcal{L}, \mathcal{M}) = \mathcal{M}$.
The Chevalley–Eilenberg coboundary operator $d_n\colon  C^n(\mathcal{L}, \mathcal{M}) \to C^{n+1}(\mathcal{L}, \mathcal{M})$ is defined by 
\begin{align*}
(d_n f)(x_1, \dots, x_{n+1}) = & \sum_{i=1}^{n+1} (-1)^{i+1} x_i \cdot f(x_1, \dots, \hat{x}_i, \dots, x_{n+1}) \\
& + \sum_{1 \le i < j \le n+1} (-1)^{i+j} f([x_i, x_j], x_1, \dots, \hat{x}_i, \dots, \hat{x}_j, \dots, x_{n+1}),
\end{align*}
where the notation $\hat{x}_i$ indicates that the element $x_i$ is omitted. 

The space of $n$-cocycles is $ Z^n(\mathcal{L}, \mathcal{M}) = \ker d_n$, and the space of $n$-coboundaries, $ B^n(\mathcal{L}, \mathcal{M}) = \operatorname{im} d_{n-1}$.
The $n$-th cohomology group of $\mathcal{L}$ with coefficients in $\mathcal{M}$ is the quotient vector space:
$$H^n(\mathcal{L}, \mathcal{M}) = Z^n(\mathcal{L}, \mathcal{M}) / B^n(\mathcal{L}, \mathcal{M}).$$

Although computing cohomology groups is a difficult problem, the Hochschild–Serre factorization theorem offers a significant simplification for certain classes of Lie algebras (see Theorem~13 in \cite{HS-1953}). 
We now state the theorem in the form needed for our purposes.

\begin{theorem}\label{thm3.1}  Let $\mathcal{L}=\mathcal S\ltimes \mathcal N$ be a finite-dimensional Lie algebra over a field $\mathbb{F}$ of characteristic $0$, and let $\mathcal{M}$ be a finite-dimensional $\mathcal{L}$-module. Suppose that $\mathcal{N}$ is an ideal of $\mathcal{L}$ such that the quotient $\mathcal{L}/\mathcal{N}$ is semi-simple. Then for all $n \geq 0$,
\begin{equation}\label{eq14}
H^{n}(\mathcal{L}, \mathcal{M}) 
\cong \sum_{i+j=n} H^{i}(\mathcal S, \mathbb{F}) \otimes
 H^{j}(\mathcal{N}, \mathcal{M})^{\mathcal S}.
\end{equation}
where 
\[
H^j(\mathcal N, \mathcal{M})^{\mathcal S} = \bigl\{ f \in H^j(\mathcal N, \mathcal{M}) \mid (s \cdot f) = 0,\ s \in \mathcal S \bigr\}
\]
is the space of $\mathcal S$-invariant cocycles of $\mathcal N$ with values in $\mathcal{M}$. The invariance being defined by
\[
(s \cdot f)(z_1, z_2, \dots, z_j) = s \cdot f(z_1, z_2, \dots, z_j) - \sum_{t=1}^j f(z_1, \dots, [s, z_t], \dots, z_j).
\]
\end{theorem}

\begin{remark}\label{rem3.2} 
The Whitehead lemmas give $H^i(\mathcal S,\mathbb F)=0$ for $i=1,2$.
Thus Theorem~\ref{thm3.1}, together with complete reducibility, yields
$$H^j(\mathcal L,\mathcal M)
\cong H^j(\mathcal N,\mathcal M)^{\mathcal S}
\cong \frac{Z^j(\mathcal N,\mathcal M)^{\mathcal S}} {B^j(\mathcal N,\mathcal M)^{\mathcal S}}, \qquad j=1,2.$$
\end{remark}

For irreducible $\mathfrak{sl}_2$-modules of the highest weights $n$ the Clebsch--Gordan formula give decompositions (see~\cite{Hum-78,Pul25}):
\begin{equation}\label{eq15}
V_n \otimes V_m \cong \bigoplus_{k=0}^{m} V_{n+m-2k}, \quad n \ge m, 
\qquad\qquad 
\wedge^2 V_n \cong \bigoplus_{k=1}^{\left\lfloor \frac{n+1}{2} \right\rfloor} V_{2n+2-4k}.
\end{equation}

Consider the Lie algebra \(\mathcal{L} = \mathfrak{sl}_2 \ltimes \mathcal{N}\), where \(\mathcal{N} = \mathcal{N}_{5,1}\). One verifies that
\begin{equation*}
\mathcal{N} \cong V_0 \oplus V_1 \oplus V_1' \oplus V_2 \oplus V_3 \oplus V_3' \oplus V_4 \oplus V_4' \oplus V_5 \oplus V_5',
\end{equation*}
where \(V_k'\) denotes a copy of \(V_k\). Accordingly, we obtain the following \(\mathfrak{sl}_2\)-module decomposition:
$$\mathcal{L} \cong V_2' \oplus \left( V_0 \oplus V_1 \oplus V_1' \oplus V_2 \oplus V_3 \oplus V_3' \oplus V_4 \oplus V_4' \oplus V_5 \oplus V_5' \right).$$

By Schur's lemma for completely reducible modules (see~\cite{Zhelobenko-1973}), the dimension of $\operatorname{Hom}_{\mathfrak{sl}_2}(\mathcal N,\mathcal L)$ is obtained by summing the products of the corresponding
multiplicities in $\mathcal N$ and $\mathcal L$ over all irreducible
isomorphism classes. This gives
\[
\dim C^1(\mathcal N,\mathcal L)^{\mathfrak{sl}_2}
=
\dim\operatorname{Hom}_{\mathfrak{sl}_2}(\mathcal N,\mathcal L)
=
19.
\]

To compute $\dim C^2(\mathcal N,\mathcal L)^{\mathfrak{sl}_2}$,
we use the canonical identification
$$C^2(\mathcal N,\mathcal L)^{\mathfrak{sl}_2}
\cong \operatorname{Hom}_{\mathfrak{sl}_2}
(\Lambda^2\mathcal N,\mathcal L)$$ 
and the decomposition
$$\Lambda^2\mathcal N \cong
\left(\bigoplus_{i=1}^{10}\Lambda^2W_i\right)
\bigoplus \left(\bigoplus_{1\leq i<j\leq10}W_i\otimes W_j\right),$$
where $\{W_1, \dots, W_{10}\} = \{V_0, V_1, V_1^{'}, \dots, V_5, V_5^{'}\}$ denotes the irreducible summands of $\mathcal{N}$. Applying Schur’s lemma we obtain 
$$\dim C^2(\mathcal{N}, \mathcal{L})^{\mathfrak{sl}_2} = \sum_k m_k(\Lambda^2 \mathcal{N}) \, m_k(\mathcal{L}),$$
where $m_k(\Lambda^2 \mathcal{N})$ and $m_k(\mathcal{L})$ denote the multiplicities of the irreducible module $V_k$ in $\Lambda^2 \mathcal{N}$ and $\mathcal{L}$, respectively. 

A symbolic computation in \textit{Mathematica} yields
$$C^2(\mathcal N,\mathcal L)^{\mathfrak{sl}_2}
= \operatorname{Span}\{\varphi_1,\ldots,\varphi_{196}\},$$
where basis elements $\varphi_i$ are determined by generalized Schur's Lemma. 

Taking into account the linear property of $d_2$ for an element  $\Psi = \sum\limits _{i}c_i \varphi_i$ of the space $Z^2(\mathcal{N}, \mathcal{L})^{\mathfrak{sl}_2}$ we have   
$\sum\limits_{i}c_i \big(d_2\varphi_i\big)(x, y, z)=0.$ 
Evaluating this identity on basis elements $x, y, z\in \mathcal{N}$ leads to a system of homogeneous linear equations in the coefficients $c_i$. Using the \textit{Mathematica} package, one obtains the corresponding matrix of this system. The computation shows that the solution space has dimension 21. Consequently,
$$\dim{Z}^2(\mathcal N,\mathcal L)^{\mathfrak{sl}_2} = 21.$$

Next, combining Proposition \ref{prop2.5} and Remark \ref{rem3.2} with $\mathcal{M}=\mathcal{L}$ we derive $H^1(\mathcal{N}, \mathcal{L})^{\mathfrak{sl}_2} = 0$. Hence, 
$$Z^1(\mathcal{N},\mathcal{L})^{\mathfrak{sl}_2}=B^1(\mathcal{N},\mathcal{L})^{\mathfrak{sl}_2}.$$

Let ${\operatorname{ad}_x}_{| \mathcal N}$ be an element of $B^1(\mathcal{N},\mathcal{L})^{\mathfrak{sl}_2}$ for some $x\in \mathcal{L}.$ Then the equality $[y,\operatorname{ad}_x(z)]=\operatorname{ad}_x([y,z])$ for all $y \in \mathfrak{sl}_2, z \in \mathcal{N}$,
implies that $\operatorname{ad}_x(\mathfrak{sl}_2)\subseteq Z(\mathcal{N}).$
By virtue of structure of $\mathcal{N}$ it  follows that $x\in Z(\mathcal{N})$, i.e., $B^1(\mathcal{N},\mathcal{L})^{\mathfrak{sl}_2}=0$. Therefore, we obtain $Z^1(\mathcal{N},\mathcal{L})^{\mathfrak{sl}_2}=0.$ In view of 
$$\dim B^2(\mathcal{N}, \mathcal{L})^{\mathfrak{sl}_2} = \dim C^1(\mathcal{N}, \mathcal{L})^{\mathfrak{sl}_2} $$
we derive 
$\dim B^2(\mathcal{N}, \mathcal{L})^{\mathfrak{sl}_2} =19.$

Applying now Remark~\ref{rem3.2} we obtain the following proposition.

\begin{proposition}\label{prop3.3}
The second adjoint cohomology of $\mathcal L_{5,1}$ is two-dimensional.
\end{proposition}

Regarding the family of generalized Sato-type Lie algebras
$\mathfrak{sl}_2\mathcal{G}(t_i)$, we obtain the following result.

\begin{theorem}\label{thm3.4}
Let $\mathcal L$ belong to the family
$\mathfrak{sl}_2\mathcal G(t_i)$. Then
$\dim H^2(\mathcal L,\mathcal L)\geq k.$
\end{theorem}

\begin{proof}
Let $\mathcal N$ be the nilradical of $\mathcal L$,
and put $\mathcal N_j=\mathcal N_{n,t_j}$.
Then $\mathcal N=\bigoplus\limits_{j=1}^k\mathcal N_j$
as a vector space.
By construction, $\mathcal N_j$ acts on $\mathcal N_i$
as in \eqref{eq9} whenever $i<j$.
In particular,
\[
[\mathcal N_i,\mathcal N_j]
\subseteq\mathcal N_{\min(i,j)}.
\]

For $1\leq j\leq k$, define a linear map
$\Psi_j:\bigwedge^2\mathcal N\to\mathcal L$ by
\[
\Psi_j(x_1^{(j)}\wedge y_2^{(j)})=c^{(j)},
\qquad
\Psi_j(x_2^{(j)}\wedge y_1^{(j)})=-c^{(j)},
\]
and let $\Psi_j$ vanish on all other basis wedges.
We use $\Psi_j(u,v)$ to denote $\Psi_j(u\wedge v)$.
Since $c^{(j)}\in Z(\mathcal L)$ and
$\Psi_j([\mathcal N,\mathcal N],\mathcal N)=0$,
we have $d\Psi_j=0$.
Moreover, \eqref{eq4} shows that $\Psi_j$ is an
$\mathfrak{sl}_2$-module homomorphism.
Set $\Psi_j(s,x)=0$ for $s\in\mathfrak{sl}_2$
and $x\in\mathcal L$. Then $\Psi_j$ is a
$2$-cocycle on $\mathcal L$.

We show that $[\Psi_1],\ldots,[\Psi_k]$ are linearly independent. By Remark~\ref{rem3.2} and complete reducibility, it suffices to prove that $\sum\limits_{\ell=1}^k\lambda_\ell\Psi_\ell=d\varphi$,
with $\varphi\in C^1(\mathcal N,\mathcal L)^{\mathfrak{sl}_2}$,
implies $\lambda_1=\cdots=\lambda_k=0$.
Fix $j$. Since $\varphi$ is an $\mathfrak{sl}_2$-module
homomorphism, Schur's lemma gives, for $i=1,2$,
\[
\varphi(x_i^{(j)})
=\sum_{\ell=1}^k
\bigl(\alpha_\ell x_i^{(\ell)}
+\beta_\ell y_i^{(\ell)}\bigr),
\qquad
\varphi(y_i^{(j)})
=\sum_{\ell=1}^k
\bigl(\gamma_\ell x_i^{(\ell)}
+\delta_\ell y_i^{(\ell)}\bigr).
\]
Set $(\beta,\gamma)=\sum\limits_{\ell=j}^k
(\beta_\ell,\gamma_\ell)$.
Applying the coboundary identity successively to pairs $\{x_1^{(j)},y_1^{(j)}\}$,
$\{x_1^{(j)},z_1^{(j)}\}$, and
$\{y_1^{(j)},(x_1^3)^{(j)}\}$,
and using \eqref{eq9}, we obtain $\beta=\gamma=0$ by comparing the coefficients of $(x_1^4)^{(j)}$ and $(y_1^4)^{(j)}$.
Applying the same identity to
$\{x_1^{(j)},y_2^{(j)}\}$ and comparing the
coefficients of $c^{(j)}$ then gives
$\lambda_j=\beta+\gamma=0$.
Since $j$ was arbitrary, all the coefficients
$\lambda_j$ vanish. Thus, $[\Psi_1],\ldots,[\Psi_k]$ are linearly
independent, and $\dim H^2(\mathcal L,\mathcal L)\geq k$.
\end{proof}

It is worth noting that the cocycles $\Psi_1,\ldots,\Psi_k$ constructed
in Theorem~\ref{thm3.4} are simultaneously integrable. Indeed, for all 
$i,j$ and $x,y,z\in\mathcal L$, one has
$\Psi_i\bigl(\Psi_j(x,y),z\bigr)=0,$ since $\operatorname{im}\Psi_j\subseteq Z(\mathcal L)$
and $\Psi_i(Z(\mathcal L),\mathcal L)=0$ for all $i,j$. Consequently, for every
$\boldsymbol{\tau}=(\tau_1,\ldots,\tau_k)\in\mathbb C^k$, the bracket
$$[x,y]_{\boldsymbol{\tau}} = [x,y]+\sum_{j=1}^k\tau_j\Psi_j(x,y)$$
defines a Lie algebra structure on $\mathcal L$. Thus, every algebra of the family $\mathfrak{sl}_2\mathcal G(t_i)$ admits at least $k$
independent integrable deformation directions and is, in particular,
formally and geometrically non-rigid. For further details on formal
deformations, we refer the reader to \cite{Ger-64,GO-93,Fox1993}.

In contrast to Theorem~\ref{thm3.4}, which establishes the non-vanishing of the second cohomology group for the algebra 
$\mathcal{L}=\mathfrak{sl}_2 \ltimes \mathcal{N}$ with $\mathcal{N}$ a direct sum of irreducible $\mathfrak{sl}_2$-modules, 
there exist examples for which $H^2(\mathcal{L},\mathcal{L})=0$.

\begin{example}
Let $\mathcal{L}=\mathfrak{sl}_2 \ltimes \mathcal{H}_{2n+1}$ be a Lie algebra, where $\mathcal{H}_{2n+1}$ denotes the 
$(2n+1)$-dimensional Heisenberg algebra. Then $\mathcal{H}_{2n+1}$ decomposes as a direct sum of two irreducible 
$\mathfrak{sl}_2$-modules, namely $V_{2n-1}$ and $V_0$. Applying arguments analogous to those used in the proof of 
Theorem~\ref{thm3.4}, one concludes that $H^2(\mathcal{L},\mathcal{L})=0$.
\end{example}

\section{Distinguished examples of Lie algebras with only inner derivations}\label{sec4}

In this section, we continue the construction of Lie algebras whose derivations are all inner. Starting from Angelopoulos’s example of sympathetic Lie algebras, we present its a natural generalization. We then modify this approach to obtain two further classes: complete but non-perfect Lie algebras, and perfect Lie algebras with non-trivial center.
We first introduce the tools needed for these constructions.

For each integer \(m\geq 0\), we realize the irreducible
\(\mathfrak{sl}_2\)-module \(V_m\) as the space of homogeneous polynomials of degree \(m\) in two variables,
$V_m=\operatorname{span}_{\mathbb C}\{X^{m-i}Y^i\mid 0\leq i\leq m\}.$ 
Equip \(\mathbb C[X,Y]\) with the Poisson bracket
$$\{u,v\} = \frac{\partial u}{\partial X}\frac{\partial v}{\partial Y} - \frac{\partial u}{\partial Y}\frac{\partial v}{\partial X}.$$
Identifying the standard generators of \(\mathfrak{sl}_2\) with the
quadratic polynomials
\[
h=-XY,\qquad e=\tfrac12 X^2,\qquad f=-\tfrac12 Y^2,
\]
we recover the action in~\eqref{eq4} as \(x\cdot u=\{x,u\}\).

For \(u\in V_a\), \(v\in V_b\), and \(0\leq r\leq\min\{a,b\}\),
define the \emph{\(r\)-th transvectant} by
\begin{equation}\label{eq:transvectant}
P_r(u,v)
=
\sum_{j=0}^{r}(-1)^j\binom{r}{j}
\frac{\partial^r u}{\partial X^{r-j}\partial Y^j}
\frac{\partial^r v}{\partial X^j\partial Y^{r-j}}.
\end{equation}
The induced map
\(P_r\colon V_a\otimes V_b\longrightarrow V_{a+b-2r}\)
is \(\mathfrak{sl}_2\)-equivariant. Explicitly, for every
\(x\in\mathfrak{sl}_2\),
\[
\{x,P_r(u,v)\}
=
P_r(\{x,u\},v)+P_r(u,\{x,v\}).
\]
Interchanging the arguments in~\eqref{eq:transvectant} gives
\(P_r(v,u)=(-1)^rP_r(u,v)\).
By the Clebsch--Gordan decomposition, each \(V_{a+b-2r}\) occurs
with multiplicity one in \(V_a\otimes V_b\). Hence the nonzero map
\(P_r\) spans $\operatorname{Hom}_{\mathfrak{sl}_2}
\bigl(V_a\otimes V_b,V_{a+b-2r}\bigr).$

When a decomposition contains several copies of \(V_m\), we write \(V_m^{(i)}\) for the \(i\)-th copy and \(w^{(i)}\) for the element
of that copy corresponding to \(w\in V_m\). In particular, \(P_r(u,v)^{(k)}\) denotes the value of the transvectant in
\(V_{a+b-2r}^{(k)}\).

\subsection{Sympathetic Lie algebras}

Recall that a Lie algebra \(\mathcal L\) is \emph{sympathetic} if it is perfect and centerless and every derivation of \(\mathcal L\) is inner. Angelopoulos~\cite{An-1988} constructed the first examples of
non-semisimple sympathetic Lie algebras. In his construction, the Levi factor is 
\(\mathfrak{sl}_2\) and the radical is nilpotent, which decomposes into a direct sum of four irreducible \(\mathfrak{sl}_2\)-modules.

In fact, for every \(n\geq 4\), Angelopoulos’s construction yields a sympathetic Lie algebra whose nilradical decomposes into exactly \(n\) irreducible \(\mathfrak{sl}_2\)-modules. Indeed, let $\mathcal{L}_1, \dots, \mathcal{L}_n$ be pairwise non-isomorphic, non-trivial irreducible $\mathfrak{sl}_2$-modules satisfying the following conditions:
\begin{itemize}
\item[(i)] $\mathcal{L}_1 \wedge \mathcal{L}_1$ contains a submodule isomorphic to $\mathcal{L}_2$ and $\mathcal{L}_j \wedge \mathcal{L}_j$, $3 \leq j \leq n-1$, contains a submodule isomorphic to $\mathcal{L}_n$. Moreover, $\mathcal{L}_1 \otimes \mathcal{L}_j$, $2 \leq j \leq n-1$,  contains a submodule isomorphic to $\mathcal{L}_n$;
\item[(ii)] $\wedge^3 \mathcal{L}_1$ contains no $\mathfrak{sl}_2$-submodule isomorphic to $\mathcal{L}_n$
\end{itemize}
and defining the non-zero Lie bracket by
\begin{equation}\label{eq17}
[\mathcal{L}_1,\mathcal{L}_1]=\mathcal{L}_2,\quad [\mathcal{L}_1,\mathcal{L}_2]=\mathcal{L}_n, \quad [\mathcal{L}_j,\mathcal{L}_j]= \mathcal{L}_n =[\mathcal{L}_1,\mathcal{L}_j], \quad 3\leq j\leq n-1,
\end{equation}
then one verifies that $\mathcal{L} = \mathfrak{sl}_2 \ltimes \Big( \bigoplus\limits_{i=1}^{n} \mathcal{L}_i \Big)$ 
is a non-semisimple sympathetic Lie algebra, thereby generalizing the Lie-Angelopoulos construction to an arbitrary number of irreducible $\mathfrak{sl}_2$-modules.
We call the resulting sympathetic Lie algebras  $\mathcal L = \mathfrak{sl}_2\ltimes \Bigl(\bigoplus\limits_{i=1}^{n}\mathcal L_i\Bigr)$ \emph{generalized Lie--Angelopoulos algebras}.

The following example gives two families of \(\mathfrak{sl}_2\)-modules for which conditions (i)-(ii) and the bracket relations~\eqref{eq17} can be realized. We specify
each family by its sequence of highest weights
\((\lambda_1,\ldots,\lambda_n)\), where \(\mathcal L_i=V_{\lambda_i}\).

\begin{example}\label{exam4.2}
Let $m=4k$ with $k\in\mathbb{N}$, and let $r_j$ ($3\le j\le n-1$) be pairwise distinct integers. We consider the following two choices of the sequence $(\lambda_1,\lambda_2,\dots,\lambda_n)$:

\begin{itemize}
\item[(a)] \quad \quad \quad  $\left( m, \, 2m - 2, \, 2m + 2r_3, \, \dots, \, 2m + 2r_{n-1}, \, 3m - 2 \right),$ \\[3mm]
where $0 \le r_j \le m-1$ and $r_j \ne \frac{m}{2} - 1$;\\

\item[(b)]  \quad \quad \quad $\left( m + 2, \, 2m + 2, \, 2m + 2r_3, \, \dots, \, 2m + 2r_{n-1}, \, 3m + 2 \right),$\\[3mm]
where $0 \le r_j \le m+2$, $r_j \ne 1$, and $r_j \ne \frac{m}{2} + 1$.
\end{itemize} 
In both cases, one can verifies that conditions (i)--(ii) hold true.
\end{example}

In the minimal case of Example~\ref{exam4.2} \textup{(a)}, taking \(n=4\), \(m=4\), and \(r_3=0\) gives the highest weights
\((4,6,8,10)\). With this choice, the bracket relations~\eqref{eq17} recover the multiplication table of the \(35\)-dimensional sympathetic Lie algebra constructed by Angelopoulos~\cite{An-1988}.

The second adjoint cohomology of Angelopoulos's \(35\)-dimensional
algebra vanishes~\cite{ABBBP-1992}. Within the same family,
Example~\ref{exam4.2}\textup{(a)} with \(n=5\), \(m=4\), \(r_3=0\),
and \(r_4=2\) gives a \(48\)-dimensional algebra \(\mathcal L\) with
highest weights \((4,6,8,12,10)\) in its nilradical. For this algebra,
a symbolic computation in \textit{Mathematica} using the method of
Section~\ref{sec3} gives $\dim H^2(\mathcal L,\mathcal L)=2.$

For higher values of $n$, the exact computation of the dimension of the second adjoint cohomology group becomes difficult; therefore, we restrict ourselves to establishing its nonvanishing.

\begin{theorem}\label{thm4.3}
Let $\mathcal{L}$ be an algebra satisfying either condition~(a) or~(b) of Example~\ref{exam4.2} with $n \geq 5$. Then 
$H^2(\mathcal{L}, \mathcal{L}) \neq 0.$
\end{theorem} 
\begin{proof}
The assumption $n \geq 5$ guarantees the existence of $\mathfrak{sl}_2$-modules 
\(\mathcal{L}_3 = V_{2m+2r_3}\) and \(\mathcal{L}_4 = V_{2m+2r_4}\). According to \eqref{eq15}, the exterior square \(\wedge^2 \mathcal{N}\) contains a direct summand \(\mathcal{L}_3 \otimes \mathcal{L}_4\). Moreover, the decomposition of \(\mathcal{L}_3 \otimes \mathcal{L}_4\) contains a term isomorphic to \(\mathcal{L}_n = V_{3m \pm 2}\). Consequently, there exists a non-zero \(\mathfrak{sl}_2\)-module homomorphism 
\(\varphi\colon\mathcal{L}_3 \otimes \mathcal{L}_4\to \mathcal{L}_n\). We define a bilinear map \(\Psi \in C^2(\mathcal{L}, \mathcal{L})\) by
$\Psi_{\mid \mathcal{L}_3 \otimes \mathcal{L}_4} = \varphi,$ and zero elsewhere.

Since \(\mathcal{L}_n \subseteq \operatorname{Center}(\mathcal{N})\) and \(\operatorname{Im}(\Psi) = \mathcal{L}_n\), to verify that \(\Psi \in Z^2(\mathcal{N}, \mathcal{L})^{\mathfrak{sl}_2}\) it suffices to check the identity
\[
\Psi([x,y], z) + \Psi([y,z], x) + \Psi([z,x], y) = 0.
\]

Observe that \(\Psi\) is non-zero only on the components \(\mathcal{L}_3\) and \(\mathcal{L}_4\), and by \eqref{eq17} the derived ideal \([\mathcal{N}, \mathcal{N}]\) contains no elements of \(\mathcal{L}_3\) or \(\mathcal{L}_4\). Therefore, $\Psi([\mathcal{N}, \mathcal{N}], \mathcal{N}) = 0,$
which shows that \(\Psi \in Z^2(\mathcal{N}, \mathcal{L})^{\mathfrak{sl}_2}\).

Assume, for the sake of contradiction, that \(\Psi \in B^2(\mathcal{N}, \mathcal{L})^{\mathfrak{sl}_2}\), i.e., \(\Psi = d_1 \Phi\) for some \(\mathfrak{sl}_2\)-invariant \(\Phi \in C^1(\mathcal{N}, \mathcal{L})\). Note that \(\mathfrak{sl}_2 \cong V_2\) and that the modules \(\mathcal{L}_i\) (\(1 \leq i \leq n\)) are pairwise non-isomorphic. By Schur's lemma, it follows that
\[
\Phi_{\mid \mathcal{L}_3} = a\, \mathrm{id}_{\mid \mathcal{L}_3}, \qquad \Phi_{\mid \mathcal{L}_4} = b\, \mathrm{id}_{\mid \mathcal{L}_4}, \quad a,b \in \mathbb{C}.
\]
However, by \eqref{eq17} we have \([\mathcal{L}_3, \mathcal{L}_4] = 0\). Therefore, \(\Psi = d_1 \Phi\) implies \(\Psi = 0\), a contradiction. Hence, $0 \neq \Psi \in H^2(\mathcal{L}, \mathcal{L}).$ 
\end{proof}

Benayadi's \(25\)-dimensional algebra~\cite{bena96} was the smallest
non-semisimple sympathetic Lie algebra known at the time of \cite{Pul25}. Garc\'ia-Pulido and Salgado proved that \(25\) is the
least possible dimension for a complex non-semisimple sympathetic Lie algebra with simple Levi factor whose nilradical is the direct sum of exactly four irreducible modules~\cite[Theorem~3.32]{Pul25}. We complete this subsection by constructing a \(17\)-dimensional example whose nilradical
has eight irreducible summands, counted with multiplicity.

Consider the \(\mathfrak{sl}_2\)-module
\[
\mathcal N
=
V_1^{(1)}\oplus V_1^{(2)}\oplus V_1^{(3)}\oplus V_1^{(4)}
\oplus V_2\oplus\mathbb C r\oplus\mathbb C s\oplus\mathbb C t,
\]
where the last three summands are trivial. Fix identifications of the four copies of \(V_1\) with the standard module, and write \(u^{(i)}\) for the element of the \(i\)-th copy corresponding to \(u\in V_1\). For $x,y\in V_1^{(1)}$, $a,b\in V_1^{(2)}$,
$z\in V_1^{(3)}$, and $q\in V_2$, define the nonzero brackets by
\begin{equation}
\label{eq:sympathetic-brackets}
\begin{alignedat}{4}
 [x,y]&=P_1(x,y)t,
 &\quad [x,a]&=P_1(x,a)s,
 &\quad [a,b]&=P_1(a,b)r,
 &\quad [x,r]&=x^{(4)},\\
 [x,t]&=x^{(4)},
 &\quad [a,r]&=a^{(3)},
 &\quad [a,s]&=a^{(4)},
 &\quad [a,z]&=P_0(a,z),\\
 [a,q]&=-P_1(a,q)^{(4)},
 &\quad [r,z]&=-3z^{(4)}.
\end{alignedat}
\end{equation}

\begin{example}\label{thm:sympathetic-example}
The brackets in~\eqref{eq:sympathetic-brackets} give rise to a
\(17\)-dimensional non-semisimple sympathetic Lie algebra $\mathcal L_{17}=\mathfrak{sl}_2\ltimes\mathcal N.$

Indeed, these brackets are \(\mathfrak{sl}_2\)-equivariant, and the
Jacobi identity on \(\mathcal N\) follows from the symmetry of \(P_0\),
the Leibniz rule for \(P_1\), and the identity
$$P_1(u,v)w+P_1(v,w)u+P_1(w,u)v=0,
\qquad u,v,w\in V_1.$$
The nilradical \(\mathcal N\) has nilpotency class five, with
\(\mathcal N^5=V_1^{(4)}\) and \(\mathcal N^6=0\). The algebra
\(\mathcal L_{17}\) is perfect since each nontrivial irreducible
summand of \(\mathcal N\) lies in \([\mathfrak{sl}_2,\mathcal N]\)
and \(r,s,t\) lie in \([\mathcal N,\mathcal N]\). Its center is
trivial: a central element belongs to the invariant subspace
\(\mathbb Cr\oplus\mathbb Cs\oplus\mathbb Ct\), and its brackets
with \(V_1^{(1)}\) and \(V_1^{(2)}\) force all three coefficients
to vanish.

By 
Schur's lemma, the derivation identity applied to
\eqref{eq:sympathetic-brackets} and \([a,t]=0\) show that every such
derivation \(\delta\) acts on the two generating summands by
$$\delta(x)=\eta x^{(4)},\qquad
\delta(a)=\psi a^{(3)}+\theta a^{(4)},
\qquad x\in V_1^{(1)},\ a\in V_1^{(2)}$$
for some \(\eta,\psi,\theta\in\mathbb C\). This action agrees on
both summands with that of
\(\operatorname{ad}(-\psi r-\theta s+(\psi-\eta)t)\), and hence
on all of \(\mathcal N\). Since \(r,s,t\) are fixed by
\(\mathfrak{sl}_2\), it follows that every derivation of
\(\mathcal L_{17}\) is inner.
\end{example}

\subsection{Complete non-perfect Lie algebras with nilpotent radical}

We answer a question of Carles~\cite[Theorem~3.4]{Car-84}
by constructing a family of complete non-perfect Lie
algebras with Levi factor $\mathfrak{sl}_2$ and
three-step nilpotent radical.
The smallest member of this family has dimension $19$.

Let $n\geq6$, and let $m_7,\ldots,m_n\geq6$ be
pairwise distinct even integers. Consider the \(\mathfrak{sl}_2\)-module
$$\mathfrak n_{\mathsf A} = \mathbb C u\oplus\mathbb C r \oplus V_2^{(1)}\oplus V_2^{(2)}\oplus V_4\oplus V_2^{(3)}
\oplus\bigoplus_{j=7}^{n}V_{m_j},$$
where \(\mathbb C u\) and \(\mathbb C r\) are trivial modules and the
last direct sum is understood to be zero when \(n=6\). Fix an
identification of each \(V_2^{(i)}\) with \(V_2\), and write \(w^{(i)}\)
for the element in the \(i\)-th copy corresponding to \(w\in V_2\).

For \(x,y\in V_2^{(1)}\), \(z\in V_2^{(2)}\) and \(a,b\in V_4\), define nonzero brackets on \(\mathfrak n_{\mathsf A}\) by
$$\begin{alignedat}{3}
[u,x]&=x^{(3)},&\qquad [r,x]&=x^{(2)},
&\qquad [r,z]&=2z^{(3)},\\
[x,y]&=P_1(x,y)^{(2)},& [x,z]&=P_1(x,z)^{(3)},
&[x,a]&=P_2(x,a)^{(3)},\\
[a,b]&=P_3(a,b)^{(3)}.
\end{alignedat}$$
For each \(7\leq j\leq n\) and $v,w\in V_{m_j}$, also set $[v,w]=P_{m_j-1}(v,w)^{(3)}.$ 

\begin{theorem}\label{thm:classA-example}
The algebra \(\mathcal L=\mathfrak{sl}_2\ltimes\mathfrak n_{\mathsf A}\) is a complete non-perfect Lie algebra of dimension 
$\dim\mathcal L=19+\sum\limits_{j=7}^{n}(m_j+1).$ For \(n=6\), we denote the resulting \(19\)-dimensional algebra
by \(\mathcal L_{19}\).
\end{theorem}

\begin{proof}
Since each \(m_j\) is even, the transvectant \(P_{m_j-1}\) is
alternating and defines a nonzero equivariant map
\(\bigwedge^2 V_{m_j}\to V_2\). All the defining brackets are
\(\mathfrak{sl}_2\)-equivariant, so it suffices to verify the Jacobi
identity on \(\mathfrak n_{\mathsf A}\). On three elements of
\(V_2^{(1)}\), it follows from the Jacobi identity for \(P_1\).
For \(x,y\in V_2^{(1)}\), the relation
\[
[r,[x,y]]=[[r,x],y]+[x,[r,y]]
\]
holds because both sides equal \(2P_1(x,y)^{(3)}\). All remaining
Jacobi terms vanish. In particular, each \(V_{m_j}\) commutes with
the other summands, and its self-bracket takes values in the central
subspace \(V_2^{(3)}\) of \(\mathfrak n_{\mathsf A}\).

The defining relations give
\[
\mathfrak n_{\mathsf A}^{2}=V_2^{(2)}\oplus V_2^{(3)},\qquad
\mathfrak n_{\mathsf A}^{3}=V_2^{(3)},\qquad
\mathfrak n_{\mathsf A}^{4}=0.
\]
A central element of \(\mathcal L\) is \(\mathfrak{sl}_2\)-invariant
and hence has the form \(su+tr\). Since
$[su+tr,x]=s x^{(3)}+t x^{(2)},$ $x\in V_2^{(1)}$, centrality forces \(s=t=0\). The images of \(u\) and \(r\) form a
basis of \(\mathcal L/[\mathcal L,\mathcal L]\), so \(\mathcal L\)
is centerless and non-perfect.

By Schur's lemma, we have 
$D(x)=\alpha x+\beta x^{(2)}+\gamma x^{(3)}$ for $x\in V_2^{(1)}.$ Since \([r,x]=x^{(2)}\) and \([u,x]=x^{(3)}\), the derivation
$D_0=D - {\operatorname{ad}(\beta r+\gamma u)}_{| \mathfrak n_{\mathsf A}}$ acts on \(V_2^{(1)}\) by the scalar \(\alpha\) and it is \(\mathfrak{sl}_2\)-equivariant because of  \(u\) and \(r\) are invariant.

The brackets \([V_2^{(1)},V_2^{(1)}]=V_2^{(2)}\) and
\([V_2^{(1)},V_2^{(2)}]=V_2^{(3)}\) show successively that \(D_0\)
acts on \(V_2^{(2)}\) and \(V_2^{(3)}\) by \(2\alpha\) and
\(3\alpha\), respectively. Applying \(D_0\) to \([u,x]=x^{(3)}\)
and \([r,x]=x^{(2)}\) then gives $D_0(u)=2\alpha u$ and $D_0(r)=\alpha r.$ 
The module \(V_4\) occurs with multiplicity one, so \(D_0\) acts on
it by a scalar \(\lambda\). Since both \([V_2^{(1)},V_4]\) and
\([V_4,V_4]\) equal \(V_2^{(3)}\), the derivation identity gives $3\alpha=\alpha+\lambda=2\lambda.$ 
Therefore, \(\alpha=\lambda=0\).

Finally, each \(V_{m_j}\) also occurs with multiplicity one. If
\(D_0\) acts on \(V_{m_j}\) by \(\mu_j\), then
$D_0([v,w])=2\mu_j[v,w]=0$ for $v,w\in V_{m_j},$ because \(D_0\) vanishes on \(V_2^{(3)}\). Since the bracket on
\(V_{m_j}\) is nonzero, so \(\mu_j=0\). Hence \(D_0=0\), and
\(D\) is the restriction of \(\operatorname{ad}(\beta r+\gamma u)\).
This proves completeness. The dimension formula follows from the
module decomposition.
\end{proof}

By Theorem~\ref{thm:classA-example}, every Lie algebra in this family is complete, so its first adjoint cohomology group vanishes. The following proposition shows that its second adjoint cohomology group is nonzero.

\begin{proposition}
\label{prop:classA-H2}
Every Lie algebra $\mathcal L$ in the family constructed in Theorem \ref{thm:classA-example} satisfies $H^2(\mathcal L,\mathcal L)\neq0$.
\end{proposition}

\begin{proof}

Define a linear map
$\Psi:\bigwedge^2\mathcal L\to\mathcal L$,
with $\Psi(a,b)=\Psi(a\wedge b)$, by
\[
\begin{alignedat}{3}
\Psi(u,r)&=u, &
\qquad \Psi(r,k)&=-2k, &
\qquad \Psi(r,v)&=-v,\\
\Psi(x,z)&=P_1(x,z)^{(2)}, &
\qquad \Psi(z,w)&=P_1(z,w)^{(3)},
\end{alignedat}
\]
where $x\in V_2^{(1)}$, $z,w\in V_2^{(2)}$,
$k\in V_2^{(3)}$, and
$v\in V_2^{(1)}\oplus V_4
\oplus\bigoplus\limits_{j=7}^n V_{m_j}$ and remaining values are zero.

We first verify that $d\Psi=0$.
Since $\Psi$ is an $\mathfrak{sl}_2$-module homomorphism
and $\Psi(\mathfrak{sl}_2,\mathcal L)=0$, we have
$(d\Psi)(s,x,y)=0$ for $s\in\mathfrak{sl}_2$
and $x,y\in\mathcal L$.
For $x,y\in V_2^{(1)}$ and $z\in V_2^{(2)}$,
the defining brackets give
$$(d\Psi)(u,r,x)=(d\Psi)(r,x,y)=(d\Psi)(r,x,z)=0.$$
The remaining triples containing \(r\) for which individual terms
can be nonzero are, up to permutation, of the form $\{r,a,b\}$, with
\((a,b)\) in \(V_2^{(1)}\times V_4\), \(V_4\times V_4\), or
\(V_{m_j}\times V_{m_j}\). In each case \([a,b]\in V_2^{(3)}\), and $(d\Psi)(r,a,b)=0.$
It remains to consider triples chosen from the summands
of $\mathfrak n_{\mathsf A}$ other than $\mathbb Cr$.
Up to permutation, all terms in $d\Psi$ vanish unless
$x,y,z\in V_2^{(1)}$ or
$x,y\in V_2^{(1)}$ and $z\in V_2^{(2)}$.
Under the fixed identifications with $V_2$,
both cases reduce to the Jacobi identity for $P_1$,
\[
P_1(x,P_1(y,z))-P_1(y,P_1(x,z))
-P_1(P_1(x,y),z)=0.
\]
Thus $d\Psi=0$.

Suppose that $\Psi=d\varphi$ for some linear map
$\varphi:\mathcal L\to\mathcal L$.
Since $[u,r]=0$, we have
$(d\varphi)(u,r)=[u,\varphi(r)]+[\varphi(u),r]
\in \mathfrak n_{\mathsf A}^2$,
contrary to $\Psi(u,r)=u\notin\mathfrak n_{\mathsf A}^2$.
Thus $0\neq[\Psi]\in H^2(\mathcal L,\mathcal L)$.
\end{proof}

The construction above gives a family of complete non-perfect Lie algebras based on a $19$-dimensional example. We now present three further examples of complete non-perfect Lie algebras.

\begin{example}\label{exam4.4}
Let $\mathcal L=\mathfrak{sl}_2\ltimes\mathcal N$, where
$\mathcal N=\bigoplus\limits_{i=1}^6\mathcal L_i$
and $\mathcal L_i\cong V_{\lambda_i}$.
Consider the following three choices of the sequence
$(\lambda_1,\ldots,\lambda_6)$ and the corresponding
nonzero brackets on $\mathcal N$, realized by transvectants.

\begin{itemize}
\item
$(0,6,6,8,10,14)$, with nonzero brackets
\begin{equation}\label{eq19}
\begin{array}{lll}
[\mathcal L_1,\mathcal L_2]=\mathcal L_3, &
[\mathcal L_2,\mathcal L_2]=\mathcal L_5, &
[\mathcal L_2,\mathcal L_4]=\mathcal L_6, \\[3pt]
[\mathcal L_2,\mathcal L_5]=\mathcal L_6, &
[\mathcal L_4,\mathcal L_4]=\mathcal L_6\oplus\mathcal L_3.
\end{array}
\end{equation}

\item
$(0,6,6,4,6,2)$, with the same nonzero bracket
relations as in~\eqref{eq19}.

\item
$(0,6,4,6,2,2)$, with nonzero brackets
\[
\begin{array}{lll}
[\mathcal L_2,\mathcal L_2]=\mathcal L_4, &
[\mathcal L_3,\mathcal L_3]=\mathcal L_5, &
[\mathcal L_1,\mathcal L_6]=\mathcal L_5, \\[3pt]
[\mathcal L_2,\mathcal L_4]=\mathcal L_5, &
[\mathcal L_2,\mathcal L_3]=\mathcal L_5, &
[\mathcal L_3,\mathcal L_6]=\mathcal L_5.
\end{array}
\]
\end{itemize}
One can show that each of these three algebras is complete and non-perfect.
\end{example}

Note that as in Theorem~\ref{thm:classA-example}, each of the above examples can be extended so that its nilradical is a direct sum of any prescribed number of irreducible $\mathfrak{sl}_2$-modules involved in $\mathcal{N}$.

\subsection{Perfect Lie algebras with nonzero center}\label{subsection4.3}

We now construct a family of perfect Lie algebras with
one-dimensional center and only inner derivations.
The construction begins with a $20$-dimensional example
and extends it by adjoining irreducible
$\mathfrak{sl}_2$-modules.

Let
\[
\mathfrak n_{\mathsf B}
=V_1^{(1)}\oplus V_1^{(2)}
\oplus V_4\oplus V_3\oplus V_2\oplus\mathbb Ct.
\]
Fix an $\mathfrak{sl}_2$-module isomorphism between
the two copies of $V_1$.
For $x,y\in V_1^{(1)}$, $a,b\in V_1^{(2)}$,
$p,q\in V_4$, and $z\in V_3$, define
\begin{equation}
\label{eq:classB-brackets}
\begin{alignedat}{3}
[x,p]&=P_1(x,p), &
\qquad [x,z]&=P_1(x,z), &
\qquad [x,y]&=P_1(x,y)t,\\
[x,a]&=P_0(x,a)+2P_1(x,a)t, &
\qquad [a,b]&=P_1(a,b)t, &
\qquad [p,q]&=P_3(p,q).
\end{alignedat}
\end{equation}
The values of $P_1(x,p)$, $P_1(x,z)$,
$P_0(x,a)$, and $P_3(p,q)$ belong to
$V_3$, $V_2$, $V_2$, and $V_2$, respectively.

Let $n\geq6$, and let $m_7,\ldots,m_n\geq6$
be pairwise distinct even integers. Set
\[
\mathcal N=\mathfrak n_{\mathsf B}
\oplus\bigoplus\limits_{j=7}^n V_{m_j}.
\]
For $n=6$, this gives $\mathcal N=\mathfrak n_{\mathsf B}$.
Define the nonzero brackets on $\mathcal N$ by
\eqref{eq:classB-brackets} and the additional relations
$[v,w]=P_{m_j-1}(v,w)\in V_2$
for $v,w\in V_{m_j}$ and $7\leq j\leq n$.
\begin{theorem}\label{thm:classB-example}
The Lie algebra \(\mathcal L=\mathfrak{sl}_2\ltimes\mathcal N\) is
perfect, has center \(\mathbb C t\), and admits only inner derivations.
Its dimension is 
$\dim\mathcal L=20+\sum\limits_{j=7}^{n}(m_j+1).$ For $n = 6$, we denote the resulting $20$-dimensional algebra by $\mathcal L_{20}$.
\end{theorem}

\begin{proof}
The proof follows the same method as that of
Theorem~\ref{thm:classA-example}.
One easily verifies that the defining brackets satisfy
the Jacobi identity, that $\mathcal L$ is perfect,
and $Z(\mathcal L)=\mathbb Ct$.

Let $D\in\operatorname{Der}(\mathcal L)$.
Then there exists
$y\in\mathcal L$ such that $D=d+\operatorname{ad}(y)$,
where $d$ vanishes on $\mathfrak{sl}_2$ and its restriction
to $\mathcal N$ is a derivation and an
$\mathfrak{sl}_2$-module homomorphism.
Applying Schur's lemma and the derivation identity
to \eqref{eq:classB-brackets} shows that $d$ vanishes
on $\mathfrak n_{\mathsf B}$.
Each $V_{m_j}$ occurs with multiplicity one, so $d$
acts on it by a scalar $\mu_j$.
For $v,w\in V_{m_j}$, the additional bracket gives
$d([v,w])=2\mu_j[v,w]=0$.
Since this bracket is nonzero, $\mu_j=0$.
Thus $d=0$, and consequently $D=\operatorname{ad}(y)$.
The dimension formula follows from the module decomposition.
\end{proof}

We now examine the second adjoint cohomology of the family constructed in Theorem \ref{thm:classB-example}. Although this cohomology vanishes for $\mathcal L_{20}$, the following proposition shows that its dimension can be arbitrarily large within the family.

\begin{proposition}\label{prop:classB-H2}
Let $\mathcal L$ belong to the family in
Theorem~\ref{thm:classB-example}, and set
$I=\{4,m_7,\ldots,m_n\}$. Then
\[
H^2(\mathcal L_{20},\mathcal L_{20})=0,
\qquad
\dim H^2(\mathcal L,\mathcal L)
\geq\#\{m\in I\mid m+2\in I\}.
\]
In particular, if $m_j=2j-8$ for $7\leq j\leq n$,
then $\dim H^2(\mathcal L,\mathcal L)\geq n-6$.
\end{proposition}

\begin{proof}
We first establish the lower bound.
For each $m\in I$ with $m+2\in I$, define a nonzero linear map
$\Psi_m:\bigwedge^2\mathcal L\to\mathcal L$ by
\[
\Psi_m(v,w)=P_m(v,w),
\qquad v\in V_m,\; w\in V_{m+2},
\]
with all remaining values zero.
Thus $d\Psi_m$ vanishes on $\bigwedge^3\mathcal N$.
Since $\Psi_m$ is an $\mathfrak{sl}_2$-module homomorphism
and $\Psi_m(\mathfrak{sl}_2,\mathcal L)=0$,
we also have $(d\Psi_m)(\mathfrak{sl}_2,\mathcal L,\mathcal L)=0$.
Hence $\Psi_m$ is a $2$-cocycle on $\mathcal L$.

We now prove that the classes $[\Psi_m]$ are linearly independent.
Suppose that $\sum\limits_m c_m\Psi_m=d\phi$.
By complete reducibility, we may assume that $\phi$
is an $\mathfrak{sl}_2$-module homomorphism.
Each $V_m$ with $m\in I$ occurs with multiplicity one
in $\mathcal L$, so Schur's lemma implies that $\phi$
acts by a scalar on each such summand.
Since $[V_m,V_{m+2}]=0$, we have
$(d\phi)(V_m,V_{m+2})=0$.
For $v\in V_m$ and $w\in V_{m+2}$, the identity gives
$c_mP_m(v,w)=(d\phi)(v,w)=0$.
Since $P_m\neq0$, we obtain $c_m=0$.
Thus the classes $[\Psi_m]$ are linearly independent,
which proves the lower bound.

It remains to show $H^2(\mathcal L_{20},\mathcal L_{20})=0$.
By Whitehead's lemmas, every class in
$H^2(\mathcal L_{20},\mathcal L_{20})$
has a representative $\Psi$ that is an
$\mathfrak{sl}_2$-module homomorphism and satisfies
$\Psi(\mathfrak{sl}_2,\mathcal L_{20})=0$.
Fix an $\mathfrak{sl}_2$-module isomorphism
$\iota:V_2\to\mathfrak{sl}_2$ and extend it by zero
on the other summands of $\mathcal L_{20}$.
Its coboundary is given by
$(d\iota)(x,y)=[x,\iota(y)]+[\iota(x),y]-\iota([x,y])$
for $x,y\in\mathcal L$.
The cocycle $\Psi$ has the form
$\Psi=\Psi_0+\lambda d\iota$, where $\lambda\in\mathbb C$
and $\Psi_0$ is given by
\[
\begin{array}{lll}
\Psi_0(x,y)=c_1P_1(x,y)t, &
\Psi_0(a,b)=c_2P_1(a,b)t, &
\Psi_0(x,p)=c_3P_1(x,p),\\[3pt]
\Psi_0(a,p)=c_4P_1(a,p), &
\Psi_0(a,z)=c_4P_1(a,z), &
\Psi_0(p,q)=c_5P_3(p,q),\\[3pt]
\multicolumn{2}{l}{\Psi_0(x,a)=c_6P_0(x,a)+c_7P_1(x,a)t,} &
\Psi_0(x,z)=c_8P_1(x,z),
\end{array}
\]
with $c_i\in\mathbb C$,
$x,y\in V_1^{(1)}$, $a,b\in V_1^{(2)}$,
$p,q\in V_4$, and $z\in V_3$, and all remaining values of $\Psi_0$ are zero. The number of parameters $\lambda,c_1,\ldots,c_8$ therefore gives $\dim Z^2(\mathfrak n_{\mathsf B}, \mathcal L_{20})^{\mathfrak{sl}_2}\leq9$.

On the other hand, Schur's lemma leads
$\dim C^1(\mathfrak n_{\mathsf B},
\mathcal L_{20})^{\mathfrak{sl}_2}=9.$
Every $1$-cocycle in the space $C^1(\mathfrak n_{\mathsf B},
\mathcal L_{20})^{\mathfrak{sl}_2}$ extends by zero on $\mathfrak{sl}_2$ to a derivation of
$\mathcal L_{20}$. Hence
$Z^1(\mathfrak n_{\mathsf B},
\mathcal L_{20})^{\mathfrak{sl}_2}=0$, and so
\[
\dim B^2(\mathfrak n_{\mathsf B},
\mathcal L_{20})^{\mathfrak{sl}_2}
=
\dim C^1(\mathfrak n_{\mathsf B},
\mathcal L_{20})^{\mathfrak{sl}_2}
=9.
\]
It follows that every $\mathfrak{sl}_2$-module $2$-cocycle is a coboundary.
Therefore $H^2(\mathcal L_{20},\mathcal L_{20})=0$.

For $m_j=2j-8$, the set $I$ is
$\{4,6,\ldots,2n-8\}$ and contains $n-6$ adjacent
pairs differing by two.
The final assertion follows.
\end{proof}

\section{Application to Kourovka Problem 16.63}\label{sec5}

In this section we apply the algebras constructed in subsection \ref{subsection4.3} to the study of automorphisms of finite $p$-groups. Our approach combines the perfectness
of $\mathcal L_{20}$, its one-dimensional center, and the innerness of
its derivations with the Baker--Campbell--Hausdorff construction for powerful $\mathbb Z_p$-Lie lattices. A suitable central extension then allows us to control the resulting group automorphisms and obtain groups whose orders coincide with those
of their automorphism groups. More precisely, for every prime $p\geq 7$ we construct a finite $p$-group $G_p$ satisfying
$$|G_p|=|\operatorname{Aut}(G_p)|=p^{7621}.$$
This result, in particular, gives an affirmative answer to Kourovka Notebook Problem~16.63. \cite{Kourovka}

 Throughout this section, let $p\geq7$ be a prime. We refer to \cite{DDMS} for background on pro-$p$ groups, $\mathbb Z_p$-Lie lattices, and the Baker--Campbell--Hausdorff correspondence.

\subsection{The reduced Lie algebra modulo  \texorpdfstring{$p$}{p}} \label{subsec:Kourovka-Lie} 

Let $\mathbb Z_p$ denote the ring of $p$-adic integers and $\mathbb F_p=\mathbb Z_p/p\mathbb Z_p.$  The structure constants of $\mathcal L_{20}$ in the chosen basis
define a Lie algebra structure on a free $\mathbb Z_p$-module $\mathfrak B$ of rank $20$. Let $\{b_1,\ldots,b_{20}\}$ denote the
corresponding basis of $\mathfrak B$.  Since $\mathfrak B$ is free of rank $20$ over $\mathbb Z_p$, the quotient $\mathfrak b=\mathfrak B/p\mathfrak B$ is a
$20$-dimensional Lie algebra over $\mathbb F_p$.

Let $\mathfrak s=\mathfrak{sl}_2(\mathbb F_p)$, and let $\mathfrak n$ be the Lie algebra obtained by reducing the structure
constants of $\mathfrak n_{\mathsf B}$ modulo $p$. Then $\mathfrak b=\mathfrak s\ltimes\mathfrak n$ and for $0\leq m\leq4$, the $\mathfrak s$-module
$V_m=\mathbb F_p[X,Y]_m$ of homogeneous polynomials of degree $m$ is irreducible whenever $p\geq7$.

The trace form $\langle x,y\rangle=\operatorname{tr}(xy)$
on $\mathfrak s$ is symmetric, invariant and  nondegenerate. The corresponding Casimir element in the universal enveloping algebra $U(\mathfrak s)$ is
\begin{equation}\label{eq:Kourovka-casimir}
\Omega=\frac12h^2+ef+fe.
\end{equation}

\begin{lemma}\label{lem:Kourovka-module-properties}
For $0\leq m\leq4$, we have
$\operatorname{End}_{\mathfrak s}(V_m)
=\mathbb F_p\operatorname{id}_{V_m}$
and $H^1(\mathfrak s,V_m)=0$.
\end{lemma}

\begin{proof}
For $m=0$, the module $V_0$ is one-dimensional and trivial, so
$\operatorname{End}_{\mathfrak s}(V_0)
=\mathbb F_p\operatorname{id}_{V_0}$.
Moreover, every $1$-cocycle $\varphi:\mathfrak s\to V_0$
satisfies $\varphi([\mathfrak s,\mathfrak s])=0$.
Since $\mathfrak s$ is perfect, $\varphi=0$, and hence
$H^1(\mathfrak s,V_0)=0$.

Assume that $1\leq m\leq4$. Since $m<p$, it follows that
$\ker(e_{| V_m})=\mathbb F_pX^m$.
Thus any $T\in\operatorname{End}_{\mathfrak s}(V_m)$
satisfies $T(X^m)=\lambda X^m$ for some
$\lambda\in\mathbb F_p$. As $T$ commutes with $f$, we obtain
$$T(f^iX^m)=f^iT(X^m)=\lambda f^iX^m, 
\qquad 0\leq i\leq m.$$
The vectors
\[
f^iX^m=\frac{m!}{(m-i)!}X^{m-i}Y^i,
\qquad 0\leq i\leq m,
\]
form a basis of $V_m$, since their coefficients are nonzero
in $\mathbb F_p$. Consequently,
$T=\lambda\operatorname{id}_{V_m}$.
The reverse inclusion is immediate, proving
$\operatorname{End}_{\mathfrak s}(V_m)
=\mathbb F_p\operatorname{id}_{V_m}$.

To establish the cohomology assertion, consider the Casimir element
$\Omega$ given in \eqref{eq:Kourovka-casimir}.
A direct calculation on the monomial basis gives
\[
\Omega(X^{m-i}Y^i)
=\lambda_mX^{m-i}Y^i,
\qquad
\lambda_m=\frac{m(m+2)}2,
\quad 0\leq i\leq m.
\]
Thus $\Omega$ acts on $V_m$ as
$\lambda_m\operatorname{id}_{V_m}$.
Since $1\leq m\leq4$ and $p\geq7$, we have
$\lambda_m\neq0$ in $\mathbb F_p$.

Let $\varphi:\mathfrak s\to V_m$ be a $1$-cocycle.
Applying the cocycle identity to the pairs $\{h, e\}$, $\{h,f\}$ and $\{e,f\}$ yields
$$e\varphi(h)=(h-2)\varphi(e),
\qquad f\varphi(h)=(h+2)\varphi(f),
\qquad \varphi(h)=e\varphi(f)-f\varphi(e).$$
Setting $v=\frac12h\varphi(h)+e\varphi(f)+f\varphi(e)$
and using these equalities and the defining relations of
$\mathfrak s$, we obtain
$xv=\Omega\varphi(x)=\lambda_m\varphi(x)$
for every $x\in\mathfrak s$. Therefore, we obtain $\varphi(x)=x\cdot(\lambda_m^{-1}v)$, $x\in\mathfrak s,$ 
so $\varphi$ is a coboundary. This proves
$H^1(\mathfrak s,V_m)=0$.
\end{proof}

We denote by $\overline t$ the image, under the natural projection $\mathfrak B\to\mathfrak b$, of the basis element of $\mathfrak B$ corresponding to the central basis element $t$ of $\mathcal L_{20}$.

\begin{proposition}\label{lem:Kourovka-structure}
The Lie algebra $\mathfrak b$ is perfect, has one-dimensional
center $Z(\mathfrak b)=\mathbb F_p\overline t$, and admits only
inner derivations.
\end{proposition}

\begin{proof}
For $1\leq m\leq4$, the $\mathfrak s$-module $V_m$ is
nontrivial and irreducible, so $[\mathfrak s,V_m]=V_m$.
Together with $[\mathfrak s,\mathfrak s]=\mathfrak s$ and
$[V_1^{(1)},V_1^{(1)}]=\mathbb F_p\overline t$, this implies
$[\mathfrak b,\mathfrak b]=\mathfrak b$.
Moreover, $\mathbb F_p\overline t$ is the only trivial summand
in the $\mathfrak s$-module decomposition of $\mathfrak b$.
Since every central element is $\mathfrak s$-invariant,
$Z(\mathfrak b)\subseteq\mathbb F_p\overline t$.
The reverse inclusion follows from \eqref{eq:classB-brackets},
and hence $Z(\mathfrak b)=\mathbb F_p\overline t$.

Let $D\in\operatorname{Der}(\mathfrak b)$.
Its restriction to $\mathfrak s$ is a $1$-cocycle with values
in the adjoint $\mathfrak s$-module $\mathfrak b$.
Applying Lemma~\ref{lem:Kourovka-module-properties} to the
irreducible summands of $\mathfrak b$ gives
$H^1(\mathfrak s,\mathfrak b)=0$.
Thus there exists $z\in\mathfrak b$ such that
$D(x)=[x,z]$ for every $x\in\mathfrak s$.
Set $D_0=D+\operatorname{ad}(z)$.
Then ${D_0}_{| \mathfrak s}=0$, and the derivation identity gives 
$D_0([x,u])=[x,D_0(u)]$ for $x\in\mathfrak s,\ u\in\mathfrak b.$ Hence, $D_0$ is an $\mathfrak s$-module endomorphism.

By Lemma~\ref{lem:Kourovka-module-properties} and the $\mathfrak s$-module decomposition of $\mathfrak b$, we may write
$$D_0(x^{(1)})=\alpha x^{(1)}+\beta x^{(2)},
\qquad D_0(x^{(2)})=\gamma x^{(1)}+\delta x^{(2)}, \qquad x\in V_1,$$
and
$${D_0}_{| V_4}=\varepsilon\operatorname{id}_{V_4},
\qquad
 {D_0}_{| V_3}=\zeta\operatorname{id}_{V_3},
\qquad D_0(\overline t)=\tau\overline t,$$
with all coefficients in $\mathbb F_p$.
Since $P_3(X^4,XY^3)=144X^2\neq0$ and $V_2$ is irreducible,
we have $[V_4,V_4]=V_2$.
The derivation identity therefore yields
$D_0|_{V_2}=\eta\operatorname{id}_{V_2}$ with $\eta=2\varepsilon.$

The maps $P_0$, $P_1$ and $P_3$ occurring in
\eqref{eq:classB-brackets} remain nonzero over $\mathbb F_p$.
Applying the derivation identity to these brackets, as in the
proof of Theorem~\ref{thm:classB-example}, gives
$$\gamma=0,\quad \beta=\alpha,\quad \delta=3\alpha, \quad \varepsilon=2\alpha,\quad
\zeta=3\alpha,\quad
\eta=4\alpha,\quad
\tau=6\alpha,$$
together with $2\tau=2\alpha+2\delta+\beta$.
All divisions involved in this calculation are valid because
$p\geq7$.
Substitution into the latter equality gives $12\alpha=9\alpha$,
and hence $\alpha=0$.
Thus all coefficients vanish, so $D_0=0$.
Consequently, $D=-\operatorname{ad}(z)$ is inner.
\end{proof}

We next construct an explicit pair of generators for $\mathfrak b$. 
\begin{lemma}\label{lem:Kourovka-generators}
The Lie algebra $\mathfrak b$ is generated by two elements.
\end{lemma}

\begin{proof}
Let $\Lambda=\{2,-2,1,-1,4\}$.
Choose highest-weight vectors
$y_1\in V_1^{(1)}$ and $y_4\in V_4$, and a 
lowest-weight vector $y_{-1}\in V_1^{(2)}$.
Set $y_2=e$, $y_{-2}=f$, and $y=\sum\limits_{\lambda\in\Lambda}y_\lambda.$
Then $[h,y_\lambda]=\lambda y_\lambda$ for every
$\lambda\in\Lambda$.

Let $\mathfrak a$ be the Lie subalgebra generated by $h$ and $y$,
and put $T=\operatorname{ad}(h)$.
Since $p\geq7$, the elements of $\Lambda$ are pairwise distinct
in $\mathbb F_p$. Thus, for each $\lambda\in\Lambda$, the polynomial
\[
F_\lambda(z)
=
\prod_{\substack{\mu\in\Lambda\\\mu\neq\lambda}}
\frac{z-\mu}{\lambda-\mu}
\]
is well defined in $\mathbb F_p[z]$ and satisfies
$F_\lambda(\mu)=\delta_{\lambda,\mu}$ for every $\mu\in\Lambda$.
Since $Ty_\mu=\mu y_\mu$, we obtain
\[
F_\lambda(T)y
=\sum_{\mu\in\Lambda}F_\lambda(\mu)y_\mu
=y_\lambda.
\]
Moreover, $\mathfrak a$ is $T$-invariant and contains $y$,
so $y_\lambda\in\mathfrak a$ for every $\lambda\in\Lambda$.

In particular, $e,f\in\mathfrak a$, and hence
$\mathfrak s\subseteq\mathfrak a$.
By irreducibility, the $\mathfrak s$-submodules generated by
$y_1$, $y_{-1}$, and $y_4$ are
$V_1^{(1)}$, $V_1^{(2)}$, and $V_4$, respectively.
These summands therefore lie in $\mathfrak a$.
The bracket identities in \eqref{eq:classB-brackets} then imply
that $V_3$, $V_2$, and $\mathbb F_p\overline t$ also lie in
$\mathfrak a$. Consequently, $\mathfrak a=\mathfrak b$,
so $h$ and $y$ generate $\mathfrak b$.
\end{proof}

\subsection{Construction of the hyperplane $K$}\label{subsec:Kourovka-explicit-K}

In the notation of the proof of Lemma~\ref{lem:Kourovka-generators} $y_1$, $y_{-1}$, and $y_4$ are the monomials $X$, $Y$, and $X^4$ in $V_1^{(1)}$, $V_1^{(2)}$, and $V_4$, respectively. Define
$$P=\operatorname{span}_{\mathbb F_p}\{h,y\}.$$

\begin{lemma}
\label{lem:Kourovka-generating-plane}
Let $\tau\in\operatorname{Aut}(\mathfrak b)$
satisfy $\sigma(P)=P$.
Then $\tau=\operatorname{id}_{\mathfrak b}$.
\end{lemma}
\begin{proof} Since $\tau(P)=P$, write
\begin{equation}
\label{eq:Kourovka-plane-images}
\tau(h)=\alpha h+\beta y,
\qquad \tau(y)=\gamma h+\delta y,
\end{equation}
where $\alpha,\beta,\gamma,\delta\in\mathbb F_p$ and $\alpha\delta-\beta\gamma\neq0$.

We first show that $\beta=0$.
By \eqref{eq:classB-brackets}, $\mathfrak n$ is nilpotent and $\mathfrak n^2=V_3\oplus V_2\oplus\mathbb F_p\overline t.$ 
Since $\mathfrak b/\mathfrak n\cong\mathfrak{sl}_2(\mathbb F_p)$
is simple and nonabelian, $\mathfrak n$ is the solvable radical of $\mathfrak b$. Thus $\tau(\mathfrak n)=\mathfrak n$,
and we may write $\tau(\mathfrak s)=\{z+\varphi(z):z\in\mathfrak s\}$ for a linear map $\varphi:\mathfrak s\to\mathfrak n$.
Since $\mathfrak n^2$ has zero $V_4$-component, the subalgebra condition implies that the $V_4$-component $\varphi_4$ of $\varphi$ is a $1$-cocycle. Lemma~\ref{lem:Kourovka-module-properties} therefore gives
$\zeta\in V_4$ such that $\varphi_4(z)=z\cdot\zeta$.

Put $z_0=\alpha h+\beta(e+f)$.
Comparing $V_4$-components in  $\tau(h)=z_0+\varphi(z_0)$
yields
\begin{equation}
\label{eq:Kourovka-plane-V4}
\beta y_4=z_0\cdot\zeta.
\end{equation}
For $u_i=X^{4-i}Y^i$, $0\leq i\leq4$, we have $u_0=y_4$ and
\begin{equation}
\label{eq:Kourovka-plane-action}
z_0\cdot u_i =(4-2i)\alpha u_i+i\beta u_{i-1}+(4-i)\beta u_{i+1},
\end{equation}
where $u_{-1}=u_5=0$.
Define $\ell_4\in V_4^*$ by
\[
\bigl(\ell_4(u_0),\ldots,\ell_4(u_4)\bigr)
= \bigl(3\beta^2,-3\alpha\beta,
2\alpha^2-\beta^2,3\alpha\beta,3\beta^2\bigr).
\]
Equation~\eqref{eq:Kourovka-plane-action} gives
$\ell_4(z_0\cdot V_4)=0$.
Applying $\ell_4$ to \eqref{eq:Kourovka-plane-V4}, we obtain
\[
0=\ell_4(z_0\cdot\zeta)=\beta\ell_4(u_0)=3\beta^3.
\]
Hence $\beta=0$, since $p\geq7$.

Thus, $\tau(h)=\alpha h$ and $\alpha\delta\neq0$.
For every $\mu\in\Lambda$, we have $[h,\tau(y_\mu)]=\alpha^{-1}\mu\,\sigma(y_\mu).$ These weights are distinct and nonzero, so $\tau(y)$ has no component of weight zero. It follows from \eqref{eq:Kourovka-plane-images} that $\gamma=0$. Comparing weights in $\tau(y)=\delta y$ gives $\alpha^{-1}\Lambda=\Lambda$.
Since $\sum\limits_{\mu\in\Lambda}\mu=4$, we obtain $4\alpha^{-1}=4$, and hence $\alpha=1$.

Now $\tau(h)=h$, so $\tau$ preserves the weight spaces of $\operatorname{ad}(h)$. Comparing components in
$\tau(y)=\delta y$ therefore gives
$\tau(y_\lambda)=\delta y_\lambda$, $\lambda\in\Lambda).$ Applying $\tau$ to $[e,f]=h$ yields $\delta^2=1$.
Consequently, $\tau$ fixes both $[f,y_1]$ and
$[e,y_{-1}]$. The bracket identities
\[
[y_1,[f,y_1]]=\overline t,
\qquad [y_1,y_{-1}]-[[f,y_1],[e,y_{-1}]]=4\overline t
\]
now give $\tau(\overline t)=\delta\overline t$ and
$\tau(\overline t)=\overline t$, respectively:
the left-hand side of the second identity is fixed by
$\tau$. Since $\overline t\neq0$, we conclude that
$\delta=1$.

Thus $\tau$ fixes the generators $h$ and $y$, and therefore $\tau=\operatorname{id}_{\mathfrak b}$.
\end{proof}

We use the following ordered basis of $\mathfrak b$:
\begin{equation}
\label{eq:Kourovka-explicit-basis}
\begin{aligned}
\{b_1,\ldots,b_{20}\}
={}&\{
h,y,e,f,X^{(1)},Y^{(1)},X^{(2)},Y^{(2)},\\
&X^3Y,X^2Y^2,XY^3,Y^4,
X^3,X^2Y,XY^2,Y^3,X^2,XY,Y^2,\overline t
\}.
\end{aligned}
\end{equation}
The monomials of degrees $4$, $3$, and $2$ in this
list lie in $V_4$, $V_3$, and $V_2$, respectively.
This is a basis because $y$ replaces $X^4$ in the
original basis with coefficient one.

Define a linear
functional $\ell : \bigwedge^2\mathfrak b \to\mathbb F_p$ by
\begin{equation}
\label{eq:Kourovka-explicit-form}
\ell(b_i\wedge b_j)=
\begin{cases}
1,& (i,j)=(2r+1,2r+2)
     \text{ for some }1\leq r\leq9,\\
0,& \text{otherwise},
\end{cases}
\end{equation}
for $1\leq i<j\leq20$ and put $K=\ker\ell$. 

Recall that, for the adjoint representation of $\mathfrak b$,
the Chevalley--Eilenberg differential is given by
\[
(d_{\mathrm{CE}}D)(u,v)
=
[u,D(v)]+[D(u),v]-D([u,v]),
\]
for every $D\in\operatorname{End}_{\mathbb F_p}(\mathfrak b)$ and all $u,v\in\mathfrak b$.

\begin{proposition}\label{prop:Kourovka-explicit-K}
The subspace $K$ is a hyperplane of dimension $189$ in $\bigwedge^2\mathfrak b$ and satisfies the following conditions:
\begin{itemize}
\item
If $D\in\operatorname{End}_{\mathbb F_p}(\mathfrak b)$ satisfies
${d_{\mathrm{CE}}D}_{| K}=0$, then
$D\in\operatorname{ad}(\mathfrak b)$;
\item If $\tau\in\operatorname{Aut}(\mathfrak b)$ satisfies
$(\bigwedge^2\tau)(K)=K$, then
$\tau=\operatorname{id}_{\mathfrak b}$.
\end{itemize}
\end{proposition}

\begin{proof}
Since $\ell(b_3\wedge b_4)=1$ and $\dim \bigwedge^2\mathfrak b=190$,
we have $\dim K=189$.
Equation~\eqref{eq:Kourovka-explicit-form} also gives
\begin{equation}
\label{eq:Kourovka-explicit-radical}
\{z\in\mathfrak b:
  \ell(z\wedge u)=0
  \text{ for all }u\in\mathfrak b\}
= P=\langle h,y\rangle.
\end{equation}

Let $D\in\operatorname{End}(\mathfrak b)$ satisfy
${d_{\mathrm{CE}}D}_{| K}=0.$ By \eqref{eq:Kourovka-explicit-radical},
$h\wedge\mathfrak b+y\wedge\mathfrak b\subseteq K.$
Put
$$\mathfrak a_D =\{z\in\mathfrak b: [D,\operatorname{ad}(z)] =\operatorname{ad}(D(z))\}.$$
This is a linear subspace containing $h,y$.
For $u,v\in\mathfrak a_D$, we have
$[D,\operatorname{ad}([u,v])]=
\operatorname{ad}(D([u,v])).$ Hence $\mathfrak a_D$ is a Lie subalgebra.
As $h,y$ generate $\mathfrak b$, it follows that
$\mathfrak a_D=\mathfrak b$.
Thus ${d_{\mathrm{CE}}D}=0$, and $D$ is inner by
Proposition~\ref{lem:Kourovka-structure}.

Now let $\tau\in\operatorname{Aut}(\mathfrak b)$
satisfy $(\bigwedge^2\tau)(K)=K$.
The nonzero linear functionals
$\ell\circ\bigwedge^2\tau$ and $\ell$
have the same kernel. Therefore there exists
$\lambda\in\mathbb F_p^\times$ such that
$\ell(\tau(u)\wedge\tau(v))=\lambda\ell(u\wedge v)$ for all $u,v\in\mathfrak b$.
For $u\in P$, we have $\ell(u\wedge v)=0$ for
every $v\in\mathfrak b$.
Surjectivity of $\tau$ and
\eqref{eq:Kourovka-explicit-radical} therefore imply $\tau(P)\subseteq P$.
Since $\tau$ is invertible, $\tau(P)=P$.
Lemma~\ref{lem:Kourovka-generating-plane} now gives $\tau=\operatorname{id}_{\mathfrak b}$.
\end{proof}

\subsection{The associated uniform pro-$p$ group and its finite quotients} \label{subsec:Kourovka-quotients}

Let $X,Y$ be noncommuting formal indeterminates and define the Baker–Campbell–Hausdorff (BCH) series by
\begin{equation}\label{BCH}
    \operatorname{BCH}(X,Y) = \log\bigl(\exp(X)\exp(Y)\bigr)
= \sum_{r=1}^{\infty}u_r(X,Y),
\end{equation}
where $u_r$ denotes the homogeneous component
of degree $r$. By \cite[Theorem~6.28]{DDMS}, each $u_r$ is a rational linear combination of iterated Lie brackets involving exactly $r$ entries.

Recall that, a $\mathbb Z_p$-Lie algebra $\mathcal L$ is called \emph{powerful} if it is free of finite rank as a
$\mathbb Z_p$-module and satisfies
$[\mathcal L,\mathcal L]\subseteq p\mathcal L$
(see~\cite[Section~9.4]{DDMS}).
A pro-$p$ group $U$ is \emph{uniform} if it is torsion-free,
topologically finitely generated, and satisfies $[U,U]\subseteq U^p$,
where $U^p$ denotes the closed subgroup generated by the $p$th powers
of elements of $U$ (see~\cite[Theorem~4.5]{DDMS}).

The following theorem associates a uniform pro-$p$ group with every powerful $\mathbb Z_p$-Lie algebra.
\begin{theorem}[{\cite[Theorem~9.8]{DDMS}}]
\label{thm:DDMS-9-8} Let $\mathcal L$ be a powerful Lie algebra. Then the operation $z\star w=\operatorname{BCH}(z,w),$ $z,w\in\mathcal L,$ makes $\mathcal L$ into a uniform pro-$p$ group. If $\{a_1,\ldots,a_d\}$
is a basis for $\mathcal L$ over $\mathbb Z_p$, then $\{a_1,\ldots,a_d\}$
is a topological generating set for the group $(\mathcal L,\star)$,
which has dimension $d$.
\end{theorem}
Since $[p\mathfrak B,p\mathfrak B]\subseteq p(p\mathfrak B),$
the $\mathbb Z_p$-Lie algebra $p\mathfrak B$ is powerful.
Hence, by Theorem~\ref{thm:DDMS-9-8}, the Campbell--Hausdorff
multiplication $\star$ makes $p\mathfrak B$ into a uniform pro-$p$
group, which we denote by $\Gamma=(p\mathfrak B,\star)$.

All iterated brackets of degree at least two involving scalar multiples of $z$ vanish. Hence \eqref{BCH} gives $(nz)\star z=(n+1)z$ for all $n\geq1$. Induction therefore yields 
$z^{\star n}=nz$ for every integer $n\geq1$,
where $z^{\star n}$ denotes the $n$th power of $z$ in $\Gamma$.

For each $i\geq0$, the sublattice $p^{i+1}\mathfrak B$ is again
powerful and thus forms a closed subgroup of $\Gamma$ under the
same Campbell--Hausdorff multiplication. Consequently,
\begin{equation}\label{power}
\Gamma^{p^i} =\{z^{\star p^i}:z\in p\mathfrak B\} =p^i(p\mathfrak B) =p^i\Gamma.
\end{equation}

For each $n\geq1$, put
$\mathsf Q_n=\Gamma/\Gamma^{p^n}$ and define
$\eta_n:\mathfrak B\to\mathsf Q_n$ by
$\eta_n(u)=(pu)\Gamma^{p^n}$.
Setting $\mathsf F_i=\mathsf Q_n^{p^i}$
for $i\geq0$. Then
\[
\mathsf Q_n
=\mathsf F_0
\supseteq\mathsf F_1
\supseteq\cdots
\supseteq\mathsf F_n=1.
\]

For $u\in\mathfrak B$ and $t\in p^{n+1}\mathfrak B$, the BCH
coefficient estimates imply
\[
(pu)\star t-pu\in p^{n+1}\mathfrak B,
\qquad
(-pu)\star(pu+t)\in p^{n+1}\mathfrak B.
\]
Together with the associativity identity
\[
(pu)\star\bigl((-pu)\star(pu+t)\bigr)=pu+t,
\]
these inclusions yield $(pu)\Gamma^{p^n}=pu+p^{n+1}\mathfrak B.$
Consequently,
\[
\begin{aligned}
\eta_n(u)=\eta_n(v)
&\iff (pu)\Gamma^{p^n}=(pv)\Gamma^{p^n}\\
&\iff pu+p^{n+1}\mathfrak B=pv+p^{n+1}\mathfrak B\\
&\iff u-v\in p^n\mathfrak B.
\end{aligned}
\]
Thus $\eta_n$ induces a bijection
$\mathfrak B/p^n\mathfrak B\longrightarrow\mathsf Q_n$.

Moreover, for each $0\leq i<n$, the map
$$\tau_i\colon \mathsf F_i/\mathsf F_{i+1}\longrightarrow(\mathfrak b,+),
\qquad \tau_i\bigl(\eta_n(p^iu)\mathsf F_{i+1}\bigr)=\overline u,$$
is a group isomorphism, where $\overline u$ denotes the image of $u$ under the natural projection $\mathfrak B\to\mathfrak b=\mathfrak B/p\mathfrak B$.

To distinguish group commutators from Lie brackets, for a groups $G$ we write $[g,h]_{G}=g^{-1}\star h^{-1}\star g\star h$ for $g,h\in G$. 
Put $z=p^{i+1}u$ and $w=p^{j+1}v$ for $u,v \in \mathfrak B$ and $i, j \geq0$. Since the inverse of $z$ in $\Gamma$ is $-z$,
our commutator convention gives
$$[\eta_n(p^iu),\eta_n(p^jv)]_{\mathsf Q_n}
=\bigl((-z)\star(-w)\star z\star w\bigr)
\Gamma^{p^n}.$$ It follows that
$$(-z)\star(-w)\star z\star w
\equiv [z,w] =p^{i+j+2}[u,v]
\pmod{p^{i+j+3}\mathfrak B}.$$

The image of
$p^{i+j+3}\mathfrak B=\Gamma^{p^{i+j+2}}$
in $\mathsf Q_n$ is
$\mathsf F_{i+j+2}$.
Since $\bigl(p^{i+j+2}[u,v]\bigr)\Gamma^{p^n}
=\eta_n(p^{i+j+1}[u,v]),$ the formula \eqref{BCH} gives
\begin{equation}
\label{eq:Kourovka-power-commutators}
[\eta_n(p^iu),\eta_n(p^jv)]_{\mathsf Q_n}
\equiv
\eta_n(p^{i+j+1}[u,v])
\pmod{\mathsf F_{i+j+2}}
\end{equation}
for $u,v\in\mathfrak B$ and $i,j\geq0$.
In particular,
$[\mathsf F_i,\mathsf F_j]_{\mathsf Q_n}
\subseteq\mathsf F_{i+j+1}$.

For a group $\mathsf G$ and elements $g,x\in\mathsf G$, write $x^g=g^{-1}xg$ and denote by $\iota_g\in\operatorname{Inn}(\mathsf G)$ the inner automorphism defined by $\iota_g(x)=x^g$.
\begin{lemma}\label{lem:Kourovka-finite-quotients}
The subgroup $\mathsf F_{n-1}$ is central and consists precisely of the elements of $\mathsf Q_n$ whose orders divide $p$.
Moreover,
\begin{equation}\label{eq:Kourovka-uniform-orders}
|\mathsf Q_n|=p^{20n}, \qquad
|Z(\mathsf Q_n)|=p^{n+19}, \qquad
|\operatorname{Inn}(\mathsf Q_n)|=p^{19(n-1)}.
\end{equation}
\end{lemma}

\begin{proof} Applying \eqref{power} we get $\eta_n(u)^p=\eta_n(pu)$, we have
$\eta_n(u)^p=1$ if and only if
$u\in p^{n-1}\mathfrak B$.  Thus the elements of order dividing $p$ form
$\eta_n(p^{n-1}\mathfrak B)
=\mathsf F_{n-1}$.
This subgroup is central because
$[\mathsf F_{n-1},\mathsf F_0]_{\mathsf Q_n}
\subseteq\mathsf F_n=1$
by \eqref{eq:Kourovka-power-commutators}.
Since $\mathfrak B$ is a free $\mathbb Z_p$-module of rank $20$, then $|\mathfrak B/p^n\mathfrak B|=p^{20n}$. Therefore, the bijection from $\mathfrak B/p^n\mathfrak B$
to $\mathsf Q_n$ gives $|\mathsf Q_n|=p^{20n}$.

The identity
\begin{equation}\label{Exercise}
\exp(-a)x\exp(a)=\exp(-\operatorname{ad}(a))(x)
\end{equation}
(see~\cite[Exercise~6.12]{DDMS}) implies that
\begin{equation}\label{conj}
\iota_{\eta_n(u)}(\eta_n(v))
=\eta_n\bigl(\exp(-p\operatorname{ad}(u))(v)\bigr), \qquad u,v\in\mathfrak B.
\end{equation}

Indeed, \eqref{Exercise} gives
\begin{equation}\label{eq35}
\exp(-a)\exp(x)\exp(a)
=\exp\bigl(\exp(-\operatorname{ad}(a))(x)\bigr).
\end{equation}
For $a,x\in p\mathfrak B$, taking logarithms in \eqref{eq35}
and using the definition of the BCH product yields
\[
(-a)\star x\star a
=\log\bigl(\exp(-a)\exp(x)\exp(a)\bigr)
=\exp(-\operatorname{ad}(a))(x).
\]
Setting $a=pu$ and $x=pv$, we obtain
$(-pu)\star(pv)\star(pu)=p\,\exp(-p\operatorname{ad}(u))(v).$ Passing to the quotient, we conclude that
\[
\begin{aligned}
\iota_{\eta_n(u)}(\eta_n(v))
&=\eta_n(u)^{-1}\eta_n(v)\eta_n(u)\\
&=\bigl((-pu)\star(pv)\star(pu)\bigr)\Gamma^{p^n}\\
&=\bigl(p\,\exp(-p\operatorname{ad}(u))(v)\bigr)\Gamma^{p^n}\\
&=\eta_n\bigl(\exp(-p\operatorname{ad}(u))(v)\bigr),
\end{aligned}
\]
as asserted.

Set $\psi_u=\sum\limits_{r=0}^{\infty}
\frac{(-p\operatorname{ad}(u))^r}{(r+1)!},$ 
so that $\exp(-p\operatorname{ad}(u))-\operatorname{id}_{\mathfrak B}
=-p\operatorname{ad}(u)\,\psi_u.$
The series defining $\psi_u$ converges in
$\operatorname{End}_{\mathbb Z_p}(\mathfrak B)$ and satisfies
$\psi_u\equiv\operatorname{id}_{\mathfrak B}\pmod p$.
Hence $\psi_u$ is a $\mathbb Z_p$-module automorphism of
$\mathfrak B$, and therefore
\[
\bigl(\exp(-p\operatorname{ad}(u))
-\operatorname{id}_{\mathfrak B}\bigr)(\mathfrak B)
=-p[u,\psi_u(\mathfrak B)]
=p[u,\mathfrak B].
\]
It follows from \eqref{conj} that
\[
\eta_n(u)\in Z(\mathsf Q_n)
\iff [u,\mathfrak B]\subseteq p^{n-1}\mathfrak B.
\]
Since $t$ is central in $\mathfrak B$, induction on $j$ gives
$$[u,\mathfrak B]\subseteq p^j\mathfrak B
\iff u\in\mathbb Z_pt+p^j\mathfrak B,
\qquad j\geq0.$$ Consequently,
$Z(\mathsf Q_n)=\eta_n(\mathbb Z_pt+p^{n-1}\mathfrak B).$ Using the bijection induced by $\eta_n$, we obtain
\[
|Z(\mathsf Q_n)|
=\left|(\mathbb Z_pt+p^{n-1}\mathfrak B)/p^n\mathfrak B\right|
=p^{n+19}.
\]
Finally, $\operatorname{Inn}(\mathsf Q_n)
\cong\mathsf Q_n/Z(\mathsf Q_n)$ yields
$|\operatorname{Inn}(\mathsf Q_n)|=p^{19(n-1)}$, completing the proof.
\end{proof}

\subsection{A central extension determined by  the hyperplane}\label{subsec:Kourovka-extension}

Fix $K$ as in Proposition~\ref{prop:Kourovka-explicit-K}, and set
$C_K=\bigwedge^2\mathfrak b/K$.
Let $\omega\colon\bigwedge^2\mathfrak b\to C_K$ be the quotient map.
Then $\dim_{\mathbb F_p}C_K=1$.
For $u,v\in\mathfrak b$, write $\omega(u,v)=\omega(u\wedge v)$.

Fix $n\geq2$ and put $\mathsf Q=\mathsf Q_{n+1}$.
For $x\in\mathsf Q$, denote its image in
$\mathsf Q/\mathsf Q^p\cong\mathfrak b$ by $\overline x$.
Define a multiplication on $\mathsf Q\times C_K$ by
\begin{equation}\label{eq:Kourovka-extension}
(x,a)(y,b)
=
\bigl(xy,a+b+\tfrac12\omega(\overline x,\overline y)\bigr).
\end{equation}
Bilinearity of $\omega$ ensures associativity, while alternation
gives the identity $(1,0)$ and the inverse
$(x,a)^{-1}=(x^{-1},-a)$.
We denote this finite $p$-group by $\mathsf G_{p,n}$. The homomorphisms $\jmath(c)=(1,c)$ and $\pi(x,a)=x$ yield
a central extension
\[
1\longrightarrow(C_K,+)
\xrightarrow{\jmath}\mathsf G_{p,n}
\xrightarrow{\pi}\mathsf Q_{n+1}
\longrightarrow1.
\]
A direct calculation gives
\begin{equation}\label{eq:Kourovka-extension-identities}
(x,a)^p=(x^p,0), \qquad \mbox{and} \qquad 
[(x,a),(y,b)]_{\mathsf G_{p,n}}=\bigl([x,y]_{\mathsf Q},\omega(\overline x,\overline y)\bigr).
\end{equation}

For a finite $p$-group $H$, let $\Phi(H)$ denote its Frattini
subgroup, the intersection of its maximal proper subgroups,
and let $\Omega_1(H)$ denote the subgroup generated by its
elements of order dividing $p$.

\begin{proposition}\label{thm:Kourovka-extension-structure}
For every $n\geq2$, the group $\mathsf G_{p,n}$ satisfies
\[
\Phi(\mathsf G_{p,n})=\mathsf Q^p\times C_K,
\qquad
\Omega_1(\mathsf G_{p,n})=\mathsf Q^{p^n}\times C_K.
\]
The subgroup $\Omega_1(\mathsf G_{p,n})$ is central, characteristic, and elementary abelian. Moreover, there are
natural isomorphisms
\[
\mathsf G_{p,n}/\Phi(\mathsf G_{p,n})
\cong(\mathfrak b,+),
\qquad
\mathsf G_{p,n}/\Omega_1(\mathsf G_{p,n})
\cong\mathsf Q_n.
\]
\end{proposition}

\begin{proof}
Write $\mathsf G=\mathsf G_{p,n}$ and
$\mathsf G'=[\mathsf G,\mathsf G]_{\mathsf G}$.
The power identity in \eqref{eq:Kourovka-extension-identities}
gives $\mathsf G^p=\mathsf Q^p\times\{0\}$.
Moreover, the commutator identity gives
$\mathsf G'\subseteq\mathsf Q^p\times C_K$.
Thus, by \cite[Proposition~1.13]{DDMS}, we obtain $\Phi(\mathsf G)=\mathsf G^p\mathsf G'
\subseteq\mathsf Q^p\times C_K.$

Multiplication on $\mathsf Q^p\times C_K$ is componentwise, so the second projection
\[
\operatorname{pr}_2\colon
\mathsf Q^p\times C_K\longrightarrow(C_K,+)
\]
is a group homomorphism.
By \eqref{eq:Kourovka-extension-identities}, its image on
$\mathsf G'$ contains all values
$\omega(\overline x,\overline y)$.
These span $C_K$, and hence
$\operatorname{pr}_2(\mathsf G')=C_K$.
For each $c\in C_K$, choose $(d,c)\in\mathsf G'$.
Then $d\in\mathsf Q^p$, so
$(1,c)=(d^{-1},0)(d,c)\in\mathsf G^p\mathsf G'.$ Together with $\mathsf G^p=\mathsf Q^p\times\{0\}$,
this proves the reverse inclusion and therefore, $\Phi(\mathsf G)=\mathsf Q^p\times C_K.$ Projection onto the first coordinate consequently induces
$$\mathsf G/\Phi(\mathsf G)\cong\mathsf Q/\mathsf Q^p\cong(\mathfrak b,+).$$

Lemma~\ref{lem:Kourovka-finite-quotients} gives $\{x\in\mathsf Q:x^p=1\}=\mathsf Q^{p^n}.$ The power identity then implies
\[
\{g\in\mathsf G:g^p=1\}
=\mathsf Q^{p^n}\times C_K.
\]
Since $\mathsf Q^{p^n}$ is central in $\mathsf Q$ and contained
in $\mathsf Q^p$, the multiplication and commutator identities
show that this set is a central elementary abelian subgroup.
It therefore equals $\Omega_1(\mathsf G)$ and is characteristic,
as automorphisms preserve element orders.

Finally, composing $\pi$ with the natural projection
$\mathsf Q_{n+1}\to\mathsf Q_n$ gives a surjective homomorphism
\[
\kappa\colon\mathsf G\longrightarrow\mathsf Q_n,
\qquad
\kappa\bigl((\eta_{n+1}(u),c)\bigr)=\eta_n(u).
\]
Its kernel is $\mathsf Q^{p^n}\times C_K=\Omega_1(\mathsf G)$. Hence $\kappa$ induces the natural isomorphism
\[
\mathsf G/\Omega_1(\mathsf G)\xrightarrow{\ \cong\ }\mathsf Q_n,
\qquad (\eta_{n+1}(u),c)\Omega_1(\mathsf G)\longmapsto\eta_n(u).
\]
\end{proof}

\subsection{Proof of Theorem~\ref{thm:Kourovka-application}} \label{subsec:Kourovka-main-proof}

We now determine the order of $\operatorname{Aut}(\mathsf G_{p,n})$. We first show that the automorphisms of $\mathsf Q_n$ induced by
automorphisms of $\mathsf G_{p,n}$ are precisely the inner
automorphisms of $\mathsf Q_n$. We then compute the order of the
kernel of the induced homomorphism and choose $n$ so that
$|\operatorname{Aut}(\mathsf G_{p,n})|=|\mathsf G_{p,n}|$.

Since $\kappa$ is surjective and
$\ker\kappa=\Omega_1(\mathsf G)$ is characteristic in $\mathsf G$,
the rule
\[
\Upsilon_n(\sigma)\bigl(\kappa(g)\bigr)
=\kappa\bigl(\sigma(g)\bigr),
\qquad
\sigma\in\operatorname{Aut}(\mathsf G),\quad g\in\mathsf G,
\]
defines a group homomorphism
$\Upsilon_n\colon\operatorname{Aut}(\mathsf G)
\rightarrow\operatorname{Aut}(\mathsf Q_n).$

For $\sigma\in\operatorname{Aut}(\mathsf G)$, set $\alpha=\Upsilon_n(\sigma)$. Under the identification
$$\mathsf Q_n/\mathsf F_1 \xrightarrow{\;\sim\;} \mathfrak b,
\qquad \eta_n(u)\mathsf F_1\longmapsto\overline u:=u+p\mathfrak B,$$
the action of $\alpha$ on $\mathsf Q_n/\mathsf F_1$ induces a
linear automorphism $\overline\alpha$ of $\mathfrak b$.
Explicitly, for $u\in\mathfrak B$, choose $v\in\mathfrak B$ such that
$\alpha(\eta_n(u))=\eta_n(v)$ and set 
$\overline\alpha(\overline u)=\overline v.$
Since $\alpha(\mathsf F_1)=\mathsf F_1$, the class $\overline v$ depends only on $\overline u$, so this definition is independent of all choices.

Since $\Phi(\mathsf G)$ and $\mathsf G^p$ are characteristic in $\mathsf G$, the automorphism $\sigma$ induces an automorphism
$$\widetilde\sigma\colon
\Phi(\mathsf G)/\mathsf G^p
\longrightarrow \Phi(\mathsf G)/\mathsf G^p,
\qquad h\mathsf G^p\longmapsto\sigma(h)\mathsf G^p.$$
Thus, the projection $\operatorname{pr}_2\colon
\Phi(\mathsf G)\to C_K$, given by $(x,c)\mapsto c$,
induces a natural isomorphism
$\Phi(\mathsf G)/\mathsf G^p\xrightarrow{\;\sim\;}C_K$ via $(x,c)\mathsf G^p\mapsto c.$
Under this identification, $\widetilde\sigma$ corresponds to the additive automorphism
$\beta\colon C_K\rightarrow C_K$ with $\beta(c)=\operatorname{pr}_2\bigl(\sigma(1,c)\bigr),$ characterized by $\sigma(1,c)\mathsf G^p=(1,\beta(c))\mathsf G^p.$ 

\begin{lemma}
\label{lem:Kourovka-first-quotients}
Let $\sigma\in\operatorname{Aut}(\mathsf G)$.
Then $\overline\alpha=\operatorname{id}_{\mathfrak b}$ and $\beta=\operatorname{id}_{C_K}.$
\end{lemma}
\begin{proof}Since $\eta_n(u)^p=\eta_n(pu)$ and
$\alpha(x^p)=\alpha(x)^p$, we obtain
$\alpha(\mathsf F_1)=\mathsf F_1$ and $\alpha(\mathsf F_2)=\mathsf F_2$, and hence induces an automorphism
\[
\alpha_1\colon\mathsf F_1/\mathsf F_2
\longrightarrow\mathsf F_1/\mathsf F_2,
\qquad
g\mathsf F_2\longmapsto\alpha(g)\mathsf F_2.
\]
For $u\in\mathfrak B$, choose $v\in\mathfrak B$ such that
$\alpha(\eta_n(u))=\eta_n(v)$, so that
$\overline\alpha(\overline u)=\overline v$.
Then we obtain
\[
\alpha_1\bigl(\eta_n(pu)\mathsf F_2\bigr)
=\alpha\bigl(\eta_n(u)^p\bigr)\mathsf F_2
=\alpha\bigl(\eta_n(u)\bigr)^p\mathsf F_2
=\eta_n(pv)\mathsf F_2.
\]
Thus, under the identification $
\mathsf F_1/\mathsf F_2\xrightarrow{\;\sim\;}\mathfrak b$ via $
\eta_n(pu)\mathsf F_2\mapsto\overline u,$ the induced automorphism $\alpha_1$ corresponds to $\overline\alpha$.

Now let $u,v\in\mathfrak B$, and choose $u',v'\in\mathfrak B$ such that
$\alpha(\eta_n(u))=\eta_n(u')$ and $\alpha(\eta_n(v))=\eta_n(v').$ Then $\overline{u'}=\overline\alpha(\overline u)$ and
$\overline{v'}=\overline\alpha(\overline v)$.
Taking $i=j=0$ in \eqref{eq:Kourovka-power-commutators} gives
\[
[\eta_n(u),\eta_n(v)]_{\mathsf Q_n}
\equiv\eta_n(p[u,v])\pmod{\mathsf F_2}.
\]
Since $\alpha$ preserves $\mathsf F_2$ and group commutators,
applying $\alpha$ and using the same formula for $u',v'$ yields
\[
\alpha\bigl(\eta_n(p[u,v])\bigr)
\equiv[\eta_n(u'),\eta_n(v')]_{\mathsf Q_n}
\equiv\eta_n(p[u',v'])\pmod{\mathsf F_2}.
\]
By the identification above, this congruence implies $\overline\alpha([\overline u,\overline v])=[\overline\alpha(\overline u),\overline\alpha(\overline v)].$
Hence, we obtain $\overline\alpha\in\operatorname{Aut}(\mathfrak b)$.

Since $\ker\kappa\leq\Phi(\mathsf G)$, the map $\kappa$ induces
the natural identification
\[
\mathsf G/\Phi(\mathsf G)
\cong\mathsf Q_n/\mathsf Q_n^p
\cong\mathfrak b.
\]
The identity $\kappa\circ\sigma=\alpha\circ\kappa$ shows that,
under this identification, the automorphism $\sigma_\Phi$
induced by $\sigma$ corresponds to $\overline\alpha$.

Let $u,v\in\mathfrak b$, and choose $x,y\in\mathsf Q$ with
$\overline x=u$ and $\overline y=v$. Write
$\sigma(x,0)=(x',a)$ and $\sigma(y,0)=(y',b).$ Then we have 
$\overline{x'}=\overline\alpha(u)$ and $\overline{y'}=\overline\alpha(v).$
The commutator identity in
\eqref{eq:Kourovka-extension-identities}, together with
$[\mathsf Q,\mathsf Q]_{\mathsf Q}\subseteq\mathsf Q^p$, gives
$[(x,0),(y,0)]_{\mathsf G}\mathsf G^p
=(1,\omega(u,v))\mathsf G^p.$ Applying $\sigma$, using its induced action $\beta$ on
$\Phi(\mathsf G)/\mathsf G^p$, and applying identity \eqref{eq:Kourovka-extension-identities} again, we obtain
\[
\begin{aligned}
(1,\beta(\omega(u,v)))\mathsf G^p
&=[(x',a),(y',b)]_{\mathsf G}\mathsf G^p\\
&=\bigl(1,\omega(\overline\alpha(u),\overline\alpha(v))\bigr)
\mathsf G^p.
\end{aligned}
\]
Under the identification $\Phi(\mathsf G)/\mathsf G^p\cong C_K$,
this yields
\begin{equation}\label{eq38}
\beta(\omega(u,v))
=\omega(\overline\alpha(u),\overline\alpha(v),
\qquad u,v\in\mathfrak b.
\end{equation}
Since $K=\ker\omega$, this identity implies 
$\bigl(\bigwedge\nolimits^2\overline\alpha\bigr)(K)\subseteq K.$ As $\bigwedge^2\overline\alpha$ is invertible and $K$ is finite-dimensional, the inclusion is an equality. Proposition~\ref{prop:Kourovka-explicit-K} therefore gives $\overline\alpha=\operatorname{id}_{\mathfrak b}$. Identity \eqref{eq38} then reduces to $\beta\circ\omega=\omega$. Since $\omega$ is surjective, we conclude that $\beta=\operatorname{id}_{C_K}$.
\end{proof}

We next show that every automorphism of $\mathsf G$ can be composed with a suitable inner automorphism so that its induced action on $\mathsf Q_n/\mathsf F_{n-1}$ is trivial.

\begin{lemma}
\label{lem:Kourovka-last-layer-reduction}
Let $\sigma\in\operatorname{Aut}(\mathsf G)$.
Then there exists $g\in\mathsf G$ such that
$$\Upsilon_n(\iota_g^{-1}\circ\sigma)(x)\mathsf F_{n-1}=x\mathsf F_{n-1} \quad \mbox{for every} \quad x\in\mathsf Q_n.$$
\end{lemma}
\begin{proof} By Lemma~\ref{lem:Kourovka-first-quotients}, $\alpha$ acts trivially
on $\mathsf Q_n/\mathsf F_1$. If $n=2$, the conclusion follows with $g=1$. We therefore assume that $n\geq3$. For the induction step, suppose that $\alpha$ acts trivially
on $\mathsf Q_n/\mathsf F_s$ for some $1\leq s<n-1$. For $x\in\mathsf Q_n$ we set $\delta(x)=x^{-1}\alpha(x)\in\mathsf F_s$.
Since $[\mathsf F_s,\mathsf Q_n]_{\mathsf Q_n}
\subseteq\mathsf F_{s+1}$, we have
\[
\delta(xy)
=y^{-1}\delta(x)y\,\delta(y)
\equiv\delta(x)\delta(y)\pmod{\mathsf F_{s+1}}.
\]
Thus, the map $\mathsf Q_n\rightarrow\mathsf F_s/\mathsf F_{s+1}$, $
x\mapsto\delta(x)\mathsf F_{s+1},$
is a homomorphism. Its kernel
contains $\mathsf F_1=\mathsf Q_n^p$. Under the identifications
\[
\eta_n(u)\mathsf F_1\longmapsto\overline u,
\qquad
\tau_s\colon
\eta_n(p^sv)\mathsf F_{s+1}\longmapsto\overline v,
\]
the resulting homomorphism defines an $\mathbb F_p$-linear map
$D_s\colon\mathfrak b\to\mathfrak b$, characterized by
\[
D_s(\overline u)=\overline v
\quad\Longleftrightarrow\quad
\eta_n(u)^{-1}\alpha(\eta_n(u))
\equiv\eta_n(p^sv)\pmod{\mathsf F_{s+1}}.
\]

We claim that $D_s$ is a derivation.
Let $u,v\in\mathfrak B$, and put
$x=\eta_n(u)$ and $y=\eta_n(v)$.
Since $[x,y]_{\mathsf Q_n}\in\mathsf F_1=\eta_n(p\mathfrak B)$,
we may choose $w\in\mathfrak B$ such that, with $z=\eta_n(w)$,
\[
[x,y]_{\mathsf Q_n}=\eta_n(pw)=z^p.
\]
Taking $i=j=0$ in \eqref{eq:Kourovka-power-commutators} gives $\eta_n(pw)\mathsf F_2
=\eta_n(p[u,v])\mathsf F_2.$ Applying $\tau_1$ to this equality, we obtain
$\overline w=[\overline u,\overline v]$.

Choose $a,b,r\in\mathfrak B$ such that
\[
d:=x^{-1}\alpha(x)=\eta_n(p^sa),
\qquad
e:=y^{-1}\alpha(y)=\eta_n(p^sb),
\qquad
h:=z^{-1}\alpha(z)=\eta_n(p^sr).
\]
Then we have $D_s(\overline u)=\overline a,$
$D_s(\overline v)=\overline b$ and $D_s(\overline w)=\overline r.$ For the next two calculations, all congruences are taken
modulo $\mathsf F_{s+2}$.
The subgroup $\mathsf F_{s+1}/\mathsf F_{s+2}$ is central
in $\mathsf Q_n/\mathsf F_{s+2}$.
In particular, $[h,z]_{\mathsf Q_n}\in\mathsf F_{s+1}$
is central modulo $\mathsf F_{s+2}$.
Since $p$ is odd, $p$ divides $\binom p2$, and collecting
the factors in $(zh)^p$ gives
\[
(zh)^p
\equiv z^ph^p[h,z]_{\mathsf Q_n}^{\binom p2}
\equiv z^ph^p.
\]
Applying $\alpha$ to $[x,y]_{\mathsf Q_n}=z^p$, we therefore obtain
\[
[x,y]_{\mathsf Q_n}^{-1}
[\alpha(x),\alpha(y)]_{\mathsf Q_n}
=(z^p)^{-1}(zh)^p
\equiv h^p
=\eta_n(p^{s+1}r).
\]

On the other hand, we have 
\[
[\mathsf F_1,\mathsf F_s]_{\mathsf Q_n}
\subseteq\mathsf F_{s+2},
\qquad
[\mathsf F_s,\mathsf F_s]_{\mathsf Q_n}
\subseteq\mathsf F_{2s+1}\subseteq\mathsf F_{s+2},
\]
where the last inclusion follows from $s\geq1$.
Using these inclusions and the centrality of
$\mathsf F_{s+1}/\mathsf F_{s+2}$, the commutator identities yield
\[
\begin{aligned}
[xd,ye]_{\mathsf Q_n}
&=[xd,e]_{\mathsf Q_n}\,
  e^{-1}[xd,y]_{\mathsf Q_n}e\\
&\equiv[x,e]_{\mathsf Q_n}
  [x,y]_{\mathsf Q_n}[d,y]_{\mathsf Q_n}\\
&\equiv[x,y]_{\mathsf Q_n}
  [d,y]_{\mathsf Q_n}[x,e]_{\mathsf Q_n}.
\end{aligned}
\]
Consequently, applying \eqref{eq:Kourovka-power-commutators}
with $(i,j)=(s,0)$ and $(i,j)=(0,s)$ gives
\[
\begin{aligned}
[x,y]_{\mathsf Q_n}^{-1}
[\alpha(x),\alpha(y)]_{\mathsf Q_n}
&\equiv[d,y]_{\mathsf Q_n}[x,e]_{\mathsf Q_n}\\
&\equiv\eta_n(p^{s+1}[a,v])\eta_n(p^{s+1}[u,b])\\
&\equiv\eta_n\bigl(p^{s+1}([a,v]+[u,b])\bigr).
\end{aligned}
\]
The last congruence follows because $\tau_{s+1}$ identifies
multiplication in $\mathsf F_{s+1}/\mathsf F_{s+2}$ with
addition in $\mathfrak b$.

Comparing the two expressions and applying $\tau_{s+1}$, we obtain $\overline r
=[\overline a,\overline v]+[\overline u,\overline b].$ Since $\overline w=[\overline u,\overline v]$, the defining
relations for $D_s$ give $D_s([\overline u,\overline v])=[D_s(\overline u),\overline v]
+[\overline u,D_s(\overline v)].$ Thus, $D_s$ is a derivation of $\mathfrak b$.

By Proposition~\ref{lem:Kourovka-structure}, $D_s$ is inner.
Choose $w_s\in\mathfrak B$ such that
$D_s=\operatorname{ad}(\overline{w_s})$, and put
$c=\eta_n(-p^{s-1}w_s)$.
For $x=\eta_n(u)$, taking $(i,j)=(0,s-1)$ in
\eqref{eq:Kourovka-power-commutators} gives
\[
\begin{aligned}
x^{-1}\iota_c(x)
&=[x,c]_{\mathsf Q_n}\\
&\equiv\eta_n(p^s[u,-w_s])
=\eta_n(p^s[w_s,u])\\
&\equiv x^{-1}\alpha(x)
\pmod{\mathsf F_{s+1}}.
\end{aligned}
\]
Here the last congruence follows from the definition of $D_s$
and the equality
$D_s(\overline u)=\overline{[w_s,u]}$.
Hence $\iota_c$ and $\alpha$ induce the same automorphism
on $\mathsf Q_n/\mathsf F_{s+1}$.

Now set $\widetilde c
=\bigl(\eta_{n+1}(-p^{s-1}w_s),0\bigr)\in\mathsf G$ and $\sigma'=\iota_{\widetilde c}^{-1}\circ\sigma.$ Since $\Upsilon_n(\iota_{\widetilde c})=\iota_c$, we have $\Upsilon_n(\sigma')=\iota_c^{-1}\circ\alpha.$
It follows that $\Upsilon_n(\sigma')(x)\mathsf F_{s+1}=x\mathsf F_{s+1}$ for  $x\in\mathsf Q_n$, which completes the induction step.

Applying these arguments successively for $s=1,\ldots,n-2$, let $\widetilde c_s$ denote the element used at step $s$, and put
$g=\widetilde c_{n-2}\cdots\widetilde c_1.$ 
The resulting automorphism is
\[
\iota_{\widetilde c_{n-2}}^{-1}\circ\cdots\circ
\iota_{\widetilde c_1}^{-1}\circ\sigma
=\iota_g^{-1}\circ\sigma.
\]
Therefore, we obtain $\Upsilon_n(\iota_g^{-1}\circ\sigma)(x)\mathsf F_{n-1}=x\mathsf F_{n-1}$ for $x\in\mathsf Q_n$ as required.
\end{proof}

We now complete the preceding reduction by showing that every automorphism of $\mathsf G$ induces an inner automorphism of $\mathsf Q_n$.

\begin{proposition}\label{lem:Kourovka-automorphisms}
We have $\operatorname{im}\Upsilon_n=\operatorname{Inn}(\mathsf Q_n)$.
\end{proposition}
\begin{proof} The surjectivity of $\kappa$ and the identity
$\Upsilon_n(\iota_g)=\iota_{\kappa(g)}$ imply that
$\operatorname{Inn}(\mathsf Q_n)
\subseteq\operatorname{im}\Upsilon_n.$ 
For the reverse inclusion, let
$\sigma\in\operatorname{Aut}(\mathsf G)$ such that 
$\alpha=\Upsilon_n(\sigma)$.
By Lemma~\ref{lem:Kourovka-last-layer-reduction}, we may assume that $\alpha$
acts trivially on $\mathsf Q_n/\mathsf F_{n-1}$.
Lemma~\ref{lem:Kourovka-first-quotients} also gives
$\overline\alpha=\operatorname{id}_{\mathfrak b}$ and
$\beta=\operatorname{id}_{C_K}$.

As in the proof of Lemma~\ref{lem:Kourovka-last-layer-reduction},
the map $x\mapsto x^{-1}\alpha(x)$ induces an
$\mathbb F_p$-linear map $D\colon\mathfrak b\to\mathfrak b$,
characterized by
\[
D(\overline u)=\overline a
\quad\Longleftrightarrow\quad
\eta_n(u)^{-1}\alpha(\eta_n(u))
=\eta_n(p^{n-1}a).
\]
We shall prove that $D$ is inner. Let $u,v\in\mathfrak B$, and put
$x=\eta_{n+1}(u)$ and $y=\eta_{n+1}(v)$.
We compute $[(x,0),(y,0)]_{\mathsf G}^{-1}
\sigma\bigl([(x,0),(y,0)]_{\mathsf G}\bigr)$ in two ways.
Since $\kappa\circ\sigma=\alpha\circ\kappa$ and the inverse image of $\mathsf F_{n-1}$ under the natural projection
$\mathsf Q\to\mathsf Q_n$ is $\mathsf Q^{p^{n-1}}$, we may write $\sigma(x,0)=(xA,c)$ and $\sigma(y,0)=(yB,d),$ where $A=\eta_{n+1}(p^{n-1}a)$ and $B=\eta_{n+1}(p^{n-1}b)$ for some $a,b\in\mathfrak B$ and $c,d\in C_K$. Applying $\kappa$ and using the definition of $D$, we obtain $\overline a=D(\overline u)$ and $\overline b=D(\overline v).$

We have $[\mathsf Q^{p^{n-1}},\mathsf Q^{p^{n-1}}]_{\mathsf Q}=\{1\}$ and $[\mathsf Q^p,\mathsf Q^{p^{n-1}}]_{\mathsf Q}=\{1\}.$ Moreover, $[A,y]_{\mathsf Q}$ and $[x,B]_{\mathsf Q}$ belong
to the central subgroup $\mathsf Q^{p^n}$.
The group commutator identities therefore yield
\[
[xA,yB]_{\mathsf Q}
=[x,y]_{\mathsf Q}[A,y]_{\mathsf Q}[x,B]_{\mathsf Q}.
\]
Applying \eqref{eq:Kourovka-power-commutators}, we obtain
\[
[A,y]_{\mathsf Q}=\eta_{n+1}(p^n[a,v]),
\qquad
[x,B]_{\mathsf Q}=\eta_{n+1}(p^n[u,b]).
\]
These equalities are exact because
$\mathsf Q^{p^{n+1}}=\{1\}$.

Since $A,B\in\mathsf Q^p$, the elements $xA,yB$ have the same
images in $\mathfrak b$ as $x,y$, respectively.
Thus \eqref{eq:Kourovka-extension-identities} shows that
$[(x,0),(y,0)]_{\mathsf G}$ and its image under $\sigma$
have the same second coordinate.
Their first coordinates lie in $\mathsf Q^p$, so the twisting
term vanishes when taking their quotient.
Using the additive structure of $\mathsf Q^{p^n}$, we conclude that
\begin{equation}\label{eq:comm1}
\begin{aligned}
&[(x,0),(y,0)]_{\mathsf G}^{-1}
 \sigma\bigl([(x,0),(y,0)]_{\mathsf G}\bigr)\\
&\qquad=
\bigl([x,y]_{\mathsf Q}^{-1}[xA,yB]_{\mathsf Q},0\bigr)\\
&\qquad=
\bigl(\eta_{n+1}(p^n([a,v]+[u,b])),0\bigr).
\end{aligned}
\end{equation}

For the second computation, use
$[x,y]_{\mathsf Q}\in\mathsf Q^p=\eta_{n+1}(p\mathfrak B)$
to choose $w\in\mathfrak B$ such that
$[x,y]_{\mathsf Q}=z^p$ , with $z=\eta_{n+1}(w)$.
The commutator formula gives
$\overline w=[\overline u,\overline v]$, while
\eqref{eq:Kourovka-extension-identities} yields
\[
[(x,0),(y,0)]_{\mathsf G}
=(z^p,0)(1,\omega(\overline u,\overline v)).
\]

As above, we may write $\sigma(z,0)=(zR,e),$ where $R=\eta_{n+1}(p^{n-1}r)$ and $r\in\mathfrak B$, $e\in C_K$, and
$\overline r=D(\overline w)$.
The commutator $[R,z]_{\mathsf Q}$ belongs to the central
subgroup $\mathsf Q^{p^n}$, which has exponent $p$.
Since $p$ is odd, collecting the factors in $(zR)^p$ gives
\[
(zR)^p
=z^pR^p[R,z]_{\mathsf Q}^{\binom p2}
=z^pR^p.
\]
Consequently, the power identity in
\eqref{eq:Kourovka-extension-identities} implies
\[
\sigma(z^p,0)
=\sigma(z,0)^p
=\bigl(z^p\eta_{n+1}(p^nr),0\bigr),
\]
and hence $(z^p,0)^{-1}\sigma(z^p,0)
=\bigl(\eta_{n+1}(p^nr),0\bigr).$

We next determine the action of $\sigma$ on $\{1\}\times C_K$.
Since $\beta=\operatorname{id}_{C_K}$, we get
$\sigma(1,c)\mathsf G^p=(1,c)\mathsf G^p$ for $c\in C_K.$ Moreover, $(1,c)\in\Omega_1(\mathsf G)$, and
$\Omega_1(\mathsf G)$ is characteristic.
It follows that
\[
\sigma(1,c)(1,c)^{-1}
\in\Omega_1(\mathsf G)\cap\mathsf G^p
=\mathsf Q^{p^n}\times\{0\}.
\]
Thus, there is a well-defined map
$\theta\colon C_K\to\mathfrak b$ characterized by
\[
\sigma(1,c)=\bigl(\eta_{n+1}(p^n\xi),c\bigr),
\qquad
\theta(c)=\overline\xi.
\]
Here well-definedness follows from the identification
$\mathsf Q^{p^n}\cong(\mathfrak b,+)$.
Since $\{1\}\times C_K$ and its image under $\sigma$ are central,
the map $c\mapsto\sigma(1,c)(1,c)^{-1}$ is a homomorphism.
Hence $\theta$ is additive and therefore $\mathbb F_p$-linear.

Choose $\zeta_\theta\in\mathfrak B$ such that
$\overline{\zeta_\theta}=\theta(\omega(\overline u,\overline v)).$ Then
\[
(1,\omega(\overline u,\overline v))^{-1}
\sigma(1,\omega(\overline u,\overline v))
=\bigl(\eta_{n+1}(p^n\zeta_\theta),0\bigr).
\]
Using the factorization of the commutator above and the
centrality of $\{1\}\times C_K$ and its image, we obtain
\begin{equation}\label{eq:comm2}
\begin{aligned}
&[(x,0),(y,0)]_{\mathsf G}^{-1}
 \sigma\bigl([(x,0),(y,0)]_{\mathsf G}\bigr)\\
&\qquad=
\bigl((z^p,0)^{-1}\sigma(z^p,0)\bigr)
\bigl((1,\omega(\overline u,\overline v))^{-1}
      \sigma(1,\omega(\overline u,\overline v))\bigr)\\
&\qquad=
\bigl(\eta_{n+1}(p^n(r+\zeta_\theta)),0\bigr).
\end{aligned}
\end{equation}

Comparing \eqref{eq:comm1} and \eqref{eq:comm2} in
$\mathsf Q^{p^n}\cong(\mathfrak b,+)$ gives $[\overline a,\overline v]+[\overline u,\overline b]
=\overline r+\overline{\zeta_\theta}.$ Substituting the defining relations for $D$ and $\theta$, we obtain
\[
[D(\overline u),\overline v]
+[\overline u,D(\overline v)]
-D([\overline u,\overline v])
=\theta(\omega(\overline u,\overline v)).
\]
Since $K=\ker\omega$, this identity implies
$\left.\partial D\right|_K=0$.
Proposition~\ref{prop:Kourovka-explicit-K} therefore shows
that $D$ is inner.

Choose $a_D\in\mathfrak B$ such that
$D=\operatorname{ad}(\overline{a_D})$.
For $u\in\mathfrak B$, put $g_u=\eta_n(u)$.
The definition of $D$ gives $g_u^{-1}\alpha(g_u)
=\eta_n(p^{n-1}[a_D,u]).$ 
Taking $(i,j)=(0,n-2)$ in
\eqref{eq:Kourovka-power-commutators}, and using
$\mathsf F_n=\{1\}$, we obtain
\[
\begin{aligned}
g_u^{-1}\iota_{\eta_n(-p^{n-2}a_D)}(g_u)
&=[g_u,\eta_n(-p^{n-2}a_D)]_{\mathsf Q_n}\\
&=\eta_n(p^{n-1}[u,-a_D])\\
&=\eta_n(p^{n-1}[a_D,u])
=g_u^{-1}\alpha(g_u).
\end{aligned}
\]
Since every element of $\mathsf Q_n$ has the form $g_u$, it follows that $\alpha=\iota_{\eta_n(-p^{n-2}a_D)}$ is inner.

The initial reduction changed $\alpha$ only by composition
with an inner automorphism of $\mathsf Q_n$.
Thus the original $\alpha$ is also inner, proving
$\operatorname{im}\Upsilon_n\subseteq
\operatorname{Inn}(\mathsf Q_n)$.
\end{proof}

Proposition~\ref{lem:Kourovka-automorphisms} determines the image
of $\Upsilon_n$. It remains to determine its kernel.
Each automorphism in $\ker\Upsilon_n$ differs from the identity
by a map into the central elementary abelian subgroup
$\Omega_1(\mathsf G)$. We shall identify these automorphisms
with $\mathbb F_p$-linear maps from
$\mathsf G/\Phi(\mathsf G)\cong\mathfrak b$ to
$\Omega_1(\mathsf G)$. Counting the kernel will then allow us
to choose $n$ so that
$|\operatorname{Aut}(\mathsf G)|=|\mathsf G|$.

We regard $\Omega_1(\mathsf G)$ as an $\mathbb F_p$-vector space,
with vector addition given by group multiplication and scalar
multiplication by taking powers. Thus
\[
\operatorname{Hom}_{\mathbb F_p}
\bigl(\mathfrak b,\Omega_1(\mathsf G)\bigr)
=
\left\{
f\colon\mathfrak b\to\Omega_1(\mathsf G)
\;\middle|\;
f\text{ is }\mathbb F_p\text{-linear}
\right\}.
\]
In multiplicative notation, linearity means that
$f(u+v)=f(u)f(v)$ for all $u,v\in\mathfrak b$ and integers $0\leq k<p$.
In particular, $f(0)=1$, where $1$ denotes the group identity. Addition in this Hom space is given pointwise by $(f_1+f_2)(u)=f_1(u)f_2(u),$ $u\in\mathfrak b.$

\begin{proof}[Proof of Theorem~\ref{thm:Kourovka-application}]
We first determine $\ker\Upsilon_n$. Recall that
$\Omega_1(\mathsf G)=\mathsf Q^{p^n}\times C_K$
is elementary abelian and contained in $\Phi(\mathsf G)$.
It is also central: indeed,
$\mathsf Q^{p^n}\subseteq Z(\mathsf Q)\cap\mathsf Q^p$
by Lemma~\ref{lem:Kourovka-finite-quotients}, so
\eqref{eq:Kourovka-extension-identities} gives
$[\Omega_1(\mathsf G),\mathsf G]_{\mathsf G}=\{1\}.$

Let $\sigma\in\ker\Upsilon_n$ and define
$\delta_\sigma(x)=x^{-1}\sigma(x)$ for $x\in\mathsf G$. Since $\Upsilon_n(\sigma)=\operatorname{id}_{\mathsf Q_n}$, the identity $\kappa\circ\sigma=\kappa$ implies that
$\delta_\sigma(x)\in\ker\kappa=\Omega_1(\mathsf G).$ 
By centrality, we have 
\[
\delta_\sigma(xy)
=y^{-1}\delta_\sigma(x)y\,\delta_\sigma(y)
=\delta_\sigma(x)\delta_\sigma(y).
\]

Thus $\delta_\sigma\colon\mathsf G\to\Omega_1(\mathsf G)$ is a homomorphism, which is trivial on both $\mathsf G'$ and $\mathsf G^p$, and hence on
$\Phi(\mathsf G)=\mathsf G^p\mathsf G'$.
It therefore induces a well-defined homomorphism
\[
f_\sigma\colon\mathsf G/\Phi(\mathsf G)
\longrightarrow\Omega_1(\mathsf G),
\qquad
f_\sigma(x\Phi(\mathsf G))=x^{-1}\sigma(x).
\]
Both groups are elementary abelian, so $f_\sigma$ is
$\mathbb F_p$-linear.

Conversely, let
$f\colon\mathsf G/\Phi(\mathsf G)\to\Omega_1(\mathsf G)$
be an $\mathbb F_p$-linear map, and define
$\sigma_f(x)=x f(x\Phi(\mathsf G))$ for $x\in\mathsf G.$ Linearity of $f$, expressed in multiplicative notation,
and centrality of $\Omega_1(\mathsf G)$ give
\[
\begin{aligned}
\sigma_f(xy)
&=xy f(x\Phi(\mathsf G))f(y\Phi(\mathsf G))\\
&=\sigma_f(x)\sigma_f(y).
\end{aligned}
\]
Thus, $\sigma_f$ is an endomorphism of $\mathsf G$.
Moreover, since $\Omega_1(\mathsf G)\subseteq\Phi(\mathsf G)$, we get 
$\sigma_f(x)\Phi(\mathsf G)=x\Phi(\mathsf G)$, $x\in\mathsf G$. 
Consequently, for any two such maps $f_1,f_2$,
\[
\begin{aligned}
(\sigma_{f_1}\circ\sigma_{f_2})(x)
&=\sigma_{f_2}(x)
  f_1(\sigma_{f_2}(x)\Phi(\mathsf G))\\
&=x f_2(x\Phi(\mathsf G))f_1(x\Phi(\mathsf G))\\
&=\sigma_{f_1+f_2}(x).
\end{aligned}
\]
In particular,
$\sigma_f\circ\sigma_{-f}
=\sigma_{-f}\circ\sigma_f=\operatorname{id}_{\mathsf G}$,
so $\sigma_f$ is an automorphism.
Since
\[
x^{-1}\sigma_f(x)=f(x\Phi(\mathsf G))
\in\Omega_1(\mathsf G)=\ker\kappa,
\]
we also have $\sigma_f\in\ker\Upsilon_n$.

Therefore gives an isomorphism of groups
$\ker\Upsilon_n \cong \operatorname{Hom}_{\mathbb F_p}
\bigl(\mathfrak b,\Omega_1(\mathsf G)\bigr).$

By Lemma~\ref{lem:Kourovka-finite-quotients},
$|\mathsf Q|=p^{20(n+1)}$.
Since $\dim_{\mathbb F_p}C_K=1$,
we have $|\mathsf G|
=|\mathsf Q|\,|C_K|
=p^{20n+21}.$ 
Furthermore, 
$\mathsf Q^{p^n}\cong(\mathfrak b,+)$
and
$\Omega_1(\mathsf G)=\mathsf Q^{p^n}\times C_K$,
so $|\Omega_1(\mathsf G)|=p^{21}$.
Since $\dim_{\mathbb F_p}\mathfrak b=20$,
the kernel isomorphism gives
$|\ker\Upsilon_n|=p^{20\cdot21}=p^{420}$.
By Lemmas~\ref{lem:Kourovka-automorphisms}
and~\ref{lem:Kourovka-finite-quotients}, we obtain 
\[
|\operatorname{im}\Upsilon_n|
=|\operatorname{Inn}(\mathsf Q_n)|
=p^{19(n-1)}.
\]
Consequently,
\[
|\operatorname{Aut}(\mathsf G_{p,n})|
=|\ker\Upsilon_n|\,|\operatorname{im}\Upsilon_n|
=p^{420}p^{19(n-1)}
=p^{19n+401}.
\]
The equality $20n+21=19n+401$ holds when $n=380$.
Thus, setting $\mathsf G_p=\mathsf G_{p,380}$,
we obtain
\[
|\mathsf G_p|
=|\operatorname{Aut}(\mathsf G_p)|
=p^{7621}
\]
for every prime $p\geq7$.
\end{proof}
\section{Application to Kourovka Problem 20.26}\label{sec6}

For a group $G$, let $G'$ denote the derived subgroup of $G$, and let $G^i$ denote the subgroup $\langle g^i\mid g\in G\rangle$. To distinguish from Lie brackets, we write $[g,h]_G=g^{-1}h^{-1}gh$ for the commutator of elements in $G$. For a normal subgroup $K$ of $G$, write $gK$ for the element of the quotient group $G/K$ represented by $g$.

\subsection{Campbell–Hausdorff series and uniform pro-$p$ groups}\label{subsec2.1}
\ 

Let $p$ be an odd prime, and let $\mathbb{Z}_p$ denote the ring of $p$-adic integers. A Lie algebra $\mathcal{L}$ over $\mathbb{Z}_p$ is called \emph{powerful} (see~\cite[Section~9.4]{DDMS}) if $\mathcal{L}$ is free of finite rank as a $\mathbb{Z}_p$-module and $[\mathcal{L},\mathcal{L}]\subseteq p\mathcal{L}$. Let $\mathcal{L}$ be a $\mathbb{Z}_p$-Lie algebra, which is free of finite rank as a $\mathbb{Z}_p$-module. Then for any $i\geq1$, the multiplication by $p^i$ defines a $\mathbb Z_p$-module isomorphism from $\mathcal{L}$ onto $p^i\mathcal L$, and so $p^i\mathcal L$ is free of finite rank. Moreover, 
$[p^i\mathcal{L},p^i\mathcal{L}]\subseteq p^{2i}\mathcal{L}\subseteq p(p^i\mathcal{L})$ implies that $p^i\mathcal L$ is powerful, as the following lemma stated.

\begin{lemma}\label{lem3}
Let $p$ be an odd prime, and let $\mathcal{L}$ be a $\mathbb{Z}_p$-Lie algebra which is free of finite rank as a $\mathbb{Z}_p$-module. Then $p^i\mathcal{L}$ is a powerful $\mathbb{Z}_p$-Lie algebra for any integer $i\geq1$.
\end{lemma}

The following elementary result is from~\cite[Remark~3.5]{GGS2026}, and we leave its proof to the reader.

\begin{lemma}\label{lem4}
Let $\mathcal{L}$ be a $\mathbb{Z}_p$-Lie algebra. Write $\bar{x}$ for the image of $x\in\mathcal{L}$ in the quotient $\mathcal{L}/p\mathcal{L}$ (of $\mathbb{Z}_p$-module or Lie algebra). Then $\mathcal{L}/p\mathcal{L}$ is a $\mathbb{F}_p$-Lie algebra under the operations
\[
\bar x+\bar y=\overline{x+y},\qquad
\bar\lambda\,\bar x=\overline{\lambda x},\qquad
[\bar x,\bar y]=\overline{[x,y]},
\]
for all $x,y\in\mathcal{L}$ and $\lambda\in\mathbb{Z}_p$, where $\bar\lambda=\lambda+p\mathbb{Z}_p$.
\end{lemma}

A pro-$p$ group $U$ is said to be \emph{uniform} if $U$ is torsion free, topologically finitely generated, and $U'$ is contained in the closure of $U^{2p}$. Now we use the Campbell–Hausdorff formula to equip a powerful $\mathbb{Z}_p$-Lie algebra with the structure of a uniform pro-$p$ group.
Let $\mathcal{L}$ be a powerful $\mathbb{Z}_p$-Lie algebra with Lie bracket $[\,,\,]$. For a tuple $e=(e_1,\ldots,e_s)$ of positive integers, let $\langle e\rangle=e_1+\cdots+e_s$. Using left-normed brackets, for any $x,y\in\mathcal{L}$, define
\[
[x,y]_e=
[x,\underbrace{y,\ldots,y}_{e_1},
\underbrace{x,\ldots,x}_{e_2},
\underbrace{y,\ldots,y}_{e_3},
\underbrace{x,\ldots,x}_{e_4},\ldots],
\]
where after the initial $x$, there are $s$ alternating blocks of $y$'s and $x$'s of length $e_1,\ldots,e_s$, respectively. By~\cite[Theorem~9.8]{DDMS}, the \emph{Campbell–Hausdorff formula} defines a uniform pro-$p$ group in $\mathcal{L}$ as follows: 
\begin{equation}\label{EQ1}
\star:\ 
\begin{array}{ccc}
\mathcal{L}\times\mathcal{L} &\to& \mathcal{L}\\
(x,y)&\mapsto & x\star y
\end{array},\ \ 
\text{where } 
x\star y=\sum_{n=1}^\infty u_n(x,y)
\end{equation}
and
\[
u_1(x,y)=x+y,\ \ 
u_2(x,y)=\frac{1}{2}[x,y],\ \ 
u_n(x,y)=\sum_{\langle e\rangle=n-1}q_e[x,y]_e
\text{ \ for $n\geq3$}.
\]
Moreover, each $q_e\in\mathbb{Q}_p$ satisfies $p^{\langle e\rangle}q_e\in p\mathbb{Z}_p$ and that $|p^{\langle e\rangle}q_e|$ tends to $0$ as $\langle e\rangle$ tends to $\infty$. The following result reveals that Campbell–Hausdorff formula defines a uniform pro-$p$ group in a powerful $\mathbb{Z}_p$-Lie algebra.

\begin{lemma}\label{lem8}
Let $p$ be an odd prime, and let $\mathcal{L}$ be a $\mathbb{Z}_p$-Lie algebra which is free of rank $d>0$ as a $\mathbb{Z}_p$-module. Then the following statements hold:
\begin{enumerate}[{\rm(a)}]
\item (\cite[Theorem~\textup{9.8}]{DDMS}) $(p\mathcal{L},\star)$ is a uniform pro-$p$ group $U$.
\item\label{lem7.2} For each $i\geq 0$, $(p^{i+1}\mathcal{L},\star)$ is a characteristic subgroup $U_i$ of $U$, and $U_i=U^{p^i}$. Moreover, $U/U_i$ is a $p$-group of order $p^{di}$.
\end{enumerate}
\end{lemma}

\begin{proof}
For any $u\in U$, it follows from~\eqref{EQ1} that $u\star u=2u$. Thus, $u^{\star i}=iu$ and so
\[
p^{i+1}\mathcal{L}=\{u^{\star p^i}\mid u\in U\}\text{ for any integer $i\geq0$}.
\]
Since $p^{i+1}\mathcal{L}$ is powerful by Lemma~\ref{lem3}, $(p^{i+1}\mathcal{L},\star)$ is a subgroup of $(p\mathcal{L},\star)$. This implies
\[
U_i:=(p^{i+1}\mathcal{L},\star)=
\{u^{\star p^i}\mid u\in U\}
=\langle u^{\star p^i}\mid u\in U\rangle
=U^{p^i}.
\]
Let $\alpha\in\Aut(U)$. Since $\alpha(U_i)=\langle\alpha(u)^{\star p^i}\mid u\in U\rangle=U_i$, the group $U_i$ is a characteristic subgroup of $U$. Let $u,v\in U_j$ for some $j\geq0$. Observing that $[u,v]\in p^{j+2}\mathcal{L}$, we derive from~\eqref{EQ1} that $(uU_{j+1})\star(vU_{j+1})=(u+v)U_{j+1}$. Hence, there is a group-isomorphism from $(U_j/U_{j+1},\star)$ to the abelian group $(p^{j+1}\mathcal{L}/p^{j+2}\mathcal{L},+)$. In particular, $|U_j/U_{j+1}|=p^d$, which completes the proof of part~\eqref{lem7.2}.
\end{proof}

\begin{construction}\label{con1}
Let $(\mathcal{L},[\,,\,])$ be a $\mathbb{Z}_p$-Lie algebra which is free of rank $d>0$ as a $\mathbb{Z}_p$-module, where $p$ is an odd prime. For $i\geq0$, write $\mathcal{L}_i=p^{i+1}\mathcal{L}$, and let $\mathfrak{P}=\mathcal{L}/\mathcal{L}_0$ be the Lie algebra over $\mathbb{F}_p$. Let $U=(\mathcal{L}_0,\star)$ be
the uniform pro-$p$ group, and let $U_i=(\mathcal{L}_i,\star)$, where $i\geq 0$. Fix an integer $m\geq2$, define the finite $p$-group $P=U/U_m$ of order $p^{dm}$. Let $P_i=U_i/U_m$ for each integer $0\leq i\leq m$. 
\end{construction}

Let $a,b\in\mathcal{L}$, and let $0\leq i\leq m-1$. 
We prove the following two equations in~\cite{GOR}:
\begin{equation}\label{EQ2}
(p^{i+1}x)\star(p^{i+1}y)\in p^{i+1}(x+y)+\mathcal{L}_{i+1}
\end{equation}
and
\begin{equation}\label{EQ6}
[p^{i+1}x,py]_U\in p^{i+2}[x,y]+\mathcal{L}_{i+2}\subseteq\mathcal{L}_{i+1}.
\end{equation}

\begin{figure}[htbp]
\centering
\begin{tikzpicture}
[
>=Latex,line width=0.6pt,
every node/.style={inner sep=2pt,outer sep=0pt},
element/.style={text height=2.4ex,text depth=0.8ex},
arrow/.style={->, >=Stealth, thick}
]
\node[element] (A) 
{\ \ $P_0/P_1$\ \ \ };
\node[element,right=12mm of A] (B) 
{\ \ $P_1/P_2$\ \ \ };
\node[element,right=12mm of B] (D) 
{\ \ \ $\cdots$\ \ \ };
\node[element,right=12mm of D] (C)
{\ \ $P_{m-1}/P_m$\ \ \ };

\draw[arrow] (A.east) -- node[above] {$\iota_0$} (B.west);
\draw[arrow] (B.east) -- node[above] {$\iota_1$} (D.west);
\draw[arrow] (D.east) -- node[above] {$\iota_{m-2}$} (C.west);

\coordinate (mid) at ($(A.west)!0.5!(C.east)$);
\node (X) at ($(mid)+(0,2)$){$\mathfrak{P}$};

\draw[arrow]($(X.south west)+(-0.15,-0.15)$)--($(A.north)+(1,0.18)$);
\draw[arrow]($(X.south)+(0,-0.15)$)--($(B.north)+(0.5,0.18)$);
\draw[arrow]($(X.south east)+(0.1,-0.15)$)--($(C.north)+(-1,0.18)$);

\node at ($(X)!0.5!(A)+(-0.3,0.3)$){$\phi_0$};
\node at ($(X)!0.5!(B)+(0.7,0.1)$){$\phi_1$};
\node at ($(X)!0.5!(B)+(1.5,0.1)$){$\cdots$};
\node at ($(X)!0.5!(C)+(0.35,0.3)$){$\phi_{m-1}$};
\end{tikzpicture}
\caption{Isomorphisms between central factors of $P$ and $\mathfrak{P}$}
\label{fig1}
\end{figure}

\begin{lemma}\label{LEM1}
Adopt the notation as in Construction~\textup{\ref{con1}}. Let $\pi\colon U\to P$ be the natural homomorphism. For $0\leq i\leq m-1$, let
\begin{equation}\label{EQ3}
\phi_i:\ 
\begin{array}{ccc}
\mathfrak{P} &\to& P_i/P_{i+1}\\
x+\mathcal{L}_0&\mapsto & \pi(p^{i+1}x)P_{i+1}
\end{array},\ \ 
\text{where }
x\in\mathcal{L}.
\end{equation}
For $0\leq j\leq m-2$, let
\begin{equation}\label{EQ5}
\iota_j:\ 
\begin{array}{ccc}
P_{j}/P_{j+1} &\to& P_{j+1}/P_{j+2}\\
uP_{j+1}&\mapsto & u^pP_{j+2}
\end{array},\ \ 
\text{where }
u\in P_j.
\end{equation}
Then these group-isomorphisms give the commutative diagram Figure~\textup{\ref{fig1}}.
\end{lemma}

\begin{proof}
Write $\bar{x}$ for the image of $x\in\mathcal{L}$ in the quotient $\mathfrak{P}$. For $0\leq i\leq m-1$, define $f_i\colon(U_i,\star)\to(\mathfrak{P},+)$ by letting $f(p^{i+1}x)=\bar{x}$ for any $x\in\mathcal{L}$. Let $x,y\in\mathcal{L}$. Then it follows from~\eqref{EQ2} that there exists $z\in\mathcal{L}$ such that
\[
f_i(p^{i+1}x\star p^{i+1}y)
=f_i(p^{i+1}(x+y+pz))
=\overline{x+y+pz}
=\bar{x}+\bar{y}
=f_i(p^{i+1}x)+f_i(p^{i+1}y).
\]
Moreover, since $\bar{x}=0$ is equivalent to $x\in\mathcal{L}_0$ and so to $p^{i+1}x\in U_{i+1}$, the mapping $f_i$ is a group-homomorphism with kernel $U_{i+1}$. Note that $U_{i}/U_{i+1}\cong P_{i}/P_{i+1}$. We obtain the isomorphisms
\begin{equation*}
\begin{array}{ccccc}
P_i/P_{i+1} & \to & 
U_i/U_{i+1} & \to & \mathfrak{P}\\
\pi(p^{i+1}x)P_{i+1} & \mapsto & 
p^{i+1}xU_{i+1} & \mapsto & 
\bar{x}
\end{array},\ \ 
\text{where }
x\in\mathcal{L}.
\end{equation*}
The inverse of this isomorphism is exactly $\phi_i$ in~\eqref{EQ3}.

Let $0\leq j\leq m-2$, and take any $u\in P_j$. Since $P_j=\pi(U_j)$, there exists $x\in\mathcal{L}$ such that $u=\pi(p^{j+1}x)$. Moreover,
\[
(\phi_{j+1}\circ\phi_j^{-1})(uP_{j+1})
=\phi_{j+1}\big(\phi_j^{-1}(\pi(p^{j+1}x)P_{j+1})\big)
=\phi_{j+1}(\bar{x})
=\pi(p^{i+2}x)P_{i+2}.
\]
Recall that $y\star y=2y$ implies $y^{\star p}=py$. It follows that $u^p=\pi((p^{j+1}x)^{\star p})=\pi(p^{i+2}x)$. By the arbitrariness of $u$, we conclude that each $\iota_j=\phi_{j+1}\circ\phi_j^{-1}$, for any $0\leq j\leq m-2$, is a group-isomorphism, which proves the lemma.
\end{proof}

\begin{lemma}\label{lem7}
Adopt the notation as in Construction~\textup{\ref{con1}}. Then
\[
P=P_0\geq P_1\geq\cdots\geq P_m=1
\]
is a central series of $P$. Moreover, this is the lower central series if and only if $\mathfrak{P}$ as a $\mathbb{F}_p$-Lie algebra is perfect.
\end{lemma}

\begin{proof}
Let $0\leq i\leq m-1$. Then it follows from~\eqref{EQ6} that $[U_i,U]_U\subseteq U_{i+1}$ and so $[P_i,P]_P\subseteq P_{i+1}$. This means that the series is a central series of $P$. Let $x,y\in\mathcal{L}$ and $0\leq j\leq m-2$. Then by~\eqref{EQ2},
\[
(p^{j+2}[x,y]+\mathcal{L}_{j+2})\star(-p^{j+2}[x,y])\subseteq U_{j+2}.
\]
Recall that the inverse of $u\in(U,\star)$ is $-u$. Applying $\pi$ therefore yields
\[
\pi\big(p^{j+2}[x,y]+\mathcal{L}_{j+2}\big)
\subseteq
\pi\big(p^{j+2}[x,y]\big)P_{j+2}.
\]
Combining this with equation~\eqref{EQ6}, we obtain
\[
[\pi(p^{j+1}x),\pi(py)]_P
=\pi([p^{j+1}a,py]_U)
\in\pi(p^{j+2}[x,y]+\mathcal{L}_{j+2})
\subseteq\pi\big(p^{j+2}[x,y]\big)P_{j+2}.
\]
Let $\phi_i$ be the map in~\eqref{EQ3}. Then it follows that
\begin{equation}\label{EQ7}
\phi_{j+1}\big([x+\mathcal{L}_0,y+\mathcal{L}_0]\big)
=\phi_{j+1}\big([x,y]+\mathcal{L}_0\big)
=\pi\big(p^{j+2}[x,y]\big)P_{j+2}
=[\pi(p^{j+1}x),\pi(py)]_PP_{j+2}.
\end{equation}

Suppose that $\mathfrak{P}$ is perfect. Then the equation~\eqref{EQ7} gives that
\[
P_{j+1}/P_{j+2}
=\phi_{j+1}(\mathfrak{P})
=\phi_{j+1}([\mathfrak{P},\mathfrak{P}])
=\{uP_{j+2}\mid u\in[P_j,P]_P\},
\]
and so
$P_{j+1}=[P_j,P]_P\star P_{j+2}$ for any $0\leq j\leq m-2$. It is clear that $[P_{m-1},P]_P=P_m=1$. Suppose inductively that $[P_{k-1},P]_P=P_k$ for some $2\leq k\leq m$. Since $P_{k-1}\leq P_{k-2}$, we have
\[
P_{k-1}=[P_{k-2},P]_P\star P_k
=[P_{k-2},P]_P\star[P_{k-1},P]_P
=[P_{k-2},P]_P.
\]
By descending induction, $[P_{k-1},P]_P=P_k$ for all $1\leq k\leq m$. Hence, the series in the lemma is the lower central series of $P$.

Conversely, let the series in the lemma be the lower central series of $P$. Then $[P,P]_P=P_1$, and equation~\eqref{EQ7} yields
\[
\phi_1([\mathfrak{P},\mathfrak{P}])
=\{uP_2\mid u\in[P,P]_P\}=P_1/P_2.
\]
Since $\phi_1$ is an isomorphism, we conclude $[\mathfrak{P},\mathfrak{P}]=\mathcal{L}_0$. This means that $\mathfrak{P}$ is perfect, which completes the proof.
\end{proof}

\subsection{An upper bound of order of the automorphism group of $P$}
\ 

Let $\mathcal{L}$ be a $\mathbb{Z}_p$-Lie algebra, let $p\mathcal{L}$ be a powerful $\mathbb{Z}_p$-Lie algebra as in subsection~\ref{subsec3.2}. Fix $m\geq2$, and let $P$ be a $p$-group as in Lemma~\ref{lem7}. In this subsection, we investigate the automorphism group of $P$.

\begin{lemma}
Adopt the notation as in Construction~\textup{\ref{con1}}. Let $\pi\colon U\to P$ be the natural homomorphism, and let $\phi_1,\ldots,\phi_{m-1}$ be as in~\textup{\eqref{EQ3}}. Then the following statements hold:
\begin{enumerate}[{\rm(a)}]
\item Let $\alpha\in\Aut(P)$. Then $P_i$ is $\alpha$-invariant for each $0\leq i\leq m$. Let $\alpha_i$ be the group-automorphism of $P_i/P_{i+1}$ induced by $\alpha$. Then the composition $\eta_i(\alpha):=\phi_i^{-1}\alpha_i\phi_i$ is a group-automorphism of $\mathfrak{P}$. Moreover, $\eta(\alpha):=\eta_0(\alpha)=\cdots=\eta_{m-1}(\alpha)$.
\item The map $\eta$ gives a group-homomorphism
\begin{equation}\label{EQ4}
\eta:\ 
\begin{array}{ccc}
\Aut(P) &\to& \Aut_{\textup{Lie}}(\mathfrak{P})\\
\alpha&\mapsto & \eta(\alpha)
\end{array},\ \ 
\text{where }
\alpha\in\Aut(P).
\end{equation}
\end{enumerate}
\end{lemma}

\begin{proof}
Let $\alpha\in\Aut(P)$. For any $0\leq i\leq m$, since $U_i=U^{p^i}$ by Lemma~\ref{lem8}, we obtain
\[
P_i=U_i/U_m
=\langle u^{p^i}U_m\mid u\in U\rangle
=\langle (uU_m)^{p^i}\mid u\in U\rangle
=\langle g^{p^i}\mid g\in P\rangle.
\]
Thus, $P_i$ is $\alpha$-invariant. Let $\alpha_i\colon gP_{i+1}\mapsto\alpha(g)P_{i+1}$ be the automorphism of $P_i/P_{i+1}$ induced by $\alpha$. Since $\phi_{i}$ is an isomorphism, we have $\eta_i(\alpha)\in\Aut(\mathfrak{P})$. Let $\iota_0,\ldots,\iota_{m-2}$ be as in~\textup{\eqref{EQ5}}. Then the equation, for each $g\in P_i$, 
\[
\alpha_{i+1}\iota_i(gP_{j+1})
=\alpha_{i+1}(g^pP_{j+2})
=(\alpha_{i+1}(g))^pP_{j+2}
=\iota_i\alpha_i(gP_{j+1})
\]
yields that $\iota_i^{-1}\alpha_{i+1}\iota_i=\alpha_i$. Combining this with the commutative diagram Figure~\ref{fig1}, we conclude that for any $0\leq i\leq m-2$,
\[
\eta_{i+1}(\alpha)
=\phi_{i+1}^{-1}\alpha_{i+1}\phi_{i+1}
=\phi_i^{-1}\iota_i^{-1}\alpha_{i+1}\iota_i\phi_i
=\phi_i^{-1}\alpha_i\phi_i
=\eta_i(\alpha).
\]
This proves the statements in part~(a).

Write $\bar{x}$ for the image of $x\in\mathcal{L}$ in the quotient $\mathfrak{P}$. Fix $\alpha\in\Aut(P)$. We first show that $\theta:=\eta(\alpha)$ preserves the Lie bracket of $\mathfrak{P}$. Since $\phi_1$ is an isomorphism and $\phi_1\theta=\alpha_1\phi_1$, this is equivalent to show
\begin{equation}\label{EQ9}
\alpha_1\phi_1[\bar{x},\bar{y}]
=\phi_1[\theta(\bar{x}),\theta(\bar{y})],
\text{ for any $x,y\in\mathcal{L}$.}
\end{equation}
Let $x,y\in\mathcal{L}$, and write $X=\pi(px)$ and $Y=\pi(py)$. Since $\pi$ is surjective, there exists $a,b\in\mathcal{L}$ such that
\[
\alpha(X)=\pi(pa)
\text{ \ and \ }
\alpha(Y)=\pi(pb).
\]
Together with equation~\eqref{EQ7}, this implies 
\[
\alpha_1\phi_1[\bar{x},\bar{y}]
=\alpha_1\big([X,Y]_PP_2\big)
=[\alpha X,\alpha Y]_PP_2
=[\pi(pa),\pi(pb)]_PP_2.
\]
Moreover, applying the conclusion in part~(a), we deduce that
\[
\phi_0(\bar{a})
=\pi(pa)P_1
=\alpha(X)P_1
=\alpha_0(XP_1)
=\alpha_0(\phi_0(\bar{x}))
=\phi_0(\theta(\bar{x})).
\]
Since $\phi_0$ is injective, this means $\bar{a}=\theta(\bar{x})$. By a similar argument, we have $\bar{b}=\theta(\bar{y})$. Hence, using equation~\eqref{EQ7} again, we obtain
\[
\phi_1[\theta(\bar{x}),\theta(\bar{y})]
=\phi_1[\bar{a},\bar{b}]
=[\pi(pa),\pi(pb)]_PP_2.
\]
Consequently, statement~\eqref{EQ9} holds. Since $\eta(\alpha)\in\Aut(\mathfrak{P})$ keeps the linear relationship, we conclude $\eta(\alpha)\in\Aut_{\textup{Lie}}(\mathfrak{P})$. Let $\alpha,\beta\in\Aut(P)$. Then $\eta(\alpha\beta)=\phi_{0}^{-1}\alpha_0\beta_0\phi_0=\eta(\alpha)\eta(\beta)$, which means that $\eta$ is a group-homomorphism. This proves part~(b) of the lemma.
\end{proof}

\begin{lemma}
Adopt the notation as in Construction~\textup{\ref{con1}}. For $1\leq i\leq m$, let $H_i$ be the kernel of the group-homomorphism $\Aut(P)\to\Aut(P/P_i)$. Then for any $1\leq i\leq m-2$, there exists a group-homomorphism from $H_i$ to $(\Der(\mathfrak{P}),+)$ with kernel $H_{i+1}$.
\end{lemma}

\begin{proof}
Fix $1\leq i\leq m-2$. Throughout this proof, write $u \equiv v$ if $uP_{i+2}=vP_{i+2}$, and write $\bar{x}$ for the image of $x\in\mathcal{L}$ in the quotient $\mathfrak{P}$. Let $\pi\colon U\to P$ be the natural homomorphism, and let $\alpha\in H_i$. By definition, it follows that $g^{-1}\star\alpha(g)\in P_i$ for any $g\in P$. Define
\[
d_\alpha:\ 
\begin{array}{ccc}
P/P_1 &\to& P_i/P_{i+1}\\
gP_1 &\mapsto &f(g)P_{i+1}
\end{array},\ \ 
\text{where $f(g):=g^{-1}\star\alpha(g)$ for any $g\in P$.}
\]

\noindent\textbf{Claim~1.} $d_\alpha$ is a well-defined group-homomorphism.
\vspace{0.5em}

Let $g,h\in P$. Then 
\[
f(gh)=(gh)^{-1}\alpha(gh)
=h^{-1}f(g)hf(h)
=[h,f(g)^{-1}]_Pf(g)f(h).
\]
Recall from Lemma~\ref{lem7} that $[P,P_i]_P\leq P_{i+1}$. The equation above yields that $f(gh)P_{i+1}=f(g)f(h)P_{i+1}$. Now take an arbitrary $y\in P_1$. Since $P_1=\pi(U_1)$, there exists $u\in\mathcal{L}$ such that $y=\pi(p^2u)=\pi((pu)^{\star p})=\pi(pu)^{\star p}$. Thus,
\[
f(y)
=f\big(\pi(pu)^{\star p}\big)
=f(\pi(pu))^{\star p}\in (P_i)^p\leq P_{i+1}.
\]
Take any $x\in P$. Then $f(xy)P_{i+1}=f(x)f(y)P_{i+1}=f(x)P_{i+1}$, which implies that $d_\alpha$ is well-defined. Moreover, one may directly verify the homomorphism by $f(gh)P_{i+1}=f(g)f(h)P_{i+1}$ for any $g,h\in P$. \qed

\vspace{0.5em}

\noindent\textbf{Claim~2.} Let $D_\alpha=\phi_i^{-1}d_\alpha\phi_0$. Then for any $x,y\in\mathcal{L}$, 
\[
\phi_{i+1}D_\alpha[\bar{x},\bar{y}]=f\big([X,Y]_P\big)P_{i+2},
\] 
where $X=\pi(px)$ and $Y=\pi(py)$.

\vspace{0.5em}

Let $x,y,X,Y$ as in the claim. Since $[P,P]_P\leq P_1$, there exists $z\in\mathcal{L}$ such that $[X,Y]_P=\pi(p^2z)$. According to equation~\eqref{EQ7}, we have $\phi_1[\bar{x},\bar{y}]=\pi(p^2z)P_2=\phi_1(\bar{z})$, and so $\bar{z}=[\bar{x},\bar{y}]$. Write $Z=\pi(pz)$. Then $Z^{\star p}=[X,Y]_P$ and
\[
\phi_iD_\alpha[\bar{x},\bar{y}]
=\phi_iD_\alpha(\bar{z})
=d_\alpha\phi_0(\bar{z})
=f(Z)P_{i+1}.
\]
Let $\iota_i$ be as in~\eqref{EQ5}. Then by Lemma~\ref{LEM1},
\begin{equation}\label{EQ8}
f(Z)^pP_{i+2}
=\iota_i(f(Z)P_{i+1})
=\iota_i\phi_i(D_\alpha[\bar{x},\bar{y}])
=\phi_{i+1}(D_\alpha[\bar{x},\bar{y}]).
\end{equation}

Note that $f(Z)\in P_i$ and $[P_{i+1},P]_P\leq P_{i+2}$. The image of the element $[f(Z),Z]_P\in P_{i+1}$ in $P/P_{i+2}$ is central. Consequently, it follows from $f(Z)Z=Zf(Z)[f(Z),Z]$ that
\[
(Zf(Z))^p
\equiv Z^pf(Z)^p([f(Z),Z]_P)^{\binom{p}{2}}
\equiv Z^pf(Z)^p,
\]
where the last equality follows from the exponent $p$ of the group $P_{i+1}/P_{i+2}$. Recall that $Z^{\star p}=[X,Y]_P$. We further obtain
\[
f([X,Y]_P)
=(Z^p)^{-1}\alpha(Z^p)
=(Z^p)^{-1}\alpha(Z)^p
=(Z^p)^{-1}(Zf(Z))^p
\equiv f(Z)^p.
\]
Combining this with~\eqref{EQ8}, we prove Claim~2. \qed

\vspace{0.5em}

\noindent\textbf{Claim~3.} $D_\alpha=\phi_i^{-1}d_\alpha\phi_0$ is a derivation of $\mathfrak{P}$.

\vspace{0.5em}

For $u\in P_i$ and $v\in P$, Lemmas~\ref{lem7} implies that the image of $[u,v]_P$ in $P/P_{i+2}$ is central. Hence, $[u,v^p]_P\equiv[u,v]_P^p\equiv1$. Since $U^p\leq U_1$,
this proves $[P_i,P_1]_P\leq P_{i+2}$. Let $x,y\in\mathcal L$, and write $X=\pi(px)$ and $Y=\pi(py)$. Since $f(X),f(Y)\in P_i\leq P_1$, we obtain $[f(X),f(Y)]_P\equiv1$.
Also, $[X,Y]_P\in P_1$, so its image commutes with the images
of $f(X)$ and $f(Y)$ modulo $P_{i+2}$.
Moreover, the images of $[f(X),Y]_P$ and $[X,f(Y)]_P$
are central, since
\[
[P_i,P]_P\leq P_{i+1},
\qquad
[P_{i+1},P]_P\leq P_{i+2}.
\]
Applying the commutator identities therefore yields
\[
\begin{aligned}
\alpha([X,Y]_P)
&=[Xf(X),Yf(Y)]_P\\
&=[X,f(Y)]_P^{\,f(X)}[f(X),f(Y)]_P
  \bigl([X,Y]_P^{\,f(X)}[f(X),Y]_P\bigr)^{f(Y)}\\
&\equiv [X,Y]_P[f(X),Y]_P[X,f(Y)]_P.
\end{aligned}
\]
Consequently,
\begin{equation}\label{eq:claim3-displacement}
f([X,Y]_P)
=[X,Y]_P^{-1}\alpha([X,Y]_P)
\equiv [f(X),Y]_P[X,f(Y)]_P.
\end{equation}

By the definition of $D_\alpha$,
\[
\phi_i(D_\alpha(\bar x))=f(X)P_{i+1},
\qquad
\phi_i(D_\alpha(\bar y))=f(Y)P_{i+1}.
\]
The inclusion $[P_i,P_1]_P\leq P_{i+2}$ shows that, for $u\in P_i$
and $v\in P$, the coset $[u,v]_PP_{i+2}$
depends only on $uP_{i+1}$ and $vP_1$.
Thus, equation~\eqref{EQ7} gives
\[
\phi_{i+1}([D_\alpha(\bar x),\bar y])=[f(X),Y]_PP_{i+2}
\text{ \ and \ }
\phi_{i+1}([\bar x,D_\alpha(\bar y)])=[X,f(Y)]_PP_{i+2}.
\]
The second equality follows by interchanging the arguments
in equation~(6), taking inverses, and using the antisymmetry
of the Lie bracket. Since $\phi_{i+1}$ is a group-isomorphism from the additive
group of $\mathcal L/p\mathcal L$,
these equations and~\eqref{eq:claim3-displacement} imply
\[
\phi_{i+1}\bigl([D_\alpha(\bar x),\bar y]+[\bar x,D_\alpha(\bar y)]\bigr)
=[f(X),Y]_P[X,f(Y)]_PP_{i+2}
=f([X,Y]_P)P_{i+2}.
\]
This proves Claim~3.
\end{proof}

\begin{proposition}\label{PROP1}
Adopt the notation as in Construction~\textup{\ref{con1}}. Let $n=\dim_{\mathbb{F}_p}\mathfrak{P}$, let $z=\dim_{\mathbb{F}_p}Z(\mathfrak{P})$, and let $\Der(\mathfrak{P})=\ad(\mathfrak{P})$. Then
\[
|P|=p^{mn}
\text{ \ and \ }
|\Aut(P)|<p^{2n^2+(m-2)(n-z)}.
\]
\end{proposition}

\subsection{A family of $\mathbb{F}_p$-Lie algebras with only inner derivations}

\begin{notation}\label{not1}
Let $\mathcal{L}=\mathfrak{sl}_2\ltimes\mathfrak{n}$ be a Lie algebra over a field $k\in\{\mathbb{Q},\mathbb{F}_p\}$, where $p\geq7$. Let 
\[
\mathfrak{n}=
X\oplus A\oplus U\oplus Y\oplus Q\oplus kt,
\]
where $X\cong A\cong V_1$, $U\cong V_4$, $Y\cong V_3$ and $Q\cong V_2$, and let the brackets be defined by~\textup{(17)}.
\end{notation}

\begin{construction}\label{def1}
Fix a positive integer $r$. For each integer $1\leq i\leq r$, let 
\[
\mathfrak{n}_i=
X_i\oplus A_i\oplus U_i\oplus Y_i\oplus Q_i\oplus kt_i
\]
be a copy of the radical $\mathfrak{n}$ in Notation~\textup{\ref{not1}} with bracket $[\,,\,]_i$. Let
\[
\mathfrak{n}=
\left(\bigoplus_{i=1}^r\mathfrak{n}_i\right)\oplus\mathfrak{c},
\text{ where }
\mathfrak{c}=
\bigoplus_{1\leq i<j\leq r} kc_{ij}.
\]
Equip $\mathfrak{n}$ with the skew-symmetric and $k$-bilinear bracket $[\,,\,]$ determined by the following rules:
\begingroup
\setlength{\leftmargini}{2em}
\begin{itemize}
\item For $1\leq i\leq r$ and $u,v\in\mathfrak{n}_i$, set $[u,v]=[u,v]_i$;
\item Set $[c_{ij},*]=0$ for any $1\leq i<j\leq r$;
\item Let $u_i,v_j$ be copies of $u,v\in V_4$ in $U_i,U_j$, respectively. For $i<j$, the only nonzero brackets between the module components of $\mathfrak{n}_i$ and components of $\mathfrak{n}_j$ are $[u_i,v_j]=P_4(u,v)c_{ij}$.
\end{itemize}
\endgroup
Let $\mathcal{L}=\mathfrak{sl}_2(k)\ltimes\mathfrak{n}$, where $\mathfrak{sl}_2(k)$ acts on each $\mathfrak{n}_i$ as in Notation~\textup{\ref{not1}} and trivially on each $kc_{ij}$.
\end{construction}

\begin{lemma}\label{lem1}
Construction~\textup{\ref{def1}} gives a perfect Lie algebra $(\mathcal{L},[\,,\,])$. Moreover,
\[
\mathfrak{n}^2=
\left(\bigoplus_{i=1}^r(Y_i\oplus Q_i\oplus kt_i)\right)\oplus\mathfrak{c},\ \ 
\mathfrak{n}^3=\bigoplus_{i=1}^rQ_i,\ \ 
\mathfrak{n}^4=0, \text{ and } Z(\mathcal{L})=\left(\bigoplus_{i=1}^rkt_i\right)\oplus\mathfrak{c}.
\]
\end{lemma}

\begin{proof}
We first verify the Jacobi identity on $\mathfrak{n}$. By definition, it holds for triples contained in a single $\mathfrak{n}_i$. Suppose that $i,j,\ell$ are not all equal. If $i=j\neq\ell$, then since $\mathfrak{n}_i^2$ has no $U_i$-component, we have $[[\mathfrak{n}_i,\mathfrak{n}_j],\mathfrak{n}_\ell]=0$. If $i\neq j$, then $[\mathfrak{n}_i,\mathfrak{n}_j]=0$. These imply that every term in the Jacobi identity is zero. Moreover, since $\mathfrak{c}$ is in the center, $\mathfrak{n}$ is a Lie algebra.

The action of $\mathfrak{sl}_2(k)$ preserves the bracket on each $\mathfrak n_i$. For the cross brackets, the
$\mathfrak{sl}_2(k)$-invariance of $P_4$ gives
\[
P_4(s\cdot u,v)+P_4(u,s\cdot v)=0,
\text{ for any }
s\in\mathfrak{sl}_2(k)
\text{ and }
u,v\in V_4.
\]
Since $\mathfrak{sl}_2(k)$ acts trivially on $\mathfrak{c}$, this is precisely the derivation identity for the cross brackets. Hence $\mathfrak{sl}_2(k)$ acts on $\mathfrak{n}$ by derivations, and $\mathcal{L}=\mathfrak{sl}_2(k)\ltimes\mathfrak{n}$ is a Lie algebra.

Each $\mathfrak{sl}_2\ltimes\mathfrak{n}_i$ is a perfect subalgebra of $\mathcal{L}$, and hence is contained in $[\mathcal{L},\mathcal{L}]$. Since $P_4$ is nonzero, we have $kc_{ij}=[U_i,U_j]\subseteq[\mathcal{L},\mathcal{L}]$ for any $i<j$. These subspaces span $\mathcal{L}$, and so $\mathcal{L}$ is perfect.

Recall that each $\mathfrak s\ltimes\mathfrak n_i$ is perfect, with
\[
\mathfrak n_i^2=Y_i\oplus Q_i\oplus kt_i,
\ \ 
\mathfrak n_i^3=Q_i,
\ \ 
\mathfrak n_i^4=0
\ \text{ and }\ 
Z(\mathfrak s\ltimes\mathfrak n_i)=kt_i.
\]
The defining bracket relations therefore yield
\[
\mathfrak n^2
=
\left(\bigoplus_{i=1}^r\mathfrak n_i^2\right)
\oplus\mathfrak c
=
\left(\bigoplus_{i=1}^r
(Y_i\oplus Q_i\oplus kt_i)\right)
\oplus\mathfrak c.
\]
Using the centrality of $\mathfrak c$ and
$[\mathfrak n_i^2,\mathfrak n_j]=0$ for $i\ne j$,
we obtain
\[
\mathfrak n^3
=
\bigoplus_{i=1}^r\mathfrak n_i^3
=
\bigoplus_{i=1}^r Q_i,
\qquad
\mathfrak n^4=0.
\]
Finally, the construction gives
\[
\left(\bigoplus_{i=1}^r kt_i\right)
\oplus\mathfrak c
\subseteq Z(\mathcal L).
\]
For the reverse inclusion, for each $i$ consider
the surjective Lie algebra homomorphism
\[
\pi_i:\mathcal L\longrightarrow
\mathfrak s\ltimes\mathfrak n_i,
\qquad
\pi_i\left(s+\sum_{j=1}^r v_j+c\right)=s+v_i.
\]
It is a homomorphism because all cross brackets
lie in $\mathfrak c$, which $\pi_i$ annihilates.
If
\[
z=s+\sum_{j=1}^r v_j+c\in Z(\mathcal L),
\]
then surjectivity gives
\[
s+v_i=\pi_i(z)
\in Z(\mathfrak s\ltimes\mathfrak n_i)=kt_i
\]
for every $i$.
Thus $s=0$ and $v_i\in kt_i$ for every $i$,
proving the asserted formula for $Z(\mathcal L)$.
\end{proof}

\begin{lemma}\label{lem2}
Let $\mathcal{L}$ be a Lie algebra given in Construction~\textup{\ref{def1}}. Then $\Der(\mathcal{L})=\ad(\mathcal{L})$.
\end{lemma}

\begin{proof}
Put $\mathfrak s=\mathfrak{sl}_2(k)$ and $B_i=\mathfrak s\ltimes\mathfrak n_i$ for $1\leq i\leq r$. By Notation~\ref{not1}, $\operatorname{Der}(B_i)=\operatorname{ad}(B_i)$. Moreover, the $\mathfrak s$-module decomposition of $B_i$ gives $C_{B_i}(\mathfrak s)=kt_i\subseteq Z(B_i)$. Consequently, any derivation of $B_i$ that vanishes on $\mathfrak s$ is zero.

For each $i$, let $\pi_i:\mathcal L\to B_i$ be the natural surjective Lie algebra homomorphism defined by $\pi_i(s+\sum_j v_j+c)=s+v_i$. Take $D_0\in\operatorname{Der}(\mathcal L)$. Since $\pi_i|_{B_i}=\operatorname{id}_{B_i}$, the derivation identity shows that $\delta_i=\pi_i\circ D_0|_{B_i}$ belongs to $\operatorname{Der}(B_i)$. Thus, $\delta_i=\operatorname{ad}(a_i+v_i)$ for some $a_i\in\mathfrak s$ and $v_i\in\mathfrak n_i$. For $s\in\mathfrak s$, the $\mathfrak s$-component of $\pi_iD_0(s)$ is independent of $i$, so $[a_i-a_j,s]=0$ for all $i,j$. As $Z(\mathfrak s)=0$, all $a_i$ coincide with some $a\in\mathfrak s$. Set
\[
z=a+\sum_{i=1}^r v_i,
\qquad
D=D_0-\operatorname{ad}z.
\]
Then $\pi_iD|_{B_i}=\delta_i-\operatorname{ad}(\pi_i z)=0$ for every $i$. In particular, $D(\mathfrak s)\subseteq\bigcap_i\ker\pi_i=\mathfrak c$. Since $\mathfrak c$ is central and $\mathfrak s=[\mathfrak s,\mathfrak s]$, we obtain $D|_{\mathfrak s}=0$. Hence $D$ is an $\mathfrak s$-module homomorphism.

Fix $i\ne j$ and put $\varphi_{ji}=\pi_jD|_{\mathfrak n_i}$. The equality $\pi_iD|_{B_i}=0$ ensures that $D(\mathfrak n_i)$ has no $\mathfrak s$-component. Thus $\varphi_{ji}:\mathfrak n_i\to\mathfrak n_j$ is an $\mathfrak s$-module homomorphism. For $x,y\in\mathfrak n_i$, applying $\pi_j$ to $D([x,y])=[D(x),y]+[x,D(y)]$ gives $\varphi_{ji}([x,y])=0$, because $\pi_j(x)=\pi_j(y)=0$. Therefore $\varphi_{ji}(\mathfrak n_i^2)=0$.

We claim that $\varphi_{ji}(\mathfrak n_i)\subseteq Z(\mathfrak n_j)$. By Lemma~\ref{lem1}, $\mathfrak n_i=(X_i\oplus A_i\oplus U_i)\oplus\mathfrak n_i^2$, so it suffices to consider $X_i\oplus A_i\oplus U_i$. If $x\in X_i\oplus A_i$ and $y\in\mathfrak n_j$, then $[x,y]=0$ by Construction~\ref{def1}. Applying $D$ and then $\pi_j$ yields $[\varphi_{ji}(x),y]=0$. Hence, $\varphi_{ji}(X_i\oplus A_i)\subseteq Z(\mathfrak n_j)$.

Since $V_4$ is absolutely irreducible and $U_j$ is the only summand of $\mathfrak n_j$ isomorphic to $V_4$, there exists $\alpha\in k$ such that $\varphi_{ji}(u_i)=\alpha u_j$ for every $u\in V_4$. For $y\in X_j\oplus A_j\oplus Y_j\oplus Q_j\oplus kt_j$, the equality $[u_i,y]=0$ similarly gives $[\varphi_{ji}(u_i),y]=0$. To handle brackets with $U_j$, note that the symmetry of $P_4$ implies $[u_i,v_j]=[v_i,u_j]$ for $u,v\in V_4$. Applying $D$ and then $\pi_j$, we obtain
\[
\alpha[u_j,v_j]
=\alpha[v_j,u_j]
=-\alpha[u_j,v_j].
\]
Since $\operatorname{char}k\ne2$, it follows that $[\varphi_{ji}(u_i),v_j]=0$. This proves the claim.

Let $\theta_{ij}:\mathfrak n_i\to\mathfrak n_j$ be the natural isomorphism between the two copies. Define $E:B_i\to B_i$ by $E|_{\mathfrak s}=0$ and $E|_{\mathfrak n_i}=\theta_{ij}^{-1}\circ\varphi_{ji}$. Then $E$ is an $\mathfrak s$-module homomorphism satisfying
\[
E(\mathfrak n_i^2)=0,
\qquad
E(\mathfrak n_i)\subseteq Z(\mathfrak n_i).
\]
These properties imply the derivation identity on $\mathfrak n_i$, while the $\mathfrak s$-module homomorphism property and $E|_{\mathfrak s}=0$ imply it for pairs involving $\mathfrak s$. Thus, $E\in\operatorname{Der}(B_i)$. Since $E$ vanishes on $\mathfrak s$, we have $E=0$, and consequently $\varphi_{ji}=0$.

We have now shown that $\pi_jD(\mathfrak n_i)=0$ for all $i,j$, whence $D(\mathfrak n_i)\subseteq\mathfrak c$ for every $i$. Since $\mathfrak c=\sum_{i<j}[U_i,U_j]$ is central, the derivation identity also gives $D(\mathfrak c)=0$. Together with $D|_{\mathfrak s}=0$, this yields   $D(\mathcal L)\subseteq\mathfrak c\subseteq Z(\mathcal L)$. By Lemma~\ref{lem1}, $\mathcal L$ is perfect, so the derivation identity forces $D(\mathcal L)=D([\mathcal L,\mathcal L])=0$. Therefore, $D_0=\operatorname{ad}z$, proving $\operatorname{Der}(\mathcal L)=\operatorname{ad}(\mathcal L)$.
\end{proof}

\subsection{Proof of Theorem~\ref{thm1}}\label{subsec3.2}

We complete the proof of Theorem~\ref{thm1} in~\cite{GOR}.



\

\bigskip
\noindent\textbf{Acknowledgements.} We thank Professor Evgeny Khukhro for drawing our attention to
the manuscript~\cite{Kriebel26}, which motivated our work on Problem~16.63 of the \emph{Kourovka Notebook}.
\\[3mm]
{\bf Funding} {{\ The second author was supported by Basic Research Program of Jiangsu (BK20251784).}}
\\[3mm]
{\bf Data availability} {{\ Data sharing not applicable to this article as no datasets were generated or analysed during
the current study.}}
\\[3mm]
{\bf Conflict of interest} {{\ The authors have no competing interests to declare that are relevant to the content of this
article.}}
\\[3mm]
{\bf AI Declaration} {{\ 
During the preparation of this manuscript, the authors used AI solely to improve the language, grammar, and readability of the text and to assist with LaTeX typesetting. The authors reviewed and edited all AI-assisted output and take full responsibility for the final manuscript.}}

\end{document}